%% file: KS_Stabiliz3_arx.tex
\documentclass[reqno]{amsart}
\usepackage{amscd,amsthm,amsmath,amssymb,mathtools}
\usepackage{verbatim}
\usepackage{xcolor}
\usepackage{url}
\usepackage{bbm}
\usepackage{enumerate,enumitem}
\usepackage{graphicx}
\usepackage{epstopdf,epsfig,subfigure}
\usepackage{curve2e}

\usepackage{todonotes}
\newcommand{\todoautc}[3][]{%
    \ifthenelse{\equal{#1}{}}{\todo[size=\scriptsize]{{\bf#2} #3}{}}{\todo[color=#1,size=\scriptsize]{{\bf#2} #3}{}}%
}

\theoremstyle{plain}

\newtheorem{theorem}{Theorem}[section]

\newtheorem{lemma}[theorem]{Lemma}

\newtheorem{assumption}[theorem]{Assumption}

\theoremstyle{definition}

\newtheorem{remark}[theorem]{Remark}

\numberwithin{equation}{section}

\usepackage[
    bookmarks=true,         
    unicode=false,          
    pdftoolbar=true,        
    pdfmenubar=true,        
    pdffitwindow=false,     
    pdfstartview={FitH},    
    pdftitle={Stabilization of Kuramoto--Sivashinsky equations},    
    pdfauthor={Author},     
    pdfsubject={Subject},   
    pdfcreator={Creator},   
    pdfproducer={Producer}, 
    pdfkeywords={keyword1, key2, key3}, 
    pdfnewwindow=true,      
    colorlinks=true,       
    linkcolor=red,          
    citecolor=blue,        
    filecolor=magenta,      
    urlcolor=cyan,           
hypertexnames=false
]{hyperref}

\usepackage{cite}

\input{Mathcommands}

\begin{document}
\title{Feedback semiglobal stabilization to trajectories for the Kuramoto--Sivashinsky equation}
\author{S\'ergio S.~Rodrigues$^{\tt1,*}$}
\author{Dagmawi A. Seifu$^{\tt1}$}
\thanks{
\vspace{-1em}\newline\noindent
\newline\noindent
$^{\tt1}$ Johann Radon Institute for Computational and Applied Mathematics,
  \"OAW, Altenbergerstr. 69, 4040 Linz, Austria.
  \newline\noindent
$^{*}$ Corresponding author.
\newline\noindent
{\sc Emails}:
{\small\tt sergio.rodrigues@ricam.oeaw.ac.at,\quad
dagmawi.seifu@ricam.oeaw.ac.at}%
}

\begin{abstract}
It is shown that an oblique projection based feedback  control 
is able to stabilize the state of the Kuramoto--Sivashinsky equation evolving in rectangular
domains to a given time-dependent trajectory. The number of actuators is finite and consists of a finite number of indicator functions supported in small subdomains.
Simulations are presented, in the one-dimensional case, showing the stabilizing
performance of the feedback control.
\end{abstract}

%

\maketitle

%

\pagestyle{myheadings} \thispagestyle{plain} \markboth{\sc S. S.
Rodrigues and D. A. Seifu}{\sc Stabilization of Kuramoto--Sivashinsky equations}

\section{Introduction}

We investigate the stabilizability to trajectories for the Kuramoto--Sivashinsky equation
\begin{subequations}\label{sys-haty-KS-Intro}
  \begin{align}
 &\dot{\widehat y}+ \nu_2\Delta^2\widehat y+ \nu_1\Delta\widehat  y+\tfrac{1}{2}\nu_0\norm{\nabla\widehat  y}{\bbR^d}^2=f,\\
 &\widehat y(0,\Bigcdot)=\widehat y_0,\qquad \clG \widehat y\rest{\p\Omega}=g,
\end{align}
\end{subequations}
where~$\Omega\subset\bbR^d$ is a bounded rectangular domain, $d\in\{1,2,3\}$, and the state~$\widehat y=\widehat y(t,x)\in\bbR$
evolves in a suitable Hilbert space~$V\subset  L^2(\Omega)$,  
that is,~$\widehat y(t,\Bigcdot)\in V$ for all~$t\ge0$.
The functions~$f=f(t,x)$ and~$g=g(t,x)$ are  external forces,
and the operator~$ \clG$ imposes the boundary conditions on the boundary~${\p\Omega}$ of~$\Omega$. 
Furthermore,~$\nu_2$, $\nu_1$, and~$\nu_0$ are strictly positive constants. %
System~\eqref{sys-haty-KS-Intro} is a model for flame propagation applications.

Suppose that the evolution of~$\widehat y$ has a desirable behavior, for example, remains bounded and nonchaotic. It turns out that the solution~$\widetilde y$ issued from a different initial condition
\begin{subequations}\label{sys-tildey-KS-Intro}
 \begin{align}
 &\dot{\widetilde y}+   \nu_2\Delta^2\widetilde  y+   \nu_1\Delta\widetilde  y+\tfrac{1}{2}\nu_0\norm{ \nabla\widetilde   y}{\bbR^d}^2=f,\\
  &\widetilde y(0,\Bigcdot)= \widetilde y_0,\qquad \clG  \widetilde y\rest{\p\Omega}=g,
\end{align}
\end{subequations}
can present a different behavior, in particular can be unstable or evolve to a chaotic behavior.
Suppose also that we want to avoid such undesired behavior, and that we can act on the system through a finite number of
actuators~$\Phi_j\in L^2(\Omega)$, $1\le j\le M_\sigma$, where~$M_\sigma\in\bbN_+$ is a positive integer.
Then our goal is to find an appropriate way to tune/activate such actuators so that~$\widetilde y(t,\Bigcdot)$ will converge to~$\widehat y(t,\Bigcdot)$, as time increases.
That is, we seek a  control function~$u=u(t)$, 
$u\in L^2((0,+\infty),\bbR^{M_\sigma})$
so that the solution of
\begin{subequations}\label{sys-tildey-KS-Intro-Bu}
\begin{align}
 &\dot{\widetilde y}+   \nu_2\Delta^2 \widetilde y+  \nu_1 \Delta \widetilde y
 +\tfrac{1}{2}\nu_0\norm{ \nabla  \widetilde y}{\bbR^d}^2=f+\sum_{j=1}^{M_\sigma}u_j\Phi_j,\\
  &\widetilde y(0,\Bigcdot)= \widetilde y_0,\qquad \clG  \widetilde y\rest{\p\Omega}=g,
\end{align} 
\end{subequations}
satisfies
\begin{align}%
 &\norm{\widetilde y(t,\Bigcdot)-\widehat y(t,\Bigcdot)}{V}\le C\ex^{-\mu t}\norm{\widetilde y_0-\widehat y_0}{V},
 \qquad\mbox{for all}\qquad t\ge0.\label{goal-Intro-exp}
\end{align}

Note that the difference~$y=\widetilde y-\widehat y$ satisfies the dynamical system
\begin{subequations}\label{sys-y-KS-Intro-Bu0}
\begin{align}
 &\dot{y}+   \nu_2\Delta^2  y+   \nu_1\Delta y
 +\tfrac{1}{2}\nu_0\left(\norm{ \nabla  \widetilde y}{\bbR^d}^2-\norm{ \nabla  \widehat y}{\bbR^d}^2\right)=\sum_{j=1}^{M_\sigma}u_j\Phi_j,\\
  &y(0,\Bigcdot)= y_0,\qquad \clG   y\rest{\p\Omega}=0,
\end{align} 
with~$y_0\coloneqq\widetilde y_0-\widehat y_0$.
\end{subequations}
Observe also that, by straightforward computations,  
\begin{align}
 \tfrac{1}{2}\left(\norm{ \nabla  \widetilde y}{\bbR^d}^2-\norm{ \nabla  \widehat y}{\bbR^d}^2\right)
 &=\tfrac{1}{2}\left(\norm{ \nabla  \widehat y + \nabla y}{\bbR^d}^2-\norm{ \nabla  \widehat y}{\bbR^d}^2\right)
=\tfrac{1}{2}\left(\norm{ \nabla  y}{\bbR^d}^2+2( \nabla  \widehat y, \nabla y)_{\bbR^d}\right),\notag
\end{align} 
where~$(\Bigcdot,\Bigcdot)_{\bbR^d}$ denotes the usual scalar product in~$\bbR^d$. Therefore, we arrive at
\begin{subequations}\label{sys-y-KS-Intro-Bu}
\begin{align}
 &\dot{y}+   \nu_2\Delta^2  y+   \nu_1\Delta y
 +\nu_0( \nabla  \widehat y, \nabla y)_{\bbR^d}+\tfrac{1}{2}\nu_0\norm{ \nabla  y}{\bbR^d}^2
  =\sum_{j=1}^{M_\sigma}u_j\Phi_j,\\
  &y(0,\Bigcdot)= y_0,\qquad \clG   y\rest{\p\Omega}=0.
\end{align} 
\end{subequations}

Equation~\eqref{goal-Intro-exp} can now be rewritten as
\begin{align}%
 &\norm{y(t,\Bigcdot)}{V}\le C\ex^{-\mu t}\norm{ y_0}{V},
 \qquad\mbox{for all}\qquad t\ge0.\label{goal-diff-exp}
\end{align}
Hence, our task is to find~$u$ such that the solution of the nonautonomous
system~\eqref{sys-y-KS-Intro-Bu} satisfies~\eqref{goal-diff-exp}. Note that~\eqref{sys-y-KS-Intro-Bu} is nonautonomous
when $\widehat y=\widehat y(t,x)$ is time-dependent, this is the case when at least one
of the external forces, $f$ and~$g$, in~\eqref{sys-haty-KS-Intro} depends on time. 

\begin{remark}
In the Kuramoto--Sivashinsky model~\eqref{sys-haty-KS-Intro} it is important that~$\nu_2>0$ is strictly positive. Though in~\eqref{sys-haty-KS-Intro} the parameters~$\nu_1$ and~$\nu_0$ are taken also positive, at this point we must say that the following results also hold for the academic cases~$\nu_1\le0$ and/or~$\nu_0\le0$. However, also from the stabilization point of view, the  more interesting cases correspond to $\nu_1>0$, which induces an anti-diffusion effect in the dynamics, thus inducing a source of instability, which we (may) need to counteract with a stabilizing control. Instead, the case ~$\nu_1<0$ would induce a diffusion stabilizing effect in the dynamics, thus making the stabilization task less challenging.
\end{remark}

As we have said, system~\eqref{sys-haty-KS-Intro} models flame propagation. We also investigate the Kuramoto--Sivashinsky model for fluid flow, which reads 
\begin{subequations}\label{sys-haty-KS-Intro-fluid}
  \begin{align}
 &\dot{\widehat \bfy}+   \nu_1\Delta^2\widehat \bfy+   \nu_2\Delta\widehat  \bfy+\nu_0\langle\widehat  \bfy\cdot\nabla\rangle\widehat  \bfy=\bff,\\
 &\widehat \bfy(0,\Bigcdot)=\widehat \bfy_0,\qquad\bfG \widehat \bfy\rest{\p\Omega}=\bfg,
\end{align}
where now~$\widehat \bfy$, $\widehat \bfy_0$, $\bff$, and~$\widehat \bfg$ are vector functions.
\end{subequations}
The analogs of~\eqref{sys-tildey-KS-Intro-Bu-fluid} and~\eqref{sys-y-KS-Intro-Bu} read
\begin{subequations}\label{sys-tildey-KS-Intro-Bu-fluid}
\begin{align}
 &\dot{\widetilde \bfy}+   \nu_1\Delta^2 \widetilde \bfy+   \nu_2\Delta \widetilde \bfy
 +\nu_0\langle\widetilde  \bfy\cdot\nabla\rangle\widetilde  \bfy=\bff+\sum_{j=1}^{M_\sigma}u_j\mathbf\Phi_j,\\
  &\widetilde \bfy(0,\Bigcdot)= \widetilde \bfy_0,\qquad\bfG  \widetilde \bfy\rest{\p\Omega}=\bfg,
\end{align} 
\end{subequations}
and
\begin{subequations}\label{sys-y-KS-Intro-Bu-fluid}
\begin{align}
 &\dot{\bfy}+   \nu_2\Delta^2  \bfy+   \nu_1\Delta \bfy
 +\nu_0\langle\widehat  \bfy\cdot\nabla\rangle  \bfy+\nu_0\langle \bfy\cdot\nabla\rangle\widehat  \bfy+\nu_0\langle \bfy\cdot\nabla\rangle \bfy
  =\sum_{j=1}^{M_\sigma}u_j\mathbf\Phi_j,\\
  &\bfy(0,\Bigcdot)= \bfy_0,\qquad\clG   \bfy\rest{\p\Omega}=0.
\end{align} 
\end{subequations}
Again, our task is to find~$u$ such that the solution of the nonautonomous
system~\eqref{sys-y-KS-Intro-Bu-fluid} satisfies the analog of~\eqref{goal-diff-exp}, for a suitable Hilbert space~$\bfV\subset L^2(\Omega)^d$,
\begin{align}%
 &\norm{\bfy(t,\Bigcdot)}{\bfV}\le C\ex^{-\mu t}\norm{\bfy_0}{\bfV},
 \qquad\mbox{for all}\qquad t\ge0.\label{goal-diff-exp-fluid}
\end{align}

\begin{remark}
We can see that, if~$\widehat y$ solves~\eqref{sys-haty-KS-Intro} (and is regular enough) then its gradient~$\widehat\bfy=\nabla\widehat y$ solves a system as~\eqref{sys-haty-KS-Intro-fluid}. However, we will use (vectors of) indicator functions of small subdomains as actuators to stabilize both models, hence the stabilizability result does not trivially transfer from one model to the other. 
\end{remark}

\subsection{The main results}\label{sS:intro-main-res}
We will show that,  for arbitrary given~$R>0$ and~$\lambda>0$, we can find a set of~$M_\sigma=M_\sigma(R)$~actuators
such that system~\eqref{sys-y-KS-Intro-Bu}
is exponentially stable with the feedback control function~$u(t)=u(t,y(t))\in\bbR^{M_\sigma}$, defined by
\begin{align}
 \sum_{j=1}^{M_\sigma}u_j(t)\Phi_j&\coloneqq K_{M}^{\lambda}(t,y(t))
\coloneqq P_{\clU_{M}}^{\clE_{M}^\perp}
 \left(   \nu_1\Delta y(t)
 +\tfrac{1}{2}\nu_0\norm{\nabla  y(t)}{\bbR^d}^2+\nu_0(\nabla  \widehat y(t),\nabla y(t))_{\bbR^d}-\lambda y(t)\right)\notag\\
&= P_{\clU_{M}}^{\clE_{M}^\perp}
 \left(   \nu_1\Delta y(t)
 +\tfrac{1}{2}\nu_0\norm{\nabla\widetilde  y(t)}{\bbR^d}^2-\tfrac{1}{2}\nu_0\norm{\nabla\widehat  y(t)}{\bbR^d}^2-\lambda y(t)\right)\label{FeedKy-N}
\end{align}
provided the initial condition is in the ball~$\{v\in V\mid\norm{v}{V}<R\}$.
Here~$P_{\clU_{M}}^{\clE_{M}^\perp}$  is the oblique projection in~$H\coloneqq L^2(\Omega)$ onto~$\clU_{M}$
along~$\clE_{M}^\perp$, where:
\begin{itemize}
 \item $\sigma\colon\bbN_+\to  \bbN_+$ is an appropriate strictly increasing sequence,
 \item $\clU_{M}\subset H$ is the linear span of our ${M_\sigma}\coloneqq\sigma(M)$ actuators,
 \item $\clE_{M}$ is the linear span of suitable~${M_\sigma}$ eigenfunctions 
 of the Laplacian~$\Delta$ satisfying the boundary conditions, and~$\clE_{M}^\perp$
 is the orthogonal space to~$\clE_{M}$, in~$H$.
\end{itemize}
More details on the choice of the spaces~$\clU_{M}$ and~$\clE_{M}$ will be given later on.

We shall use tools from~\cite{Rod20-eect}, devoted to second-order (heat-like) parabolic equations,
to derive our stabilizability result for the fourth-order Kuramoto--Sivashinsky parabolic equation.
The stabilizability result in~\cite{Rod20-eect} holds under an assumption on the uniform boundedness of the operator
norm~$\norm{P_{\clU_{M}}^{\clE_{M}^\perp}}{\clL(H)}\le C_P$ of the oblique projection, with~$C_P$ independent of~$M$; see~\cite[Cor.~2.9]{Rod20-eect}. The 
satisfiability of this assumption has been proven 
in~\cite{KunRod19-cocv,RodSturm20}
for both Dirichlet and Neumann
boundary conditions, where~$\clE_{M_\sigma}^\perp$
is spanned by eigenfunctions of the Laplacian~$\Delta$ under the respective boundary  conditions.

As in many works on Kuramoto--Sivashinsky equation, here we also consider the case of periodic boundary conditions. 
We will show that the oblique projection~$\norm{P_{\clU_{M}}^{\clE_{M}^\perp}}{\clL(H)}$ is well defined and the mentioned uniform boundedness is also satisfiable for periodic
boundary conditions. As we will see the proofs of these facts (cf. Lemmas~\ref{LA:DS} amd~\ref{LA:bddratP=suffalpha}) use elementary tools but are nontrivial.

Our stabilizability result for the Kuramoto--Sivashinsky equations will hold for both periodic, bi-Dirichlet and bi-Neumann boundary conditions. 

At this point we must clarify our terminology for bi-Dirichlet and bi-Neumann boundary conditions concerning fourth-order parabolic equations. By bi-Dirichlet boundary conditions we mean~$u\rest{\p\Omega}=0=\Delta u\rest{\p\Omega}$ and by bi-Neumann boundary conditions we mean~$\tfrac{\p u}{\p\bfn}\rest{\p\Omega}=0=\tfrac{\p \Delta u}{\p\bfn}\rest{\p\Omega}$. 
In the literature, we can find that there are some combinations of Dirichlet-like and Neumann-like boundary conditions that we can use to deal with fourth order parabolic equations involving the bi-Laplacian $\Delta^2$ as higher order term; see~\cite{GanderLiu17,Begehr06}. For example, the bi-Dirichlet and bi-Neumann boundary conditions we use here do correspond to those in~\cite[Eqs.~(5) and~(6)]{GanderLiu17}.
We refer to~\cite{CerpaGuzMerc17,Gao19} where ``Dirichlet'' and ``Neumann'' boundary conditions have different meanings than  bi-Dirichlet and bi-Neumann. 
Finally, we refer to~\cite{KalogKeavPapa14,AmbroseMazzucato21} where the case of periodic boundary conditions is investigated.

We will also show the analogous stabilizability result  holds for system~\eqref{sys-y-KS-Intro-Bu-fluid} with the analogous feedback control function~$u(t)=u(t,y(t))\in\bbR^{dM_\sigma}$, defined by
\begin{align}
 \sum_{j=1}^{dM_\sigma}u_j(t)\mathbf\Phi_j
 &\coloneqq\bfK_{M}^{\lambda}(t,y(t))\notag\\
 &\coloneqq P_{\underline\clU_{M}}^{\underline\clE_{M}^\perp}
 \left(   \nu_1\Delta \bfy(t)
 +\nu_0\langle\widehat  \bfy\cdot\nabla\rangle  \bfy+\nu_0\langle \bfy\cdot\nabla\rangle\widehat  \bfy+\nu_0\langle \bfy\cdot\nabla\rangle \bfy
  -\lambda \bfy(t)\right)\notag\\
&= P_{\underline\clU_{M}}^{\underline\clE_{M}^\perp}
 \left(   \nu_1\Delta \bfy(t)
 +\nu_0\langle\widetilde  \bfy\cdot\nabla\rangle\widetilde  \bfy-\nu_0\langle \widehat\bfy\cdot\nabla\rangle\widehat  \bfy-\lambda \bfy(t)\right).\label{FeedKy-N-fluid}
\end{align}
Here, the set of actuators spanning~$\underline\clU_{M}$ is given by
\begin{align}
&\mathbf\Phi_i\coloneqq \Phi_i, &&  && &&\mbox{for $d=1$};\notag\\
&\mathbf\Phi_i\coloneqq (\Phi_i,0), &&\mathbf\Phi_{M_\sigma+i}\coloneqq
(0,\Phi_i),\quad  && &&\mbox{for $d=2$};\notag\\
&\mathbf\Phi_i\coloneqq (\Phi_i,0,0), &&\mathbf\Phi_{M_\sigma+i}\coloneqq
(0,\Phi_i,0), &&\mathbf\Phi_{2M_\sigma+i}\coloneqq
(0,0,\Phi_i),\quad   &&\mbox{for $d=3$};\notag
\end{align}
with~$1\le i\le M_\sigma$, where the~$\Phi_i$s are the actuators used in~\eqref{FeedKy-N}.
Analogously, the set of (vector) eigenfunctions spanning~$\underline\clE_{M}$, is given by
\begin{align}
&\bfe_i\coloneqq e_i, &&  && &&\mbox{for $d=1$};\notag\\
&\bfe_i\coloneqq (e_i,0), &&\bfe_{M_\sigma+i}\coloneqq
(0,e_i),\quad  && &&\mbox{for $d=2$};\notag\\
&\bfe_i\coloneqq (e_i,0,0), &&\bfe_{M_\sigma+i}\coloneqq
(0,e_i,0), &&\bfe_{2M_\sigma+i}\coloneqq
(0,0,e_i),\quad   &&\mbox{for $d=3$};\notag
\end{align}
where the~$e_i$s are the scalar eigenfunctions used in~\eqref{FeedKy-N}.

Finally, we present simulations showing the stabilizing performance of the feedback control to/around targeted trajectories (depending on both space and time), for both  flame propagation and fluid flow models. Here we restrict ourselves to the one-dimensional case under periodic boundary conditions. For the spatial discretization we use spectral elements, that is, we compute the solution of standard spectral Galerkin approximations based on ``the'' first ~$N$ eigenfunctions.

\subsection{Main strategy}
To derive the stabilizability result, we write the  dynamics of the dynamics of the difference to the target, in control systems~\eqref{sys-y-KS-Intro-Bu} and~\eqref{sys-y-KS-Intro-Bu-fluid}, as an appropriate abstract semilinear parabolic-like equation, which enable us to show that all the assumptions on the state operators required in~\cite[sect.~3.1]{Rod20-eect} are satisfied. Then, to show that the assumptions on the feedback control required in~\cite[sect.~3.3]{Rod20-eect} are satisfied for bi-Dirichlet and for bi-Neumann boundary conditions, we exploit the fact that for such boundary conditions, the eigenfunctions of the  Kuramoto--Sivashinsky linear operator~$  \nu_2\Delta^2+   \nu_1\Delta$ coincide with those of the Laplacian~$\Delta$ with Dirichlet and Neumann boundary conditions, and use results available on the literature concerning, in particular, a suitable uniform bound for the oblique projection, in~\eqref{FeedKy-N} and~\eqref{FeedKy-N-fluid}, as we increase the number of actuators. Such uniform bound is not available in the literature for the case of periodic
boundary conditions and thus we prove it in section~\ref{sS:feedtriple}; see Lemma~\ref{LA:bddratP=suffalpha}.

\subsection{On the literature}\label{sS:litter}
Most articles devoted to control properties of the Kuramoto-Sivashinsky equation have dealt with the one-dimensional case $d=1$. We refer to \cite{ArmaouChristofides00} and \cite{ChristofidesArmaou00} studying the internal feedback control on a periodic domain and to~\cite{KangFridman18} where stabilization is investigated under ``Dirichlet'' boundary conditions where the actuators cover the entire spatial domain. We also find 
\cite{Lhachemi21-arx} and \cite{LiuKrstic01} where boundary controls are considered to study stabilizability properties. Regarding controllability, in \cite{Guo02} the author proved the boundary null-controllability of a related fourth-order parabolic equation using two boundary inputs. In \cite{Cerpa10,CerpaGuzMerc17} the null-controllability with one input only is studied. The same properties for the nonlinear Kuramoto-Sivashinsky equation are addressed in \cite{CerpaMercado11,Gao19,Takahashi17}. More recently, other control issues as cost of control \cite{CarrenoGuzman16}, time delayed inputs \cite{GuzmanMarxCerpa19,Chentouf21}, hierarchical control \cite{CarrenoSantos19}  have been considered in the literature, among others.

In dimension higher or equals to two, the number of results in the literature is smaller. However, several researchers have been interested in the subject.  We find the recent works \cite{GuerreroKassab19} and \cite{Kassab20} dealing with the internal null-controllability of the Kuramoto--Sivashinsky equation. The main difficulty raised in those papers is the proof of a Carleman estimate which is very useful to get both linear and nonlinear results. The internal null-controllability obtained in these work for particular boundary conditions and it is not trivial to deal with different ones. Our stabilizability result do not use/need the null-controllability property, still, attention must be paid to the type of boundary conditions as well.

The work~\cite{Takahashi17} mentioned above also contains a boundary controllability result for a plane rectangular domain with a control~$u$ ``supported'' on a suitable portion~$\Gamma$ of the boundary~$\p\Omega$, $u(t)\in L^2(\Gamma)$.

The stabilization result in~\cite{KangFridman18} mentioned above has been extended recently to the two-dimensional case $d=2$ in~\cite{KangFridman22} where again the actuators cover the entire spatial domain~$\Omega$, in particular, the total volume covered by the actuators equals the volume of~$\Omega$. In our results the total volume covered by the actuators can be arbitrarily small and fixed  a priori.

Many results in the literature are concerned with the stabilization to the trivial trajectory~$\widehat y(x,t)=0$, thus for the case of vanishing external forces~$f(x,t)=g(x,t)=0$. Here we consider a rather general targeted trajectory~$\widehat y(x,t)$, thus allowing nonvanishing external forces depending on both time and space variables.

\subsection{Contents and general notation}\label{sS:notation}
The rest of the paper is organized as follows. The stabilizability result for the (scalar) flame propagation model, by means of an explicit oblique projection based feedback is shown in section~\ref{S:stabil-fire} and the analogous result for the (vector) fluid flow model is presented in section~\ref{S:stabil-fluid}. In those sections we also give the explicit spatial location of the actuators, which is (for a fixed number of actuators) independent of the data~$(\nu_2,\nu_1,\nu_0,f,g,\clG,\widehat y_0,\widetilde y_0)$. Finally, section~\ref{S:simul} gathers results of numerical simulations showing the stabilizing performance of the explicit oblique projections  feedback control.

\medskip

Concerning the notation, we write~$\bbR$ and~$\bbN$ for the sets of real numbers and nonnegative
integers, respectively, and we define $\bbR_+\coloneqq(0,\,+\infty)$ and~$\overline{\bbR_+}\coloneqq[0,\,+\infty)$, and $\mathbb
N_+\coloneqq\mathbb N\setminus\{0\}$.

For an open interval $I\subseteq\bbR$ and two Banach spaces~$X,\,Y$, we write
$W(I,\,X,\,Y)\coloneqq\{y\in L^2(I,\,X)\mid \dot y\in L^2(I,\,Y)\}$,
where~$\dot y\coloneqq\frac{\ed}{\ed t}y$ is taken in the sense of
distributions. This space is endowed with the natural norm
$|y|_{W(I,\,X,\,Y)}\coloneqq\bigl(|y|_{L^2(I,\,X)}^2+|\dot y|_{L^2(I,\,Y)}^2\bigr)^{1/2}$.

If the inclusions $X\subseteq Z$ and~$Y\subseteq Z$ are continuous, where~$Z$ is a Hausdorff topological space,
then we can define the Banach spaces $X\times Y$, $X\cap Y$, and $X+Y$,
endowed with the norms defined as
$|(a,\,b)|_{X\times Y}:=\bigl(|a|_{X}^2+|b|_{Y}^2\bigr)^{\frac{1}{2}}$,
$|a|_{X\cap Y}:=|(a,\,a)|_{X\times Y}$, and
$|a|_{X+Y}:=\inf_{(a^X,\,a^Y)\in X\times Y}\bigl\{|(a^X,\,a^Y)|_{X\times Y}\mid a=a^X+a^Y\bigr\}$,
respectively.
In case we know that $X\cap Y=\{0\}$, we say that $X+Y$ is a direct sum and we write $X\oplus Y$ instead.

If the inclusion
$X\subseteq Y$ is continuous, we write $X\xhookrightarrow{} Y$. We write
$X\xhookrightarrow{\rm d} Y$, respectively $X\xhookrightarrow{\rm c} Y$, if the inclusion is also dense, respectively compact.

The space of continuous linear mappings from~$X$ into~$Y$ is denoted by~$\clL(X,Y)$. In case~$X=Y$ we 
write~$\clL(X)\coloneqq\clL(X,X)$.
The continuous dual of~$X$ is denoted~$X'\coloneqq\clL(X,\bbR)$.

The space of continuous functions from~$X$ into~$Y$ is denoted by~$\clC(X,Y)$, and its subspace of
increasing functions, defined in~$\overline{\bbR_0}$ and vanishing at~$0$, by:
\begin{equation}\notag
 \clC_{0,\rm i}(\overline{\bbR_0},\bbR) \coloneqq \{\mathfrak n\!\mid \mathfrak n\in \clC(\overline{\bbR_0},\bbR),
 \quad\!\! \mathfrak n(0)=0,\quad\!\!\mbox{and}\quad\!\!
 \mathfrak n(\varkappa_2)\ge\mathfrak n(\varkappa_1)\;\mbox{ if }\; \varkappa_2\ge \varkappa_1\ge0\}.
\end{equation}
Next,  we denote by~$\clC_{\rm b, i}(X, Y)$ the vector subspace 
\begin{equation}\notag
 \clC_{\rm b, i}(X, Y)\coloneqq
 \left\{f\in \clC(X,Y) \mid \exists\mathfrak n\in \clC_{0,\rm i}(\overline{\bbR_0},\bbR)\;\forall x\in X:\;
\norm{f(x)}{Y}\le \mathfrak n (\norm{x}{X})
\right\}.
\end{equation}

Given a subset~$S\subset H$ of a Hilbert space~$H$, with scalar product~$(\Bigcdot,\Bigcdot)_H$, the orthogonal complement of~$S$ is
denoted~$S^\perp\coloneqq\{h\in H\mid (h,s)_H=0\mbox{ for all }s\in S\}$.

For a given open subset~$\Omega\subseteq\bbR^d$, $d\in\bbN_+$, we shall use the usual notation, $L^p(\Omega)$ and~$W^{k,p}(\Omega)=\left\{f\in L^p(\Omega)\mid\frac{\p^i t}{\p x_1^{\bfi_1}...\p x_d^{\bfi_d}}\in L^p(\Omega),\; \bfi\in\bbN^d,\;\bfi_1+\dots+\bfi_d=i\le k\right\}$,  for the Lebesgue and Sobolev spaces, where $p\ge 1$ and $k\in\bbN_+$.

Given a sequence~$(a_j)_{j\in\{1,2,\dots,n\}}$ of real constants, $n\in  \bbN_+$, $a_i\ge0$, we
denote~$\|a\|\coloneqq\max\limits_{1\le j\le n} a_j$. Further, by
$\overline C_{\left[a_1,\dots,a_n\right]}$ we denote a nonnegative function that
increases in each of its nonnegative arguments.

Finally, $C,\,C_i$, $i=0,\,1,\,\dots$, stand for unessential positive constants.

\section{Stabilizability for the flame propagation model}\label{S:stabil-fire}
We show that
system~\eqref{sys-y-KS-Intro-Bu} with the feedback control~\eqref{FeedKy-N} can be written as an abstract parabolic-like evolution equation satisfying the assumptions in~\cite{Rod20-eect}. Then we apply the results in~\cite{Rod20-eect} to conclude the
stabilizability result for~\eqref{sys-y-KS-Intro-Bu}. Hereafter, for bi-Diriclet and bi-Neumann boundary conditions, we assume that $\Omega$ is a rectangular domain $\Omega=\bigtimes\limits_{i=1}^d(0,L_i)$. For periodic boundary conditions, $\Omega$ is the Torus $\Omega=\bbT^d_L\coloneqq\bigtimes\limits_{i=1}^d\tfrac{L_i}{2\pi}\bbT^1\sim\bigtimes\limits_{i=1}^d[0,L_i)$, where $\bbT^1\sim\{(r_1,r_2)\in\bbR^2\mid r_1^2+r_2^2=1\}\sim[0,2\pi)$ is the one dimensional Torus.

\subsection{Kuramoto--Sivashinsky equation in abstract form}\label{sS:KS-abstract}

Let us denote~$H\coloneqq L^2(\Omega)$, which we consider as a pivot space, $H=H'$, as usual.

For the Kuramoto--Sivashinsky equation, for regular enough functions, we define the bi-Dirichlet and bi-Neumann boundary conditions as
\begin{align}
 &\clG_{\rm Dir}y\rest{\p\Omega}=0,\quad\mbox{with}\quad\clG_{\rm Dir}y\coloneqq(y,\Delta y);\notag\\
 &\clG_{\rm Neu}y\rest{\p\Omega}=0,\quad\mbox{with}\quad\clG_{\rm Neu}y\coloneqq(\bfn\cdot \nabla y,\bfn\cdot \nabla\Delta y).\notag
\end{align}
For periodic ``boundary'' conditions, we simply assume that our
equation evolves in the Torus $\bbT^d_L$, so that in fact we have
no boundary conditions at all.

\begin{remark}
 Note that the perimeter of~$\bbT^1$ is equal to~$2\pi$.
Thus considering the evolution on~$\bbT^d_L$ implies that the evolution when seen in~$\bbR^d$ is $L$-periodic, that is, we are assuming the periodicity with
period~$L_n$ in the $n$th coordinate, 
\[
y(x)=y(\hat x), \quad\mbox{if}\quad \norm{x_n-\hat x_n}{\bbR}\in L_n\bbN\quad\mbox{for all}\quad 1\le n\le d,
\]
for any given~$(x,\hat x)\in\bbR^d\times\bbR^d$.
\end{remark}

Let us denote by~$\clA=-\Delta+\Id$ the shifted negative Laplacian under either Dirichlet, Neumann, or periodic boundary conditions, which can be defined as
\[
\langle\clA z, w\rangle_{\clV',\clV}\coloneqq(\nabla z,\nabla w)_{H^d}+(z,w)_{H},
\]
defining an isometry $\clA\colon\clV\to\clV'$, where we recall
\begin{align}
 \clV&\coloneqq\{h\in W^{1,2}(\Omega)\mid h\rest{\p\Omega}=0\},\qquad&&\mbox{for}\quad \clG=\clG_{\rm Dir},\notag\\
 \clV&\coloneqq W^{1,2}(\Omega),\qquad&&\mbox{for}\quad \clG=\clG_{\rm Neu},\notag\\
 \clV&\coloneqq W^{1,2}(\bbT^2_L),\qquad&&\mbox{for}\quad \clG=\clG_{\rm Per},\notag
\end{align}
where the subscripts~`${\rm Dir}$', `${\rm Neu}$', and~`${\rm Per}$' underline the type of boundary conditions under consideration, with corresponding domains~$\rmD(\clA)\coloneqq\{h\in \clV\mid \clA h\in H\}$ given by
\begin{align}
 \rmD(\clA)&=\{h\in W^{2,2}(\Omega)\mid h\rest{\p\Omega}=0\},\qquad&&\mbox{for}\quad \clG=\clG_{\rm Dir},\notag\\
 \rmD(\clA)&=\{h\in W^{2,2}(\Omega)\mid \bfn\cdot\nabla h\rest{\p\Omega}=0\},\qquad&&\mbox{for}\quad \clG=\clG_{\rm Neu},\notag\\
\rmD(\clA)&= W^{2,2}(\bbT^2_L),\qquad&&\mbox{for}\quad \clG=\clG_{\rm Per}.\notag
\end{align}

We endow~$\clV$ with the standard scalar product inherited from~$W^{1,2}(\Omega)$,
\[
(z,w)_\clV\coloneqq\langle\clA z, w\rangle_{\clV',\clV}=(\nabla z,\nabla w)_{H^d}+(z,w)_{H},
\]
with~$H^d\coloneqq\bigtimes_{n=1}^d H=\bigtimes_{n=1}^d L^2(\Omega)$. It follows that~$\clA$ is a linear continuous bijection from~$\rmD(\clA)$ onto~$H$ and has a compact inverse~$\clA^{-1}\colon H\mapsto H$.
By choosing a complete basis of eigenpairs~$(\alpha_i,e_i)\in\bbR\times H$,
\[
\clA e_i=\alpha_ie_i,\qquad i\in\bbN_+,
\]
we can define the fractional powers of the self adjoint operator~$\clA$: given~$s\in\bbR$, we say that~$h\in\rmD(\clA^s)$ if
\[
\clA^sh\coloneqq{\textstyle\sum\limits_{i=1}^{+\infty}}\alpha_i^se_i\in H. 
\]
The space~$\rmD(\clA^s)$ is assumed to be endowed with the scalar product
\[
(z,w)_{\rmD(\clA^s)}\coloneqq(\clA^sz,\clA^sw)_H
\]
and corresponding norm. For~$s\ge0$ we have that~$\rmD(\clA^s)\subseteq H$ and~$\rmD(\clA^s)=\{h\in H\mid \clA^s \in H\}$; for~$s<0$, $\rmD(\clA^s)\supset H$ is the completion of~$H$ in the~$\rmD(\clA^s)$-norm. Observe also that~$\rmD(\clA^{-s})\coloneqq\rmD(\clA^{s})'$, and that
$\clA^s$ is an isometry from~$\rmD(\clA^{r})$ onto~$\rmD(\clA^{r-s})$ for every~$(r,s)\in\bbR^2$.
In particular, we find that
$\rmD(\clA^2)=\{h\in H\mid \clA h\in\rmD(\clA)\}$.
Hence
\begin{align}
 \rmD(\clA^2)&=\{h\in W^{2,2}(\Omega)\mid \clA h\in W^{2,2}(\Omega),\quad (h,\clA h)\rest{\p\Omega}=(0,0)\},&&\mbox{ for}\quad \clG=\clG_{\rm Dir},\notag\\
 \rmD(\clA^2)&=\{h\in W^{2,2}(\Omega)\mid \clA h\in W^{2,2}(\Omega),\quad \bfn\cdot \nabla (h,\clA h)\rest{\p\Omega}=(0,0)\},&&\mbox{ for}\quad \clG=\clG_{\rm Neu},\notag\\
\rmD(\clA^2)&= \{h\in W^{2,2}(\bbT^2_L),\mid \clA h\in W^{2,2}(\bbT^2_L)\}&&\mbox{ for}\quad \clG=\clG_{\rm Per}.\notag
\end{align}
Further, since~$\clA h=(-\Delta +\Id)h$ for $h\in\rmD(A)$,
we can write
\begin{align}
 \rmD(\clA^2)&=\{h\in W^{2,2}(\Omega)\mid \Delta h\in W^{2,2}(\Omega),\quad \clG h\rest{\p\Omega}=0\},&&\mbox{ for}\quad \clG\in\{\clG_{\rm Dir},\clG_{\rm Neu}\},\notag\\
\rmD(\clA^2)&= \{h\in W^{2,2}(\bbT^2_L),\mid \Delta h\in W^{2,2}(\bbT^2_L)\}&&\mbox{ for}\quad \clG=\clG_{\rm Per}.\notag
\end{align}

Observe that for~$(z,w)\in \rmD(\clA)\times \rmD(\clA)$
\begin{align*}
\langle \clA^2 z,w\rangle_{\rmD(\clA)',\rmD(\clA)}&=(\clA z,\clA w)_H= ((-\Delta+\Id)z,(-\Delta+\Id)w)_H\notag\\
&=(\Delta z,\Delta w)_H+2(\nabla z,\nabla w)_{H^d}+(z,w)_H.\notag
\end{align*}

For convenience, we denote~$V\coloneqq\rmD(\clA)$ and~$A=\nu_2\clA^2$, 
\begin{equation}\label{scpV}
(z,w)_V\coloneqq \nu_2(\Delta z,\Delta w)
+2\nu_2(\nabla z,\nabla w) +\nu_2(z,w), 
\end{equation}
where~$\nu_2>0$ is as in system~\eqref{sys-haty-KS-Intro}. 
The domain of the operator~$A$,
\begin{subequations}\label{absOper}
\begin{equation}\label{absOper-AinV}
A\in\clL(V,V'),\qquad \langle A z,w\rangle_{V',V}\coloneqq (z,w)_V,
\end{equation}
is given by~$\rmD(A)=\rmD(\clA^2)$. Observe also that
the following relations
\begin{align}
(Az,w)_H&=\langle A z,w\rangle_{V',V}=\nu_2(\Delta z,\Delta w)_H
+2\nu_2(\nabla z,\nabla w) +\nu_2(z,w)\notag\\
&=-\nu_2(\nabla\Delta z,\nabla w)_H
-2\nu_2(\Delta z, w)_H +\nu_2(z,w)_H=(\nu_2\Delta^2 z-2\nu_2\Delta z+z, w)_H,\notag
\end{align}
hold for $(z,w)\in\rmD(A)\times V$ and for both periodic, bi-Dirichlet, and bi-Neumann boundary conditions.
From~$V\xhookrightarrow{\rmd}H$
and from~$\nu_2\Delta^2 z-2\nu_2\Delta z+\nu_2z\in H$,  we can conclude that~$(Az,w)_H=(\nu_2\Delta^2 z-2\nu_2\Delta z+\nu_2z, w)_H$ for all~$w\in H$, which gives us
\begin{equation}\notag%
 Ay= \nu_2 \Delta^2y-2\nu_2\Delta y+\nu_2y\quad\mbox{for}\quad y\in\rmD(A),
\end{equation}

We also set the state operators
\begin{align}
  A_{\rm rc}y&\coloneqq (\nu_1 +2\nu_2) \Delta y-\nu_2y+\nu_0(\nabla  \widehat y,\nabla y)_{\bbR^d},\\
N(y)&\coloneqq\tfrac{1}{2}\nu_0\norm{\nabla  y}{\bbR^d}^2.
\end{align}

Finally, by setting the feedback operator~\eqref{FeedKy-N}, which now reads
\begin{align}
 y\mapsto K_{M}^{\lambda}(t,y)&= P_{\clU_{M}}^{\clE_{M}^\perp}
 \bigl(\nu_1\Delta y +\nu_0(\nabla  \widehat y,\nabla y)_{\bbR^d}+N(y)-\lambda y\bigr),\notag\\
 &= P_{\clU_{M}}^{\clE_{M}^\perp}
 \bigl(Ay+A_{\rm rc}y+N(y)-(\nu_2\Delta^2+\lambda\Id) y\bigr)
 \label{FeedKy-N-abs}
\end{align}
\end{subequations}
we can write~\eqref{sys-y-KS-Intro-Bu} as the evolutionary parabolic-like equation
\begin{align}\label{sys-y-KS-Bu}
 &\dot{y}+A  y+ A_{\rm rc} y
 +N(y)
  =K_{M}^{\lambda}(t,y),\qquad
  y(0)= y_0.
\end{align}

\subsection{Actuators and auxiliary eigenfunctions}\label{sS:acteigf}
Recall our spatial rectangular domain: $\Omega=\bigtimes_{n=1}^d(0,L_n)$ for bi-Dirichlet and bi-Neumann boundary conditions, and $\Omega=\bbT^d_L\sim\bigtimes_{n=1}^d[0,L_n)$ for periodic boundary conditions.

As actuators we take indicator functions of small subdomains as in~\cite[Sect.~2.2]{Rod20-eect}; see also~\cite[Sects.~4.5 and 4.8.1]{KunRod19-cocv} and~\cite[Sect.~2.2]{RodSturm20}. Namely, for a given $r\in(0,1)$ and for a given integer $M\ge 1$ we consider, for each~$n\in\{1,...,d\}$ the intervals
 \begin{equation}\label{omeM1D}
     \omega_{n,j}^{M} = (c_{n,j}^{M} - \tfrac{rL_n}{2M}, c_{n,j}^{M} + \tfrac{rL_n}{2M})\subset(0,L_n) \qquad \mbox{for}\quad \quad j\in\{1,\ldots,M\},
 \end{equation}
each of length $\frac{rL_n}{M}$, with centers $c_{n,1}^{M},\ldots,c_{n,M}^{M} $  given by
 \begin{equation}\label{cM}
     c_{n,j}^{M}\coloneqq \tfrac{(2j-1)L_n}{2M}, \qquad \mbox{for}\quad  j\in\{1,\ldots,M\}.
 \end{equation}
Then, we consider the subdomains~$\omega_{\bfj}^{M}\subset\Omega$ given by the following Cartesian products
 \begin{equation}\notag
  \omega_{\bfj}^{M}\coloneqq  \bigtimes\limits_{n=1}^d \omega_{n,\bfj_n}^{M}  \qquad\mbox{for}\quad \bfj\in\{1,\ldots,M\}^d.
 \end{equation}
As actuators we take the indicator functions of these subdomains~$\Phi_{M,\bfj}\coloneqq\indf_{\omega_{\bfj}^{M}}$. That is, as our set~$U_M$ of actuators (cf.~\eqref{FeedKy-N}) we take 
\begin{equation}\label{UM-explicit}
U_{M}\coloneqq \{\indf_{\omega_{\bfj}^{M}}\mid \bfj\in\{1,2,\ldots,M\}^d\},\quad
\clU_M\coloneqq\linspan U_M,\quad \dim\clU_M=M^d.
\end{equation}

Finally, again following~\cite[Sect.~2.2]{Rod20-eect}, as the set of auxiliary eigenfunctions  we take Cartesian products as follows (see also~\cite[Sects.~4.5 and 4.8.1]{KunRod19-cocv}).
For each~$n\in\{1,...,d\}$ we take the set of the first  eigenfunctions of the Laplacian operator in the interval~$(0,L_n)$, under the corresponding boundary conditions, where the complete system of eigenfunctions is given by the sequence~$(e_{n,i})_{i\in\bbN_+}$ with
\begin{subequations}\label{eigf}
 \begin{align}
& e_{n,i}(x_n)\coloneqq\sin(\tfrac{i\pi x_n}{L_n}),&&\mbox{ for }\clG=\clG_{\rm Dir},\\
& e_{n,i}(x_n)\coloneqq\cos(\tfrac{(i-1)\pi x_n}{L_n}),&&\mbox{ for }\clG=\clG_{\rm Neu},\\
& e_{n,i}(x_n)\coloneqq \begin{cases}
\cos(\tfrac{(i-1)\pi x_n}{L_n}),&\mbox{ if $i$ is odd},\\
\sin(\tfrac{i\pi  x_n}{L_n}),&\mbox{ if $i$ is even},
\end{cases}
&&\mbox{ for }\clG=\clG_{\rm per},
 \end{align}
\end{subequations}
Note that these eigenfunctions of the Laplacian operator are also eigenfunctions for the bi-Laplacian (under the corresponding ``bi'' type of boundary conditions), and form a complete orthogonal basis in~$L^2((0,L_n),\bbR)$. 
Then, we can obtain a complete orthogonal basis of eigenfunctions of the bi-Laplacian in~$\Omega$, by  taking
 \begin{equation}\notag
  e_{\bfj}^\times\coloneqq  \bigtimes\limits_{n=1}^d e_{n,\bfj_n},  \qquad\mbox{with}\quad  \bfj=(\bfj_1,\dots,\bfj_d)\in\bigtimes\limits_{n=1}^d\bbN_+.
 \end{equation}

Finally, as  auxiliary eigenfunctions of the bi-Laplacian in~$\Omega$ we take
 \begin{equation}\notag
  e_{\bfj}^{M}\coloneqq  \bigtimes\limits_{n=1}^d e_{n,\bfj_n} , \qquad\mbox{with}\quad \bfj\in\bigtimes\limits_{n=1}^d\{1,\ldots,M\}.
 \end{equation}
That is, for the set of auxiliary eigenfunctions (cf.~\eqref{FeedKy-N})  we set
\begin{equation}\label{EM-explicit}
E_{M}\coloneqq \{e_{\bfj}^{M}\mid \bfj\in\{1,2,\ldots,M\}^d\},\quad
\clE_M\coloneqq\linspan E_M,\quad \dim\clE_M=M^d.
\end{equation}

\subsection{The main result}

Concerning the flame propagation model, the main stabilizability result of this manuscript is the following.
\begin{theorem}\label{T:Main}
System~\eqref{sys-y-KS-Bu} with the state and feedback operators in~\eqref{absOper} is semiglobally exponentially stable. More precisely, for arbitrary given~$\mu>0$ and~$R>0$, there exists~$M\in\bbN_+$, $0<\mu\le\lambda$, and~$C\ge1$ such that the solution satisfies
\[
\norm{y(t)}{V} \le C\rme^{-\mu t} \norm{y_0}{V},\quad\mbox{for all}\quad t\ge0 \quad\mbox{and all}\quad y_0\in V\mbox{ with } \norm{y_0}{V}<R.
\]
\end{theorem}

\medskip\noindent

\subsection{On the state operators}
We will need the following boundedness assumption on the targeted trajectory.
\begin{assumption}\label{A:haty}
We have that~$\norm{\nabla  \widehat y}{L^\infty(\bbR_+,L^3(\Omega)^d)}\le C_{\widehat y}<+\infty$.
\end{assumption}

We shall show that the state operators~$A_{\rm rc}$, and~$\clN$,  satisfy the assumptions
as in~\cite[Sect.~3.1]{Rod20-eect}. For that, it is enough to prove the following Lemmas~\ref{LA:A0sp}--\ref{LA:NN} corresponding to the assumptions on the state operators we find in~\cite[Assumps.~3.1--3.4]{Rod20-eect}.

\begin{lemma}\label{LA:A0sp}
 $A\in\clL(V,V')$ is an isomorphism from~$V$ onto~$V'$, $A$ is symmetric, and $(y,z)\mapsto\langle Ay,z\rangle_{V',V}$
 is a complete scalar product on~$V.$
\end{lemma}

\begin{lemma}\label{LA:A0cdc}
The inclusion $V\subseteq H$ is continuous, dense, and compact. 
\end{lemma}

\begin{lemma}\label{LA:A1}
For  all~$t>0$ we have~$A_{\rm rc}(t)\in\clL(V,H)$, and there is a nonnegative constant~$C_{\rm rc}$
such that, $\norm{A_{\rm rc}}{L^\infty(\bbR_0,\clL(V,H))}\le C_{\rm rc}.$
\end{lemma}

\begin{lemma}\label{LA:NN}
 We have~$N(t,\Bigcdot)\in\clC_{\rm b,i}(\rmD(A),H)$ 
and
there exist constants $C_\clN\ge 0$, $n\in  \bbN_+$, 
$\zeta_{1j}\ge0$, $\zeta_{2j}\ge0$,
 $\delta_{1j}\ge 0$, ~$\delta_{2j}\ge 0$, 
 with~$j\in\{1,2,\dots,n\}$, such that
 for  all~$t>0$ and
 all~$(y_1,y_2)\in \rmD(A)\times \rmD(A)$, we have
\begin{align}
&\norm{N(t,y_1)-N(t,y_2)}{H}\le C_N\textstyle\sum\limits_{j=1}^{n}
  \left( \norm{y_1}{V}^{\zeta_{1j}}\norm{y_1}{\rmD(A)}^{\zeta_{2j}}+\norm{y_2}{V}^{\zeta_{1j}}\norm{y_2}{\rmD(A)}^{\zeta_{2j}}\right)
   \norm{d}{V}^{\delta_{1j}}\norm{d}{\rmD(A)}^{\delta_{2j}},
\notag
\end{align}
with~$d\coloneqq y_1-y_2$, $\zeta_{2j}+\delta_{2j}<1$ and~$\delta_{1j}+\delta_{2j}=1$.
\end{lemma}

Now we give the proofs of the Lemmas.

\begin{proof}[Proof of Lemma~\ref{LA:A0sp}]
From the discussion in section~\ref{sS:KS-abstract} we know that $A=\clA^2$ is an isometry from~$\rmD(\clA)=V$ onto~$\rmD(\clA^{-1})=\rmD(\clA)'$. It is also not hard to see that~$\langle Ay,z\rangle_{V',V}$ as in~\eqref{absOper-AinV} defines a scalar product in~$V$, whose associated norm is equivalent to the usual norm in~$W^{2,2}(\Omega)$.
\end{proof}

\begin{proof}[Proof of Lemma~\ref{LA:A0cdc}]
The statement follows from the inclusions~$\rmD(\clA)\xhookrightarrow{\rm d,c} \clV\xhookrightarrow{\rm d,c} H$ and from the equivalence of the norms of~$V$ and~$\rmD(\clA)$.
\end{proof}

\begin{proof}[Proof of Lemma~\ref{LA:A1}]
We observe that
\[
\norm{A_{\rm rc}(t)h}{H}\le (2\nu_2+\nu_1+\nu_0) \left(\norm{\Delta y}{H}+\norm{y}{H}+\norm{(\nabla  \widehat y,\nabla y)_{\bbR^d}}{H}\right)\le C_0\left(\norm{y}{\rmD(\clA)}+C_{\widehat y}\norm{\nabla y}{(L^6)^d}\right),
\]
with $C_{\widehat y}$ as in Assumption~\ref{A:haty}. We finish the proof by recalling that~$V\xhookrightarrow{}\rmD(\clA)$ and~$W^{1,2}\xhookrightarrow{}L^6$. The latter gives us~$\norm{\nabla y}{(L^6)^d}\le C_1\norm{y}{W^{2,2}}\le C_2\norm{y}{\rmD(\clA)}$.
\end{proof}

\begin{proof}[Proof of Lemma~\ref{LA:NN}]
With~$z\coloneqq y_1-y_2$, we write
\begin{align}
N(y_1)-N(y_2)=\tfrac12\nu_0(\nabla  y_1,\nabla z)_{\bbR^d}+\tfrac12\nu_0(\nabla z,\nabla  y_2)_{\bbR^d}\notag
\end{align}
and obtain, recalling that $1\le d\le 3$,
\begin{align}
\norm{N(y_1)-N(y_2)}{H}
&\le \tfrac12\nu_0\left(\norm{\norm{\nabla y_1}{\bbR^d}}{L^4(\Omega)}+\norm{\norm{\nabla y_2}{\bbR^d}}{L^4(\Omega)}\right)\norm{\norm{\nabla z}{\bbR^d}}{L^4(\Omega)}\notag\\
&\le C_1\left(\norm{y_1}{V}+\norm{y_2}{V}\right)\norm{z}{V}\notag
\end{align}
Thus, the Lemma follows with~$n=1$, $C_N=C_1$, and~$(\zeta_{1j},\zeta_{2j},\delta_{1j},\delta_{2j})=(1,0,1,0)$.
\end{proof}

\subsection{On the feedback triple}\label{sS:feedtriple}
Recalling~\eqref{FeedKy-N-abs}, we define the mapping
\[
F_{M}\colon \clE_{M}\to \clE_{M},\quad\mbox{with}\quad F_{M}(q)\coloneqq \nu_2\Delta^2 q+\lambda q.
 \]
 We
show that the triple~$(F_{M},\clU_{M},\clE_{M})$  
appearing in the feedback operator satisfies the assumptions
as in~\cite[Sect.~3.3]{Rod20-eect}, for suitably spaces~$\clU_{M}$ and~$\clE_{M}$.
For that, it is enough to prove the following Lemmas~\ref{LA:FFM}--\ref{LA:bddratP=suffalpha} corresponding to the assumptions on the feedback operator we find in~\cite[Assumps.~3.6--3.8, and Cor.~3.9]{Rod20-eect}.

Recall the constants in Lemma~\ref{LA:NN}, and consider the system
\begin{equation}\label{sys-q}
 \dot q=-F_{M}(q),\qquad q(0)=q_0\in\clE_M,
\end{equation}
and the eigenvalues
\begin{subequations}\label{VDAalphaM}
 \begin{align}
\overline \alpha_{M}&\coloneqq \max\{\alpha_i\mid \mbox{there is }\phi\in\clE_{M}
 \mbox{ such that } A\phi=\alpha_i\phi\},\\
 \overline\alpha_{M_{+}}&\coloneqq \min\{\alpha_i\mid \mbox{there is }\phi\in\clE_{M}^\perp
 \mbox{ such that } A\phi=\alpha_i\phi\}. 
\end{align}
\end{subequations}

With the tuple~$\tau_i\coloneqq(\zeta_{i1},\zeta_{i2},\delta_{i1},\delta_{i2})$ as in Lemma~\ref{LA:NN}, for a given map~$\tau_i\mapsto \xi(\tau_i)\in[0,+\infty]$, we denote
$\dnorm{\xi(\zeta_{1},\zeta_{2},\delta_{1},\delta_{2})}{}=\max\limits_{1\le i\le n}\xi(\zeta_{i1},\zeta_{i2},\delta_{i1},\delta_{i2})$ (cf.~Sect.~\ref{sS:notation}).

\begin{lemma}\label{LA:FFM}
 We have that~$F_{M}\in\clC_{\rm b,i}(\clE_{M},\clE_{M})$ and there are
 constants~$C_{q0}\ge0$,~$C_{q1}\ge1$,~$C_{q2}\ge0$, $C_{q3}\ge0$, $\xi\ge1$,~$\lambda>0$,
 $\beta_{1}\ge0$, 
 $\beta_{2}\ge0$, $\eta_{1}\ge0$, 
 $\eta_{2}\ge0$, and~$1<\fkr<\frac{1}{\|\zeta_2+\delta_2\|}$, all independent of~$M$, such that:
\begin{align}
&\norm{F_{M}(\widehat q)}{H}\le C_{q0}\norm{\widehat q}{\rmD(A)}^\xi,
\quad\mbox{for all}\quad \widehat q\in \clE_{M},\label{FFM-F}
\end{align}
\begin{align}
&\beta_1+\beta_2\ge1,\qquad\fkr\|\zeta_1+\delta_1+(\eta_1+\eta_2)(\zeta_2+\delta_2)\|\ge1,\label{FFM-betaeta1}\\%
&\left(\dnorm{\tfrac{\zeta_1+\delta_1}{1-\zeta_2-\delta_2}}{}-1\right)\beta_2<1-\dnorm{\tfrac{\zeta_2}{1-\delta_2}}{}
,\quad\!
\left(\dnorm{\tfrac{\zeta_1+\delta_1}{1-\zeta_2-\delta_2}}{}\!-1\!\right)\!
\dnorm{\zeta_{2}+\delta_{2}}{}\fkr\eta_2<1-\dnorm{\tfrac{\zeta_2}{1-\delta_2}}{},\label{FFM-betaeta2}
\end{align}
and  every solution~$q$ of system~\eqref{sys-q} satisfies, for all~$t\ge0$,
\begin{align}
    &\norm{q(t)}{V}\le C_{q1}\ex^{-\lambda t}\norm{q(0)}{V},\qquad\norm{q}{L^2(\bbR_{0},\rmD(A))}\le
 C_{q2}\norm{q(0)}{V}^{\eta_1}\norm{q(0)}{\rmD(A)}^{\eta_2},\label{FFM-q}\\
  &\norm{Aq-F_{M}(q)}{L^{2\fkr}(\bbR_{0},H)}\le
  C_{q3}\norm{q(0)}{V}^{\beta_1}\norm{q(0)}{\rmD(A)}^{\beta_2}.\label{FFM-A-F}
\end{align}
\end{lemma}

\begin{lemma}\label{LA:DS}
 We have the direct sum~$H=\clU_{M}\oplus \clE_{M}^\perp$.
\end{lemma}

\begin{lemma}\label{LA:bddratP=suffalpha}
We have that~$\sup\limits_{M\in\bbN_+}\frac{\overline\alpha_{M}}{\overline\alpha_{M_{+}}}<+\infty$ and that~
$\sup\limits_{M\in\bbN_+}\norm{P_{\clE_{M}^\perp}^{\clU_{M}}}{\clL(H)}<+\infty$.
\end{lemma}

Now we give the proofs of the Lemmas.

\begin{proof}[Proof of Lemma~\ref{LA:FFM}]
 We have that~$F_M\in\clL(\clE_M)\subset\clC_{\rm b,i}(\clE_M,\clE_M)$.
Next, for~$\widehat q\in\clE_M$ we find that~$\norm{F_M\widehat q}{H}\le
\norm{\nu_2\Delta^2\widehat q}{H}+\lambda\norm{\widehat q}{H}\le(C_1+\lambda\norm{\Id}{\clL(\rmD(A),H)})\norm{\widehat q}{\rmD(A)}$, which gives us~\eqref{FFM-F}.
From its proof, we can see that Lemma~\ref{LA:NN} holds true with~$n=1$, $C_N=C_1$, and~$(\zeta_{1j},\zeta_{2j},\delta_{1j},\delta_{2j})=(1,0,1,0)$. Hence,  from~\eqref{FFM-betaeta1}--\eqref{FFM-betaeta2}, we are looking for constants satisfying
\begin{align}
&\beta_1+\beta_2\ge1,\qquad2\fkr\ge1,\qquad\beta_2<1,\label{betaKS}
\end{align}
so that the solution~$q$ of~\eqref{sys-q} satisfies~\eqref{FFM-q}--\eqref{FFM-A-F}.
Such solution, given by~$q(t)=\rme^{-(\nu_2\Delta^2+\lambda\Id) t}q_0$, satisfies~$\norm{q(t)}{V}\le\norm{\rme^{-\lambda t}q_0}{V}$ and~$\norm{q}{L^2(\bbR_{0},\rmD(A))}\le\norm{\rme^{-\lambda t}q_0}{L^2(\bbR_{0},\rmD(A))}\le
 (2\lambda)^{-\frac12}\norm{q(0)}{\rmD(A)}$. Therefore, \eqref{FFM-q} holds with
$(C_{q1},C_{q2},\eta_1,\eta_2)=(1,(2\lambda)^{-\frac12},0,1)$.
Furthermore,  we find
\begin{align}
\norm{Aq-F_{M}(q)}{H}&=\norm{-2\nu_2\Delta q+(\nu_2+\lambda)q}{H}\le2\nu_2\norm{\Delta q}{H}+(\nu_2+\lambda)\norm{q}{H}\notag\\
&\le(3\nu_2+\lambda)2^\frac12(\norm{\Delta q}{H}^2+\norm{q}{H}^2)^\frac12\le C_2\norm{q}{V}\notag
\end{align}
with~$C_2\coloneqq(3\nu_2+\lambda)2^\frac12\nu_2^{-\frac12}$, where we have used~\eqref{scpV} in the last inequality. 
Therefore, for an arbitrary~$\fkr>1$, 
\[
\norm{Aq-F_{M}(q)}{L^{2\fkr}(\bbR_{0},H)}\le C_2\norm{q}{L^{2\fkr}(\bbR_{0},V)}\le
C_2\norm{\rme^{-\lambda t}q_0}{L^{2\fkr}(\bbR_{0},V)}
\le
C_2(2\fkr\lambda)^{-\frac1{2\fkr}}\norm{q(0)}{V},
\] hence~\eqref{FFM-A-F} holds with~$(C_{q3},\beta_1,\beta_2)=(C_2(2\fkr\lambda)^{-\frac1{2\fkr}},1,0)$. In particular, $(\fkr,\beta_1,\beta_2)=(\fkr,1,0)$ with~$\fkr>1$ satisfies~\eqref{betaKS}. This ends the proof. Note that though~$\fkr\ge\frac12$ would be enough in~\eqref{betaKS}, in the statement of the lemma  it is required that~$1<\fkr<+\infty$ (cf.~\cite[Assump.~3.6]{Rod20-eect}).
\end{proof}

\begin{proof}[Proof of Lemma~\ref{LA:DS}]
It is sufficient to consider the case~$d=1$, because for higher dimensions the result follows by the arguments in~\cite[Sect.~4.8.1, Lem.~4.3]{KunRod19-cocv}.
The result, for $d=1$ and for Dirichlet and Neumann boundary conditions, can be found in~\cite[Lems.~4.3 and~5.1]{RodSturm20}. Hence, it remains to give the proof for the case of periodic boundary conditions with $d=1$, $\Omega\sim[0,L_1)$. To this purpose we use the result in~\cite[Lem.~2.7]{KunRod19-cocv} stating that the direct sum~$H=\clU_{M}\oplus \clE_{M}^\perp$ holds if, and only if, the matrix
$[(E_M,U_M)_H]\in\bbR^{M\times M}$ is invertible, where $[(E_M,U_M)_H]=[(e_i,\indf_{\omega_j^M})_H]$ stands for the matrix with entry~$(e_i,\indf_{\omega_j^M})_H$ the $i$th-row and~$j$th column. Recalling the periodic eigenfunctions in~\eqref{eigf}, by direct computations we find that
\begin{subequations}\label{scpEU}
\begin{align}
    (e_i,\indf_{\omega_j^M})_H &= \tfrac{rL_1}{M};&&\hspace{-6em}\mbox{ for }i=1;
\\
 (e_i,\indf_{\omega_j^M})_H &=\tfrac{L_1}{2\pi (i-1)}
\left(\sin\left(\tfrac{2\pi}{L_1}(i-1)(c_j+\tfrac{rL_1}{2M})\right)-\sin\left(\tfrac{2\pi}{L_1}(i-1)(c_j-\tfrac{rL_1}{2M})\right)\right)
\notag\\
    &=\tfrac{L_1}{\pi (i-1)}
\sin\left(\tfrac{(i-1) r\pi}{M}\right)\cos\left(\tfrac{2\pi}{L_1}(i-1)c_j\right),
&&\hspace{-6em}\mbox{ for odd }i\ge3;
\\
    (e_i,\indf_{\omega_j^M})_H &=\tfrac{L_1}{2\pi i}
\left(\cos\left(\tfrac{2\pi}{L_1}i(c_j-\tfrac{rL_1}{2M})\right)-\cos\left(\tfrac{2\pi}{L_1}i(c_j+\tfrac{rL_1}{2M})\right)\right)\notag\\
    &=\tfrac{L_1}{\pi i}
\sin\left(\tfrac{ir\pi}{M}\right)\sin\left(\tfrac{2\pi}{L_1}ic_j\right)
&&\hspace{-6em}\mbox{ for even }i.
\end{align}
\end{subequations}

Now let us consider the cases~$M$ odd and~$M$ even separately.

\medskip
\noindent
$\bullet$ \emph{The case~$M=1$.} In this case
\[
[(E_M,U_M)_H]=[rL_1]\in\bbR^{1\times1}
\]
which is nonsingular, because $rL_1\ne0$.

\medskip
\noindent
$\bullet$ \emph{The case of odd~$M\ge3$.} In this case, denoting for simplicity
\[
\delta_M\coloneqq \tfrac{r\pi}{M}\quad\mbox{and}\quad \widehat c_j\coloneqq\tfrac{2\pi}{L_1}c_j,
\]
we arrive at
\begin{align}\label{matrixEU}
&[(E_M,U_M)_H]\\
&= \frac{L_1}{\pi}\begin{bmatrix}
 \delta_M& \delta_M & \ldots & \delta_M
\\
 \frac{\sin(1\delta_{M})\sin(1\widehat c_1)}{1}&
 \frac{\sin(1\delta_{M})\sin(1\widehat c_2)}{1} & \ldots &
 \frac{\sin(1\delta_{M})\sin(1\widehat c_{M})}{1}
\\
 \frac{\sin(1\delta_{M})\cos(1\widehat c_1)}{1} &
 \frac{\sin(1\delta_{M})\cos(1\widehat c_2)}{1} & \ldots &
 \frac{\sin(1\delta_{M})\cos(1\widehat c_{M})}{1}
\\
 \frac{\sin(2\delta_{M})\sin(2\widehat c_1)}{2} &
 \frac{\sin(2\delta_{M})\sin(2\widehat c_2)}{2} & \ldots &
 \frac{\sin(2\delta_{M})\sin(2\widehat c_{M})}{2}
\\
 \frac{\sin(2\delta_{M})\cos(2\widehat c_1)}{2} & 
 \frac{\sin(2\delta_{M})\cos(2\widehat c_2)}{2} & \ldots &
 \frac{\sin(2\delta_{M})\cos(2\widehat c_{M})}{2}
\\
 \vdots & \vdots & \ddots & \vdots
\\
 \frac{\sin(\frac{M-1}2\delta_{M})\sin(\frac{M-1}2\widehat c_1)}{\frac{M-1}2} &
 \frac{\sin(\frac{M-1}2\delta_{M})\sin(\frac{M-1}2\widehat c_2)}{\frac{M-1}2} & \ldots & 
 \frac{\sin(\frac{M-1}2\delta_{M})\sin(\frac{M-1}2\widehat c_{M})}{\frac{M-1}2}
\\
 \frac{\sin(\frac{M-1}2\delta_{M})\cos(\frac{M-1}2\widehat c_1)}{\frac{M-1}2} &
 \frac{\sin(\frac{M-1}2\delta_{M})\cos(\frac{M-1}2\widehat c_2)}{\frac{M-1}2} & \ldots &
 \frac{\sin(\frac{M-1}2\delta_{M})\sin(\frac{M-1}2\widehat c_{M})}{\frac{M-1}2}
 \end{bmatrix}\notag
\end{align}

Dividing each row by appropriate nonzero constants, we can see that~$[(E_M,U_M)_H]$ is singular if, and only if, the matrix~$\clM$ as follows is singular,
\begin{equation}\label{matrix-clM}
\clM\coloneqq \begin{bmatrix}
 1&1 & \ldots & 1
\\
 \sin(1\widehat c_1)&
\sin(1\widehat c_2) & \ldots &
\sin(1\widehat c_{M})
\\
\cos(1\widehat c_1) &
 \cos(1\widehat c_2) & \ldots &
 \cos(1\widehat c_{M})
\\
\sin(2\widehat c_1) &
\sin(2\widehat c_2) & \ldots &
\sin(2\widehat c_{M})
\\
\cos(2\widehat c_1) & 
\cos(2\widehat c_2) & \ldots &
\cos(2\widehat c_{M})
\\
 \vdots & \vdots & \ddots & \vdots
\\
\sin(\frac{M-1}2\widehat c_1) &
\sin(\frac{M-1}2\widehat c_2) & \ldots & 
\sin(\frac{M-1}2\widehat c_{M})
\\
\cos(\frac{M-1}2\widehat c_1) &
\cos(\frac{M-1}2\widehat c_2) & \ldots &
\sin(\frac{M-1}2\widehat c_{M})
 \end{bmatrix}
\end{equation}
Note that~$\delta_M\ne0$ and that~$\sin(i\delta_{M})\ne0$ for all~$1\le i\le \frac{M-1}2$, because~$0<i\delta_{M}\le r\pi<\pi$.

Let us be given~$v\in\bbR^{1\times M}$ such that~$v\clM =0\in\bbR^{1\times M}$. This means that the~$\widehat c_j$s are distinct zeros in the open interval~$(0,2\pi)$ of the function
\[
\varphi(x)\coloneqq v_{1,1}+{\textstyle\sum\limits_{i=1}^{\frac{M-1}2}}(v_{1,2i+1}\cos(ix)+v_{1,2i}\sin(ix)).
\]
Now note that the set~$\widehat C\coloneqq\{\widehat c_j\mid 1\le j\le M\}$ satisfies~$\widehat C=2\pi-\widehat C$, because recalling~\eqref{cM},
\begin{align}\notag
c_{M+1-j}&=\tfrac{(2(M+1-j)-1)L_1}{2M}=L_1+\tfrac{(2(1-j)-1)L_1}{2M}=L_1-c_j.
\end{align}
Thus~$\widehat c_{M+1-j}=2\pi-\widehat c_j$.
Now,  since both~$\widehat c_{M+1-j}$ and~$\widehat c_j$ are zeros of~$\varphi$, we find
\begin{align}
0&=v_{1,1}+{\textstyle\sum\limits_{i=1}^{\frac{M-1}2}}(v_{1,2i+1}\cos(i\widehat c_j)+v_{1,2i}\sin(i\widehat c_j)),\notag\\
0&=v_{1,1}+{\textstyle\sum\limits_{i=1}^{\frac{M-1}2}}(v_{1,2i+1}\cos(i(2\pi-\widehat c_j))+v_{1,2i}\sin(i(2\pi-\widehat c_j)))\notag\\
&=v_{1,1}+{\textstyle\sum\limits_{i=1}^{\frac{M-1}2}}(v_{1,2i+1}\cos(i\widehat c_j)-v_{1,2i}\sin(\widehat c_j)),\notag
\end{align}
from which, after summing up, we obtain
\begin{align}\notag
0&=v_{1,1}+{\textstyle\sum\limits_{i=1}^{\frac{M-1}2}}v_{1,2i+1}\cos(i\widehat c_j),
\end{align}
which implies that the function
$
\psi(x)\coloneqq v_{1,1}+{\textstyle\sum\limits_{i=1}^{\frac{M-1}2}}v_{1,2i+1}\cos(ix)
$
has~$\tfrac{M-1}{2}+1$ distinct zeros in~$(0,\pi]$, namely the zeros~$\widehat c_j$ with
$1\le j\le \tfrac{M-1}{2}+1$; note that~$\widehat c_{\tfrac{M-1}{2}+1}=\pi$. Therefore, we can conclude that 
\begin{equation}\label{v-cos0}
v_{1,1}=v_{1,2i}=0,\qquad 1\le i\le \tfrac{M-1}2,
\end{equation}
because from~\cite[Sect.~5, Thm.~4]{Boyd07} , if~$\psi$ is not identically zero then it has at most ~$\tfrac{M-1}{2}$ zeros in~$[0,\pi]$.
Consequently,~$\varphi=\psi+\xi=\xi$ with
$
\xi(x)\coloneqq{\textstyle\sum\limits_{i=1}^{\frac{M-1}2}}v_{1,2i+1}\sin(ix)
$
 having $\tfrac{M-1}{2}+1$ zeros in~$(0,\pi]$ which implies that
\begin{equation}\label{v-sin0}
v_{1,2i+1}=0,\qquad 1\le i\le \tfrac{M-1}2,
\end{equation}
because if~$\xi(x)$
 is not identically zero then it has at most ~$\tfrac{M-1}{2}-1$ zeros in~$(0,\pi)$, due to
~\cite[Prop.~4.1]{KunRod19-cocv} \cite[Sect.~5, Thm.~4]{Boyd07} .
From~\eqref{v-cos0} and~\eqref{v-sin0} it follows that~$v=0$. Hence, $\clM$ is invertible in the case of odd~$M$..

\medskip
\noindent
$\bullet$ \emph{The case of even~$M$.} In this case we have
\begin{align}\label{matrixEU-e}
&[(E_M,U_M)_H]\\
&= \frac{L_1}{\pi}\begin{bmatrix}
 \delta_M& \delta_M & \ldots & \delta_M
\\
 \frac{\sin(1\delta_{M})\sin(1\widehat c_1)}{1}&
 \frac{\sin(1\delta_{M})\sin(1\widehat c_2)}{1} & \ldots &
 \frac{\sin(1\delta_{M})\sin(1\widehat c_{M})}{1}
\\
 \frac{\sin(1\delta_{M})\cos(1\widehat c_1)}{1} &
 \frac{\sin(1\delta_{M})\cos(1\widehat c_2)}{1} & \ldots &
 \frac{\sin(1\delta_{M})\cos(1\widehat c_{M})}{1}
\\
 \frac{\sin(2\delta_{M})\sin(2\widehat c_1)}{2} &
 \frac{\sin(2\delta_{M})\sin(2\widehat c_2)}{2} & \ldots &
 \frac{\sin(2\delta_{M})\sin(2\widehat c_{M})}{2}
\\
 \frac{\sin(2\delta_{M})\cos(2\widehat c_1)}{2} & 
 \frac{\sin(2\delta_{M})\cos(2\widehat c_2)}{2} & \ldots &
 \frac{\sin(2\delta_{M})\cos(2\widehat c_{M})}{2}
\\
 \vdots & \vdots & \ddots & \vdots
\\
 \frac{\sin(\frac{M-2}2\delta_{M})\sin(\frac{M-2}2\widehat c_1)}{\frac{M-2}2} &
 \frac{\sin(\frac{M-2}2\delta_{M})\sin(\frac{M-2}2\widehat c_2)}{\frac{M-2}2} & \ldots & 
 \frac{\sin(\frac{M-2}2\delta_{M})\sin(\frac{M-2}2\widehat c_{M})}{\frac{M-2}2}
\\
 \frac{\sin(\frac{M-2}2\delta_{M})\cos(\frac{M-2}2\widehat c_1)}{\frac{M-2}2} &
 \frac{\sin(\frac{M-2}2\delta_{M})\cos(\frac{M-2}2\widehat c_2)}{\frac{M-2}2} & \ldots &
 \frac{\sin(\frac{M-2}2\delta_{M})\sin(\frac{M-2}2\widehat c_{M})}{\frac{M-2}2}
\\
 \frac{\sin(\frac{M}2\delta_{M})\sin(\frac{M}2\widehat c_1)}{\frac{M}2} &
 \frac{\sin(\frac{M}2\delta_{M})\sin(\frac{M}2\widehat c_2)}{\frac{M}2} & \ldots & 
 \frac{\sin(\frac{M}2\delta_{M})\sin(\frac{M}2\widehat c_{M})}{\frac{M}2}
 \end{bmatrix}\notag
\end{align}

The matrix~$[(E_M,U_M)_H]$ is singular if, and only if, the matrix
\begin{equation}\label{matrix-clM}
\clM\coloneqq \begin{bmatrix}
 1&1 & \ldots & 1
\\
 \sin(1\widehat c_1)&
\sin(1\widehat c_2) & \ldots &
\sin(1\widehat c_{M})
\\
\cos(1\widehat c_1) &
 \cos(1\widehat c_2) & \ldots &
 \cos(1\widehat c_{M})
\\
\sin(2\widehat c_1) &
\sin(2\widehat c_2) & \ldots &
\sin(2\widehat c_{M})
\\
\cos(2\widehat c_1) & 
\cos(2\widehat c_2) & \ldots &
\cos(2\widehat c_{M})
\\
 \vdots & \vdots & \ddots & \vdots
\\
\sin(\frac{M-2}2\widehat c_1) &
\sin(\frac{M-2}2\widehat c_2) & \ldots & 
\sin(\frac{M-2}2\widehat c_{M})
\\
\cos(\frac{M-2}2\widehat c_1) &
\cos(\frac{M-2}2\widehat c_2) & \ldots &
\sin(\frac{M-2}2\widehat c_{M})
\\
\sin(\frac{M}2\widehat c_1) &
\sin(\frac{M}2\widehat c_2) & \ldots & 
\sin(\frac{M}2\widehat c_{M})
 \end{bmatrix}
\end{equation}
is singular. 
Note that~$\sin(i\delta_{M})\ne0$ for all~$1\le i\le \frac{M}2$, because~$0<i\delta_{M}\le r\pi<\pi$.

If~$v\in\bbR^{1\times M}$ satisfies~$v\clM =0\in\bbR^{1\times M}$, then the~$\widehat c_j$s are distinct zeros in the open interval~$(0,2\pi)$ of the function
\[
\varphi(x)\coloneqq v_{1,1}+{\textstyle\sum\limits_{i=1}^{\frac{M-2}2}}v_{1,2i+1}\cos(ix)+{\textstyle\sum\limits_{i=1}^{\frac{M}2}}v_{1,2i}\sin(ix).
\]
The set~$\widehat C\coloneqq\{\widehat c_j\mid 1\le j\le M\}$ satisfies again~$\widehat C=2\pi-\widehat C$, and since both~$\widehat c_{M+1-j}$ and~$\widehat c_j$ are zeros of~$\varphi$, we can obtain
\begin{align}\notag
0&=v_{1,1}+{\textstyle\sum\limits_{i=1}^{\frac{M-2}2}}v_{1,2i+1}\cos(i\widehat c_j),
\end{align}
which implies that the function
$
\psi(x)\coloneqq v_{1,1}+{\textstyle\sum\limits_{i=1}^{\frac{M-2}2}}v_{1,2i+1}\cos(ix)
$
has~$\tfrac{M}{2}$ distinct zeros in~$(0,\pi)$, namely the zeros~$\widehat c_j$ with
$1\le j\le \tfrac{M}{2}$; note that~$\widehat c_{\tfrac{M}{2}}<\pi$. Therefore,  due to
~\cite[Sect.~5, Thm.~4]{Boyd07}, we can conclude that 
\begin{equation}\label{v-cos0-e}
v_{1,1}=v_{1,2i+1}=0,\qquad 1\le i\le \tfrac{M-2}2.
\end{equation}
Consequently, ~$\varphi=\psi+\xi=\xi$ with
$
\xi(x)\coloneqq{\textstyle\sum\limits_{i=1}^{\frac{M}2}}v_{1,2i}\sin(ix)
$
 having $\tfrac{M}{2}$ zeros in~$(0,\pi)$ which implies that
\begin{equation}\label{v-sin0-e}
v_{1,2i}=0,\qquad 1\le i\le \tfrac{M}2,
\end{equation}
 due to
~\cite[Prop.~4.1]{KunRod19-cocv} \cite[Sect.~5, Thm.~4]{Boyd07}.
From~\eqref{v-cos0-e} and~\eqref{v-sin0-e} it follows that~$v=0$. Hence, $\clM$ is invertible also in the case of even~$M$. This finishes the proof.
\end{proof}

\begin{proof}[Proof of Lemma~\ref{LA:bddratP=suffalpha}]
An upper bound for the quotient~$\frac{\overline\alpha_{M}}{\overline\alpha_{M_{+}}}$ can be found in~\cite[Sect.~2.2]{Rod20-eect} for the cases of Dirichlet and Neumann boundary conditions.
In the case of periodic boundary conditions, recalling the eigenvalues~$\overline\alpha_{M}$ and~$\overline\alpha_{M_+}$ defined in~\eqref{VDAalphaM} we obtain
\[
\overline\alpha_{M}\le4\pi^2{\textstyle\sum\limits_{n=1}^{d}}(\tfrac{M}{2})^2(\tfrac{1}{L_n})^2\le d M^2\pi^2\underline L^{-2}\quad\mbox{and}\quad\overline\alpha_{M_+}\ge M^2\pi^2\overline L^{-2}
\]
with~$\underline L\coloneqq\min\{L_n\mid 1\le n\le d\}$ and~$\overline L\coloneqq\max\{L_n\mid 1\le n\le d\}$. Hence~$\frac{\overline\alpha_{M}}{\overline\alpha_{M_{+}}}\le
d \underline L^{-2}\overline L^{2}$.

 Concerning an upper bound for~$\norm{P_{\clE_{M}^\perp}^{\clU_{M}}}{\clL(H)}$, due to the arguments in~\cite[Lems.~4.3 and~4.4]{KunRod19-cocv} we can restrict ourselves to the case~$d=1$. In this case, an upper bound  can be derived from~\cite[Thms.~4.4 and~5.2, and~(3.5) in Cor.~3.5]{RodSturm20}, for Dirichlet and Neumann boundary conditions.
Then, it remains to consider the case~$d=1$ under periodic boundary conditions.

We compute~$\norm{P_{\clE_{M}^\perp}^{\clU_{M}}}{\clL(H)}$ by using~\cite[Cor.~2.9]{KunRod19-cocv} and following arguments as in~\cite[Sect.~4.2]{RodSturm20}.
In order to apply~\cite[Cor.~2.9]{KunRod19-cocv} we orthonormalize the family of actuators and (auxiliary) eigenfunctions, since it is clear that such families are orthogonal, we just need to normalize them. Doing so we obtain
 \begin{align}
& \overline e_i(x_1)\coloneqq \begin{cases}
(\frac1{L_1})^{\frac12},&\mbox{ if $i=1$},\\
(\frac2{L_1})^{\frac12}\cos(\tfrac{2(i-1)\pi x_1}{L_1}),&\mbox{ if $i\ge3$ is odd},\\
(\frac2{L_1})^{\frac12}\sin(\tfrac{2i\pi  x_1}{L_1}),&\mbox{ if $i$ is even},
\end{cases}
&&\mbox{ for }\clG=\clG_{\rm per},
\\
& \overline \indf_{\omega_j^M}\coloneqq (\tfrac{M}{rL_1})^{\frac12}\indf_{\omega_j^M},
 \end{align}
and, for $(\overline e_i,\overline\indf_{\omega_j^M})_H$, recalling~\eqref{scpEU},
\begin{align}
    (\overline e_i,\overline\indf_{\omega_j^M})_H &
= (\tfrac{r}{M})^\frac12;&&\hspace{0em}\mbox{ for }i=1;\notag\\
 (\overline e_i,\overline\indf_{\omega_j^M})_H &=(\tfrac{2M}{r})^{\frac12}\tfrac{1}{\pi (i-1)}
\sin\left(\tfrac{(i-1) r\pi}{M}\right)\cos\left(\tfrac{2\pi}{L_1}(i-1)c_j\right),
&&\hspace{0em}\mbox{ for odd }i\ge3;
\notag\\
   (\overline e_i,\overline\indf_{\omega_j^M})_H &=(\tfrac{2M}{r})^{\frac12}\tfrac{1}{\pi i}
\sin\left(\tfrac{ir\pi}{M}\right)\sin\left(\tfrac{2\pi}{L_1}ic_j\right)
&&\hspace{0em}\mbox{ for even }i.\notag
\end{align}

We construct the analogs of~\eqref{matrixEU} and~\eqref{matrixEU-e},
\begin{align}\label{matrixEUo}
&[(\overline E_M,\overline U_M)_H]\coloneqq[(\overline e_i, \overline\indf_{\omega_j^M})_H]\\
&= \left(\frac{2M}{r\pi^2}\right)^\frac12 \begin{bmatrix}
2^{-\frac12}\delta_M& 2^{-\frac12}\delta_M & \ldots & 2^{-\frac12}\delta_M
\\
 \frac{\sin(1\delta_{M})\sin(1\widehat c_1)}{1}&
 \frac{\sin(1\delta_{M})\sin(1\widehat c_2)}{1} & \ldots &
 \frac{\sin(1\delta_{M})\sin(1\widehat c_{M})}{1}
\\
 \frac{\sin(1\delta_{M})\cos(1\widehat c_1)}{1} &
 \frac{\sin(1\delta_{M})\cos(1\widehat c_2)}{1} & \ldots &
 \frac{\sin(1\delta_{M})\cos(1\widehat c_{M})}{1}
\\
 \frac{\sin(2\delta_{M})\sin(2\widehat c_1)}{2} &
 \frac{\sin(2\delta_{M})\sin(2\widehat c_2)}{2} & \ldots &
 \frac{\sin(2\delta_{M})\sin(2\widehat c_{M})}{2}
\\
 \frac{\sin(2\delta_{M})\cos(2\widehat c_1)}{2} & 
 \frac{\sin(2\delta_{M})\cos(2\widehat c_2)}{2} & \ldots &
 \frac{\sin(2\delta_{M})\cos(2\widehat c_{M})}{2}
\\
 \vdots & \vdots & \ddots & \vdots
\\
 \frac{\sin(\frac{M-1}2\delta_{M})\sin(\frac{M-1}2\widehat c_1)}{\frac{M-1}2} &
 \frac{\sin(\frac{M-1}2\delta_{M})\sin(\frac{M-1}2\widehat c_2)}{\frac{M-1}2} & \ldots & 
 \frac{\sin(\frac{M-1}2\delta_{M})\sin(\frac{M-1}2\widehat c_{M})}{\frac{M-1}2}
\\
 \frac{\sin(\frac{M-1}2\delta_{M})\cos(\frac{M-1}2\widehat c_1)}{\frac{M-1}2} &
 \frac{\sin(\frac{M-1}2\delta_{M})\cos(\frac{M-1}2\widehat c_2)}{\frac{M-1}2} & \ldots &
 \frac{\sin(\frac{M-1}2\delta_{M})\sin(\frac{M-1}2\widehat c_{M})}{\frac{M-1}2}
 \end{bmatrix}\notag
\end{align}
for odd~$M$, and
\begin{align}\label{matrixEU-eo}
&[(\overline E_M,\overline U_M)_H]\coloneqq[(\overline e_i, \overline\indf_{\omega_j^M})_H]\\
&= \left(\frac{2M}{r\pi^2}\right)^\frac12\begin{bmatrix}
2^{-\frac12} \delta_M& 2^{-\frac12}\delta_M & \ldots & 2^{-\frac12}\delta_M
\\
 \frac{\sin(1\delta_{M})\sin(1\widehat c_1)}{1}&
 \frac{\sin(1\delta_{M})\sin(1\widehat c_2)}{1} & \ldots &
 \frac{\sin(1\delta_{M})\sin(1\widehat c_{M})}{1}
\\
 \frac{\sin(1\delta_{M})\cos(1\widehat c_1)}{1} &
 \frac{\sin(1\delta_{M})\cos(1\widehat c_2)}{1} & \ldots &
 \frac{\sin(1\delta_{M})\cos(1\widehat c_{M})}{1}
\\
 \frac{\sin(2\delta_{M})\sin(2\widehat c_1)}{2} &
 \frac{\sin(2\delta_{M})\sin(2\widehat c_2)}{2} & \ldots &
 \frac{\sin(2\delta_{M})\sin(2\widehat c_{M})}{2}
\\
 \frac{\sin(2\delta_{M})\cos(2\widehat c_1)}{2} & 
 \frac{\sin(2\delta_{M})\cos(2\widehat c_2)}{2} & \ldots &
 \frac{\sin(2\delta_{M})\cos(2\widehat c_{M})}{2}
\\
 \vdots & \vdots & \ddots & \vdots
\\
 \frac{\sin(\frac{M-2}2\delta_{M})\sin(\frac{M-2}2\widehat c_1)}{\frac{M-2}2} &
 \frac{\sin(\frac{M-2}2\delta_{M})\sin(\frac{M-2}2\widehat c_2)}{\frac{M-2}2} & \ldots & 
 \frac{\sin(\frac{M-2}2\delta_{M})\sin(\frac{M-2}2\widehat c_{M})}{\frac{M-2}2}
\\
 \frac{\sin(\frac{M-2}2\delta_{M})\cos(\frac{M-2}2\widehat c_1)}{\frac{M-2}2} &
 \frac{\sin(\frac{M-2}2\delta_{M})\cos(\frac{M-2}2\widehat c_2)}{\frac{M-2}2} & \ldots &
 \frac{\sin(\frac{M-2}2\delta_{M})\sin(\frac{M-2}2\widehat c_{M})}{\frac{M-2}2}
\\
 \frac{\sin(\frac{M}2\delta_{M})\sin(\frac{M}2\widehat c_1)}{\frac{M}2} &
 \frac{\sin(\frac{M}2\delta_{M})\sin(\frac{M}2\widehat c_2)}{\frac{M}2} & \ldots & 
 \frac{\sin(\frac{M}2\delta_{M})\sin(\frac{M}2\widehat c_{M})}{\frac{M}2}
 \end{bmatrix}\notag
\end{align}
for even~$M$.
Then, for either case we construct the symmetric matrix
\[
\Theta\coloneqq [(\overline E_M,\overline U_M)_H][(\overline E_M,\overline U_M)_H]^\top
\]
where~$\Xi^\top$ stands for the transpose of the matrix~$\Xi$. By~\cite[Cor.~2.9]{KunRod19-cocv} we have that
\begin{equation}\label{normOP}
\norm{P_{\clU_M}^{\clE_M^\perp}}{\clL(H)}=\left(\min_\theta\{\theta\mbox{ is an eigenvalue of }\Theta\}\right)^{-\frac12}.
\end{equation}

We show next that~$\Theta$ is a diagonal matrix. The entry~$\Theta_{1j}$ in the $1$st row and~$j$th column, with~$1<j\le M$ is given by
\begin{align}
\Theta_{1j}&=\tfrac{2^{\frac12}M}{r\pi^2} \delta_M \tfrac{\sin\left(\tfrac{(j-1)}{2}\delta_{M}\right)}{\tfrac{(j-1)}{2}}{\textstyle\sum\limits_{k=1}^{M}}\cos\left(\tfrac{(j-1)}{2}\widehat c_k\right)\quad&&\mbox{if $3\le j\le M$ is odd};\notag\\
\Theta_{1j}&=\tfrac{2^{\frac12}M}{r\pi^2} \delta_M \tfrac{\sin\left(\tfrac{j}{2}\delta_{M}\right)}{\tfrac{j}{2}}{\textstyle\sum\limits_{k=1}^{M}}\sin\left(\tfrac{j}{2}\widehat c_k\right)\quad&&\mbox{if $2\le j\le M$ is even}.\notag
\end{align}
Writing for a given~$ m\in\bbN_+$,
\begin{align}\label{mck}
m\widehat c_k=m\tfrac{2\pi(2k-1)}{2M}=m\tfrac{\pi (2(k-1)+1)}{M}=\tfrac{m \pi}{M}
+(k-1)\tfrac{2m\pi }{M}
\end{align}
we find
\begin{align}
&{\textstyle\sum\limits_{k=1}^{M}}\cos\left(m\widehat c_k\right)={\textstyle\sum\limits_{n=0}^{M-1}}\cos\left(a
+nb\right)
\quad\mbox{and}\quad
{\textstyle\sum\limits_{k=1}^{M}}\sin\left(m\widehat c_k\right)={\textstyle\sum\limits_{n=0}^{M-1}}\sin\left(a
+nb\right),\notag\\
\intertext{with}
&a\coloneqq \tfrac{m \pi}{M}\quad\mbox{and}\quad b\coloneqq\tfrac{2m\pi}{M}.\notag
\end{align}
By using the results in~\cite{Knapp09}, we obtain that if~$\sin(\frac{b}{2})\ne0$ then
\begin{align}
&{\textstyle\sum\limits_{k=1}^{M}}\cos\left(m\widehat c_k\right)=\tfrac{\sin(\frac{Mb}{2})}{\sin(\frac{b}{2})}\cos(a+\tfrac{(M-1)b}{2})
\quad\mbox{and}\quad
{\textstyle\sum\limits_{k=1}^{M}}\sin\left(m\widehat c_k\right)=\tfrac{\sin(\frac{Mb}{2})}{\sin(\frac{b}{2})}\sin(a+\tfrac{(M-1)b}{2}).\notag
\end{align}
Note that for~$1\le m<M$ we have that~$0<\frac{b}2<\pi$. Thus~$\sin(\frac{b}{2})\ne0$ and~$\sin(\frac{Mb}{2})=\sin(m\pi)=0$, we can conclude that
\begin{align}\label{sumCS-(M-1)2}
&{\textstyle\sum\limits_{k=1}^{M}}\cos\left(m\widehat c_k\right)=0=
{\textstyle\sum\limits_{k=1}^{M}}\sin\left(m\widehat c_k\right)\quad\mbox{for all}\quad
1\le m<M,
\end{align}
which implies that
\begin{equation}\label{Theta1j}
\Theta_{1j}=0\quad\mbox{for all}\quad 1<j\le M.
\end{equation}

Next we compute~$\Theta_{ij}$ for~$1<i<j\le M$. We find
\begin{align}
\Theta_{ij}&=\tfrac{2M}{r\pi^2} \tfrac{\sin\left(\tfrac{(i-1)}{2}\delta_{M}\right)}{\tfrac{(i-1)}{2}}\tfrac{\sin\left(\tfrac{(j-1)}{2}\delta_{M}\right)}{\tfrac{(j-1)}{2}}{\textstyle\sum\limits_{k=1}^{M}}\cos\left(\tfrac{(i-1)}{2}\widehat c_k\right)\cos\left(\tfrac{(j-1)}{2}\widehat c_k\right),&&\mbox{for odd $i$, odd~$j$};
\notag\\
\Theta_{ij}&=\tfrac{2M}{r\pi^2} \tfrac{\sin\left(\tfrac{(i-1)}{2}\delta_{M}\right)}{\tfrac{i-1}{2}}\tfrac{\sin\left(\tfrac{j}{2}\delta_{M}\right)}{\tfrac{j}{2}}{\textstyle\sum\limits_{k=1}^{M}}\cos\left(\tfrac{(i-1)}{2}\widehat c_k\right)\sin\left(\tfrac{j}{2}\widehat c_k\right),&&\mbox{for odd $i$, even $j$};
\notag\\
\Theta_{ij}&=\tfrac{2M}{r\pi^2} \tfrac{\sin\left(\tfrac{i}{2}\delta_{M}\right)}{\tfrac{i}{2}}\tfrac{\sin\left(\tfrac{(j-1)}{2}\delta_{M}\right)}{\tfrac{(j-1)}{2}}{\textstyle\sum\limits_{k=1}^{M}}\sin\left(\tfrac{i}{2}\widehat c_k\right)\cos\left(\tfrac{(j-1)}{2}\widehat c_k\right),&&\mbox{for even $i$, odd $j$};
\notag\\
\Theta_{ij}&=\tfrac{2M}{r\pi^2} \tfrac{\sin\left(\tfrac{i}{2}\delta_{M}\right)}{\tfrac{i}{2}}\tfrac{\sin\left(\tfrac{j}{2}\delta_{M}\right)}{\tfrac{j}{2}}{\textstyle\sum\limits_{k=1}^{M}}\sin\left(\tfrac{i}{2}\widehat c_k\right)\sin\left(\tfrac{j}{2}\widehat c_k\right),&&\mbox{for even $i$, even $j$}.\notag
\end{align}

Recalling the relations
\begin{align}
2\cos(z)\cos(w)&=\cos(z-w)+\cos(z+w),\notag\\
2\sin(z)\sin(w)&=\cos(z-w)-\cos(z+w),\notag\\
2\sin(z)\cos(w)&=\sin(z+w)+\sin(z-w),\notag 
\end{align}
we find that
\begin{align}
&2\cos\left(\tfrac{(i-1)}{2}\widehat c_k\right)\cos\left(\tfrac{(j-1)}{2}\widehat c_k\right)
=\cos(\tfrac{(j-i)}{2}\widehat c_k)+\cos(\tfrac{(i+j-2)}{2}\widehat c_k), &&\mbox{(odd $i$, odd~$j$)}
\notag\\
&2\cos\left(\tfrac{(i-1)}{2}\widehat c_k\right)\sin\left(\tfrac{j}{2}\widehat c_k\right)
=\sin(\tfrac{(j+i-1)}{2}\widehat c_k)+\sin(\tfrac{(j-i+1)}{2}\widehat c_k),&&\mbox{(odd $i$, even~$j$)}
\notag \\
&2\sin\left(\tfrac{i}{2}\widehat c_k\right)\cos\left(\tfrac{(j-1)}{2}\widehat c_k\right)
=\sin(\tfrac{(j+i-1)}{2}\widehat c_k)-\sin(\tfrac{-(i-j+1)}{2}\widehat c_k).
&&\mbox{(even $i$, odd~$j$)}
\notag\\
&2\sin\left(\tfrac{i}{2}\widehat c_k\right)\sin\left(\tfrac{j}{2}\widehat c_k\right)
=\cos(\tfrac{(j-i)}{2}\widehat c_k)-\cos(\tfrac{(i+j)}{2}\widehat c_k),&&\mbox{(even $i$, even~$j$)}.\notag
 \end{align}

Note that the frequencies on the right-hand side
\[
m\in\left\{\tfrac{(j-i)}{2},\tfrac{(i+j-2)}{2},\tfrac{(j+i-1)}{2},\tfrac{(j-i+1)}{2},\tfrac{(j+i)}{2}\right\}\quad\mbox{and}\quad n=\tfrac{(j-i-1)}{2}
\]
satisfy $1\le m<M$ and~$0\le n<M$, due to~$1< i< j\le M$. Hence, by~\eqref{sumCS-(M-1)2}, we obtain
\begin{equation}\label{Thetaij}
\Theta_{ij}=0\quad\mbox{for all}\quad 1<i<j\le M.
\end{equation}

From~\eqref{Theta1j} and~\eqref{Thetaij} we conclude that~$\Theta$ is a lower triangular matrix. Then, the symmetry of~$\Theta$ implies that~$\Theta$ is a diagonal matrix. Consequently, its eigenvalues are its diagonal entries given by
\begin{align}
\Theta_{11}&=\tfrac{M}{r\pi^2}M\delta_M^2=r,&&\mbox{}
\notag\\
\Theta_{ii}&=\tfrac{2M}{r\pi^2} \sin^2\left(\tfrac{(i-1)}{2}\delta_{M}\right)\tfrac4{(i-1)^2}{\textstyle\sum\limits_{k=1}^{M}}\cos^2\left(\tfrac{(i-1)}{2}\widehat c_k\right)&&\mbox{for odd $1\le i\le M$}\notag\\
\Theta_{ii}&=\tfrac{2M}{r\pi^2} \sin^2\left(\tfrac{i}{2}\delta_{M}\right)\tfrac4{i^2}{\textstyle\sum\limits_{k=1}^{M}}\sin^2\left(\tfrac{i}{2}\widehat c_k\right),&&\mbox{for even $2\le i\le M$}.\notag
\notag
\end{align}

In order to find an explicit expression for the eigenvalues we compute
\begin{align}
{\textstyle\sum\limits_{k=1}^{M}}\cos^2\left(n\widehat c_k\right)
&=\frac12{\textstyle\sum\limits_{k=1}^{M}}1+\cos\left(2n\widehat c_k\right)
=\frac{M}2+{\textstyle\sum\limits_{k=1}^{M}}\cos\left(2n\widehat c_k\right),\notag\\
{\textstyle\sum\limits_{k=1}^{M}}\sin^2\left(n\widehat c_k\right)
&=\frac12{\textstyle\sum\limits_{k=1}^{M}}1-\cos\left(2n\widehat c_k\right)
=\frac{M}2-{\textstyle\sum\limits_{k=1}^{M}}\cos\left(2n\widehat c_k\right),\notag
\end{align}
and use~\eqref{sumCS-(M-1)2}  to obtain
\[
{\textstyle\sum\limits_{k=1}^{M}}\cos^2\left(n\widehat c_k\right)=\frac{M}2={\textstyle\sum\limits_{k=1}^{M}}\sin^2\left(n\widehat c_k\right), \quad\mbox{for~$2n<M$}.
\]

For the case~$2n=M$, which appear in the case of even~$M$, we find that
\[
n\widehat c_k=\tfrac{n \pi}{M}
+(k-1)\tfrac{2n\pi }{M}=\tfrac{\pi}{2}
+(k-1)\pi,
\]
which implies
\[
{\textstyle\sum\limits_{k=1}^{M}}\sin^2\left(n\widehat c_k\right)=M\quad\mbox{for~$2n=M$}.
\]
We can conclude that the set~${\rm Eig}(\Theta)$ of eigenvalues is given by
\begin{align}
&{\rm Eig}(\Theta)=\{r\},&&\mbox{$M=1$};\notag\\
&{\rm Eig}(\Theta)=
\{r\}\;{\textstyle\bigcup}\;\left\{\tfrac{M^2}{r\pi^2}\left(\left.\tfrac{\sin\left(n\delta_{M}\right)}{n}\right)^2 \,\right|\; 1\le n\le\tfrac{M-1}2\right\},&&\mbox{odd~$M>1$};
\notag\\
&{\rm Eig}(\Theta)=
\left\{r,\tfrac{2M^2}{r\pi^2}\left(\tfrac{\sin\left(\tfrac{M}2\delta_{M}\right)}{\tfrac{M}2}\right)^2\right\},&&\mbox{$M=2$}.
\notag\\
&{\rm Eig}(\Theta)=
\left\{r,\tfrac{2M^2}{r\pi^2}\left(\tfrac{\sin\left(\tfrac{M}2\delta_{M}\right)}{\tfrac{M}2}\right)^2\right\}\;{\textstyle\bigcup}\;\left\{\tfrac{M^2}{r\pi^2}\left(\left.\tfrac{\sin\left(n\delta_{M}\right)}{n}\right)^2 \,\right|\; 1\le n\le\tfrac{M}2-1\right\},&&\mbox{even~$M>2$}.\notag
\end{align}
By~\cite[Lem.~4.6]{RodSturm20}.  the function~$t\mapsto\frac{\sin^2(\frac{\delta_M}2 t)}{t^2}$ is strictly decreasing for~$t\in(0,M]$. Hence,
\begin{align}
&\min{\rm Eig}(\Theta)=
\min\left\{r,\tfrac{M^2}{r\pi^2}\left(\tfrac{\sin\left(\tfrac{M-1}2\delta_M\right)}{\tfrac{M-1}2}\right)^2\right\},\quad&&\mbox{odd~$M\ge3$};
\notag\\
&\min{\rm Eig}(\Theta)=
\min\left\{r,\tfrac{2M^2}{r\pi^2}\left(\tfrac{\sin\left(\tfrac{M}2\delta_{M}\right)}{\tfrac{M}2}\right)^2\right\},\quad&&\mbox{even~$M=2$};
\notag\\
&\min{\rm Eig}(\Theta)=
\min\left\{r,\tfrac{2M^2}{r\pi^2}\left(\tfrac{\sin\left(\tfrac{M}2\delta_{M}\right)}{\tfrac{M}2}\right)^2,\tfrac{M^2}{r\pi^2}\left(\tfrac{\sin\left(\tfrac{M-2}2\delta_M\right)}{\tfrac{M-2}2}\right)^2\right\},\quad&&\mbox{even~$M\ge4$};\notag
\end{align}
which we can further make more explicit, recalling that~$\delta_M=\frac{r\pi}{M}$, as
\begin{align}
&\min{\rm Eig}(\Theta)=
\min\left\{r,r(\tfrac{2M}{(M-1)r\pi})^2\sin^2\left(\tfrac{(M-1)r\pi}{2M}\right)\right\},&&\mbox{for odd~$M\ge3$},\notag\\
&\min{\rm Eig}(\Theta)=
\min\left\{r,2r(\tfrac{2}{r\pi})^2\sin^2\left(\tfrac{r\pi}{2}\right)\right\},&&\mbox{for even~$M=2$}.\notag\\
&\min{\rm Eig}(\Theta)=
\min\left\{r,2r(\tfrac{2}{r\pi})^2\sin^2\left(\tfrac{r\pi}{2}\right),
r(\tfrac{2M}{(M-2)r\pi})^2\sin^2\left(\tfrac{(M-2)r\pi}{2M}\right)\right\},&&\mbox{for even~$M\ge4$}.\notag
\end{align}

Next, note that~$\sin(t)<t<2\sin(t)$ for all~$0<t\le\frac{\pi}2$. Indeed,  the function~$\phi(t)\coloneqq t-\sin(t)$ strictly increases for~$t\in(0,\frac{\pi}2]$ and satisfies~$\phi(0)=0$, and the function~$\psi(t)\coloneqq 2\sin(t)-t$  strictly increases for~$t\in[0,\frac{\pi}3)$, strictly decreases for~$t\in(\frac{\pi}3,\frac{\pi}2)$ and satisfies~$\psi(0)=0$ and~$\psi(\frac{\pi}2)=2-\frac{\pi}2>0$.
Therefore, we can conclude that 
\begin{align}
&r>r(\tfrac{2M}{(M-1)r\pi})^2\sin^2\left(\tfrac{(M-1)r\pi}{2M}\right)
&&\quad\mbox{for all $M>1$},\notag\\
&2r(\tfrac{2}{r\pi})^2\sin^2\left(\tfrac{r\pi}{2}\right)>r>r(\tfrac{2M}{(M-2)r\pi})^2\sin^2\left(\tfrac{(M-2)r\pi}{2M}\right),&&\quad\mbox{for all $M>2$},\notag
\end{align}
which leads us
\begin{subequations}\label{minEigTheta_exp}
\begin{align}
&\min{\rm Eig}(\Theta)=r,&&\mbox{for $M\in\{1,2\}$};\\
&\min{\rm Eig}(\Theta)=r(\tfrac{2M}{(M-1)r\pi})^2\sin^2\left(\tfrac{(M-1)r\pi}{2M}\right),&&\mbox{for odd~$M\ge3$},\\
&\min{\rm Eig}(\Theta)=
r(\tfrac{2M}{(M-2)r\pi})^2\sin^2\left(\tfrac{(M-2)r\pi}{2M}\right),&&\mbox{for even~$M\ge4$}.
\end{align}
\end{subequations}

Since~$\{\tfrac{(M-1)r\pi}{2M},\tfrac{(M-2)r\pi}{2M}\}\subset (0,\tfrac{r\pi}{2})\subset (0,\tfrac{\pi}{2})$, we find that  the sequence~$(\vartheta_M)_{M\in\bbN_+}$ as
\begin{align}
\vartheta_M\coloneqq\min{\rm Eig}(\Theta)\quad\mbox{satisfies}\quad \vartheta_M>\vartheta_\infty\coloneqq
r(\tfrac{2}{r\pi})^2\sin^2\left(\tfrac{r\pi}{2}\right)=\lim_{M\to+\infty}\vartheta_M>0.
\end{align}

Recalling~\eqref{normOP}, for~$d=1$ we find~$\norm{P_{\clE_{M}^\perp}^{\clU_{M}}}{\clL(H)}^2=\vartheta_M^{-1}<\vartheta_\infty^{-1}<+\infty$. For general~$d$, from~\cite[Lems.~4.3 and~4.4]{KunRod19-cocv} we obtain~$\norm{P_{\clE_{M}^\perp}^{\clU_{M}}}{\clL(H)}^2=\vartheta_M^{-d}<\vartheta_\infty^{-d}<+\infty$, for  actuators and eigenfunctions given by the Cartesian products in section~\ref{sS:acteigf}. This finishes the proof.
\end{proof}

\subsection{Proof of Theorem~\ref{T:Main}} We simply observe that Theorem~\ref{T:Main} is a consequence of~\cite[Thm.~4.1]{Rod20-eect}. Indeed, 
due to Lemmas~\ref{LA:A0sp}--\ref{LA:NN} and~\ref{LA:FFM}--\ref{LA:bddratP=suffalpha} above, and due to~\cite[Cor.3.9]{Rod20-eect}, we have that the Assumptions~3.1--3.4 and~3.6--3.8 in~\cite{Rod20-eect} are satisfied. Therefore,\cite[Thm.~4.1]{Rod20-eect} holds true as well as its corollary~\cite[Thm.~2.5]{Rod20-eect} (cf.~\cite[Sect.~4.5]{Rod20-eect}), from which we conclude the result stated in Theorem~\ref{T:Main}.
\qed

\section{Stabilizability for the fluid flow model}\label{S:stabil-fluid}
We follow a strategy as in section~\ref{S:stabil-fire}. We write
system~\eqref{sys-y-KS-Intro-Bu-fluid} with the feedback control~\eqref{FeedKy-N-fluid} as an abstract parabolic-like evolution equation satisfying the assumptions in~\cite{Rod20-eect}, and apply the results in~\cite{Rod20-eect}. Again, $\Omega$ is a rectangular domain~$\Omega=\bigtimes\limits_{i=1}^d(0,L_i)$, and we consider bi-Dirichlet, bi-Neumann, and periodic boundary conditions.

\subsection{Kuramoto--Sivashinsky equation in abstract form}\label{sS:KS-abstract-fluid}
While the state~$y$ of the flame propagation model, investigated in section~\ref{S:stabil-fire}, is given by a scalar function, $y(x,t)\in\bbR$, the state~$\bfy$ of the fluid flow model is given by a vector function, $\bfy(x,t)\in\bbR^d$.
Concerning the mathematical setting we simply take Cartesian products of the spaces and linear operators involved in section~\ref{S:stabil-fire}. Namely, as pivot space, $\bfH=\bfH'$, we consider the Cartesian product space~$\bfH\coloneqq \bigtimes\limits_{k=1}^d H$ . We also take the diffusion-like operator
\begin{subequations}\label{absOper-fluid}
\begin{equation}
\bfA\bfy=\bfA(\bfy_1,...,\bfy_d)\coloneqq(A\bfy_1,...,A\bfy_d),\qquad \bfy\in\bfV=\bigtimes\limits_{k=1}^d V\label{bfA-A}
\end{equation}
with
\[
\langle \bfA \bfz,\bfw\rangle_{\bfV',\bfV}\coloneqq (\bfz,\bfw)_\bfV\coloneqq (\bfz_1,\bfw_1)_V+...+(\bfz_d,\bfw_d)_V.
\]
The domain of the operator~$\bfA\in\clL(\bfV,\bfV')$
is given by~$\rmD(\bfA)=\bigtimes\limits_{k=1}^d\rmD(A)$ and we have that
\begin{equation}\notag%
 \bfA\bfy=    \nu_2\Delta^2\bfy-2  \nu_2\Delta\bfy+  \nu_2\bfy\quad\mbox{for}\quad \bfy\in\rmD(\bfA).
\end{equation}
We further set the state operators
\begin{align}
   \bfA_{\rm rc}\bfy&\coloneqq (\nu_1+2\nu_2)\Delta\bfy-  \nu_2\bfy+\nu_0\langle\widehat \bfy\cdot\nabla\rangle \bfy+\nu_0\langle \bfy\cdot\nabla\rangle\widehat \bfy,\\
\bfN(\bfy)&\coloneqq\nu_0\langle \bfy\cdot\nabla\rangle\widehat \bfy,
\end{align}
and the  feedback operator~\eqref{FeedKy-N-fluid}, which now reads
\begin{align}
 y\mapsto\bfK_{M}^{\lambda}(t,y)
 &= P_{\underline\clU_{M}}^{\underline\clE_{M}^\perp}
 \left(   \nu_1\Delta \bfy
 +\bfA_{\rm rc}\bfy+\bfN(\bfy)-\lambda \bfy\right)
 \label{FeedKy-N-abs0-fluid}\\
&= P_{\underline\clU_{M}}^{\underline\clE_{M}^\perp}
 \left(   \bfA \bfy
 +\bfA_{\rm rc}\bfy+\bfN(\bfy)-(\nu_2\Delta^2+\lambda\Id) \bfy\right)\label{FeedKy-N-abs-fluid}
\end{align}
\end{subequations}
and we can write~\eqref{sys-y-KS-Intro-Bu-fluid} as the evolutionary parabolic-like equation
\begin{align}\label{sys-y-KS-Bu-fluid}
 &\dot{\bfy}+\bfA  \bfy+ \bfA_{\rm rc} \bfy
 +\bfN(\bfy)
  =\bfK_{M}^{\lambda}(t,\bfy),\qquad
  \bfy(0)= \bfy_0.
\end{align}

\subsection{Actuators and auxiliary eigenfunctions}\label{sS:acteigf-fluid}
As actuators we take indicator functions of small subdomains as in section~\ref{sS:acteigf-fluid} in each coordinate of the vector state~$\bfy$, as already mentioned in section~\ref{sS:intro-main-res}. That is, the space spanned by the actuators is
\begin{equation}\label{clUM-explicit-fluid}
\underline\clU_{M}\coloneqq\bigtimes_{k=1}^d \clU_{M,k}
\end{equation}
where~$\clU_{M,k}=\clU_{M}$ is as in~\eqref{UM-explicit}, for each coordinate index~$k$. Analogously,
we take the auxiliary functions~$\clE_{M,k}=\clE_{M}$ as in~\eqref{EM-explicit}, for each coordinate index~$k$. The space spanned by such eigenfunctions is  
\begin{equation}\label{EM-explicit-fluid}
\underline\clE_{M}\coloneqq \bigtimes_{k=1}^d\clE_{M,k}.
\end{equation}

\subsection{The main result}

Concerning the fluid flow model, the main stabilizability result of this manuscript is the following.
\begin{theorem}\label{T:Main-fluid}
System~\eqref{sys-y-KS-Bu-fluid} with the state and feedback operators in~\eqref{absOper-fluid} is semiglobally exponentially stable. More precisely, for arbitrary given~$\lambda>0$ and~$R>0$, there exists~$M\in\bbN_+$, $0<\mu\le\lambda$, and~$C\ge1$ such that the solution satisfies
\[
\norm{\bfy(t)}{V} \le C\rme^{-\mu t} \norm{\bfy_0}{V},\quad\mbox{for all}\quad t\ge0 \quad\mbox{and all}\quad \bfy_0\in \bfV\quad\mbox{with}\quad \norm{\bfy_0}{\bfV}<R.
\]
\end{theorem}

\subsection{On the state operators}
Again, we will need the following boundedness assumption on the targeted trajectory.
\begin{assumption}\label{A:haty-fluid}
We have that~$\norm{\widehat \bfy}{L^\infty(\bbR_+,(W^{1,2}(\Omega))^d)}\le C_{\widehat y}<+\infty$.
\end{assumption}

We need to show the analogs of Lemmas~\ref{LA:A0sp}--\ref{LA:NN} , for the state operators~$\bfA$, $\bfA_{\rm rc}$, $\bfN$.

For simplicity we do not ``repeat'' here the statements of those analog Lemmas. We will simply sketch the main reason (why) they hold true.

To see that the analogs of Lemmas~\ref{LA:A0sp} and~\ref{LA:A0cdc} are a consequence of those Lemmas themselves it is enough to observe that~$\bfA$ acts coordinatewise through~$A$; see~\eqref{bfA-A}.

To show the analog of Lemma~\ref{LA:A1}, we observe that
\[
\norm{\bfA_{\rm rc}(t)\bfy}{\bfH}\le C_0 \left(\norm{\Delta \bfy}{\bfH}+\norm{\bfy}{\bfH}+\norm{\langle\bfy\cdot\nabla\rangle \widehat \bfy+\langle\widehat \bfy\cdot\nabla\rangle \bfy}{\bfH}\right)\le C_1(1+C_{\widehat \bfy})\norm{\bfy}{W^{2,2}(\Omega)^d},
\]
with $C_{\widehat \bfy}$ as in Assumption~\ref{A:haty-fluid}, which we can obtain by using the fact that
\[
\norm{\langle\bfy\cdot\nabla\rangle \widehat \bfy+\langle\widehat \bfy\cdot\nabla\rangle \bfy}{\bfH}\le C_2\left(\norm{\bfy}{L^\infty(\Omega)^d}\norm{\widehat \bfy}{W^{1,2}(\Omega)^d}+\norm{\widehat \bfy}{L^3(\Omega)^d}\norm{\nabla \bfy}{L^6(\Omega)^{d^2}}\right)
\]
and then, for $1\le d\le 3$, due to suitable Sobolev embeddings and to the Agmon inequality, it follows the analog of Lemma~\ref{LA:A1}, by also recalling that~$\bfV\xhookrightarrow{}W^{2,2}(\Omega)^d$.

To show the analog of Lemma~\ref{LA:NN} we note that
with~$\bfz\coloneqq \bfy^{[1]}-\bfy^{[2]}$, we have
\begin{align}
\bfN(\bfy^{[1]})-\bfN(\bfy^{[2]})=\nu_0\langle\bfy^{[1]}\cdot\nabla\rangle \bfz+\nu_0\langle \bfz\cdot\nabla\rangle \bfy^{[2]}\notag
\end{align}
and obtain, recalling that $1\le d\le 3$,
\begin{align}
\norm{\bfN(\bfy^{[1]})-\bfN(\bfy^{[2]})}{H}
&\le C_1\norm{\bfy^{[1]}}{L^3}\norm{\nabla \bfz}{L^6}+C_1\norm{\bfz}{L^\infty}\norm{\nabla\bfy^{[2]}}{L^2}\notag\\
&\le C_2\left(\norm{\bfy^{[1]}}{\bfV}+\norm{\bfy^{[2]}}{\bfV}\right)\norm{\bfz}{\bfV}.\notag
\end{align}
Thus, the Lemma follows with~$n=1$, $C_\bfN=C_2$, and~$(\zeta_{1j},\zeta_{2j},\delta_{1j},\delta_{2j})=(1,0,1,0)$.

\subsection{On the feedback triple}\label{sS:feedtriple}
We show that the analogs of Lemmas~\ref{LA:FFM}--\ref{LA:bddratP=suffalpha} hold true.
We start by defining the mapping
\[
\bfF_{M}\colon \underline\clE_{M}\to \underline\clE_{M},\quad\mbox{with}\quad\bfF_{M}(q)\coloneqq \nu_2\Delta^2 q+\lambda q.
 \]
To see that the analog of Lemma~\ref{LA:FFM} holds true is enough to observe that~$\bfF_{M}$ acts coordinatewise, $\bfF_{M}(q)=\bfF_{M}(q_1,\dots,q_d)=( \nu_2\Delta^2 q_1+\lambda q_1,\dots,\nu_2\Delta^2 q_d+\lambda q_d)$.

Next, to show the analog of Lemma~\ref{LA:DS}, we observe that giving $\bfh=(\bfh_1,\dots,\bfh_d)\in\bfH$ we can use Lemma~\ref{LA:DS} to write for each coordinate index~$k\in\{1,\dots,d\}$, $\bfh_k=\bfh_k^U+\bfh_k^E$, with~$(\bfh_k^U,\bfh_k^E)\in\clU_{M,k}\times \clE_{M,k}^\perp$. Hence
 $\bfh=(\bfh_1^U,\dots,\bfh_d^U)+(\bfh_1^E,\dots,\bfh_d^E)$, with $(\bfh_1^U,\dots,\bfh_d^U)\in\underline\clU_{M}$ and $(\bfh_1^E,\dots,\bfh_d^E)\in\underline\clE_{M}^\perp$. Therefore, we can conclude that~$\bfH=\underline\clU_{M}+\underline  \clE_{M}^\perp$. It remains to observe that~$\underline\clU_{M}\bigcap\underline  \clE_{M}^\perp=\{0\}$ which is a consequence of the identities~$\clU_{M,k}\bigcap \clE_{M,k}^\perp=\{0\}$ in each coordinate index~$k$.

Finally we focus on the analog of Lemma~\ref{LA:bddratP=suffalpha}.
Let us denote, the analogs of~\eqref{VDAalphaM} as
 \begin{align}
\overline \alpha_{M}&\coloneqq \max\{\alpha_i\mid \mbox{there is }\phi\in\underline\clE_{M}
 \mbox{ such that } \bfA\phi=\alpha_i\phi\},\notag\\
 \overline\alpha_{M_{+}}&\coloneqq \min\{\alpha_i\mid \mbox{there is }\phi\in\underline\clE_{M}^\perp
 \mbox{ such that } \bfA\phi=\alpha_i\phi\}. \notag
\end{align}
Reasoning coordinatewise, we can see that~$\overline \alpha_{M}=\max\{\overline \alpha_{M,k}\mid 1\le k\le d\}$ and~$\overline \alpha_{M_+}=\min\{\overline \alpha_{M_+,k}\mid 1\le k\le d\}$ where the~$\overline \alpha_{M,k}$s and~$\overline \alpha_{M_+,k}$s are as in~\eqref{VDAalphaM}, for each coordinate~$k$. Hence, $\sup\limits_{M\in\bbN_+}\frac{\overline\alpha_{M}}{\overline\alpha_{M_{+}}}<+\infty$.
It remains to show that~
$\sup\limits_{M\in\bbN_+}\norm{P_{\underline\clE_{M}^\perp}^{\underline\clU_{M}}}{\clL(H)}<+\infty$.
For that purpose, we recall that we can write every~$\bfh=(\bfh_1,\dots,\bfh_d)\in\bfH$ in a unique way as $\bfh=(\bfh_1^U,\dots,\bfh_d^U)+(\bfh_1^E,\dots,\bfh_d^E)$, with $(\bfh_1^U,\dots,\bfh_d^U)\in\underline\clU_{M}$ and $(\bfh_1^E,\dots,\bfh_d^E)\in\underline\clE_{M}^\perp$. Thus,  $P_{\underline\clE_{M}^\perp}^{\underline\clU_{M}}\bfh=(\bfh_1^E,\dots,\bfh_d^E)$ and $\norm{P_{\underline\clE_{M}^\perp}^{\underline\clU_{M}}\bfh}{\bfH}^2\le
\max\left\{ \norm{P_{\clE_{M,k}^\perp}^{\clU_{M,k}}}{L^2(0,L_k)}^2\mid 1\le k\le d\right\}\norm{\bfh}{\bfH}^2$. By Lemma~\ref{LA:bddratP=suffalpha} we have that
$\norm{P_{\clE_{M,k}^\perp}^{\clU_{M,k}}}{L^2(0,L_k)}^2<+\infty$, thus the analog of Lemma~\ref{LA:bddratP=suffalpha} holds true.
As a remark, we mention that $\norm{P_{\clE_{M,k}^\perp}^{\clU_{M,k}}}{L^2(0,L_k)}$ is independent of~$k$, since it is independent of~$L_k$. This can be concluded from~\eqref{minEigTheta_exp} and~\eqref{normOP} (cf.~\cite[discussion at end of sect.~4.6]{KunRod19-cocv}).

\section{Numerical simulations}\label{S:simul}
Here we include the results of numerical simulations showing the stabilizing performance of the feedback control to targeted trajectories. For simplicity, we restrict ourselves to the one-dimensional case under periodic boundary conditions.

For the spatial discretization we use spectral elements and compute the solution of spectral Galerkin approximations based on ``the'' first ~$N$ eigenfunctions. For the temporal discretization we use the implicit Crank--Nicolson scheme for the linear operator and an explicit Adams--Bashford extrapolation for the nonlinearity and for the feedback control.

For the free dynamics, in an toy example, we also check the convergence of the numerical solution to an exact spatio-temporal dependent solution.

The spatial domain is the one-dimensional unit Torus~$\Omega=\bbT^1_1=\frac{1}{2\pi}\bbT^1\sim[0,1)\subset\bbR$ (i.e., we consider periodic boundary conditions). 

\subsection{Discretization}
We briefly mention the way we have discretized the equations. We present a discretization which allow us to deal in a direct straightforward way with the oblique projection feedback control. Our main goal is to confirm the stabilizing performance of the feedback control, rather than a comparison to other potential discretizations.

Let us consider simultaneously the free dynamics and the controlled systems, namely, \eqref{sys-haty-KS-Intro} and~\eqref{sys-tildey-KS-Intro-Bu} for the flame propagation model,  or \eqref{sys-haty-KS-Intro-fluid} and~\eqref{sys-tildey-KS-Intro-Bu-fluid} for the fluid flow model. 

Essentially, we seek for approximations~$z^N$ of the states of those systems as~$z^N(t)=\sum\limits_{n=1}^N z^N_n(t)e_n\in\clE_N$. That is, $z^N(t)$ lives in the linear span of ``the'' first~$N$ (periodic) eigenfunctions $e_n$, $1\le n\le N$, of the Laplacian.  We compute~$z^N$ by solving Galerkin approximations as follows (cf.~\cite[sect.~3.4]{Rod21-jnls}, \cite[sect.~4.3]{Rod20-eect}, \cite[Ch.~3, sect.~3.2]{Temam01}),
\begin{subequations}\label{galerkin}
\begin{align}
 &\dot{z}^N+ \nu_2\Delta^2 z^N+ \nu_1\Delta z^N+\nu_0 P_{\clE_N}\clN(z^N)=C_{\rm feed}P_{\clE_N}K(t,z^N)+P_{\clE_N}f,\notag\\
 &z^N_n(0,\Bigcdot)=P_{\clE_N}z_0,\notag
\end{align}
\end{subequations}
where~$C_{\rm feed}=0$ for the free dynamics and~$C_{\rm feed}=1$ for the controlled dynamics. Here, in dimension one,
\begin{align}
&\clN(z^N)=\tfrac12\norm{\p_x z^N}{\bbR}^2,\quad&&\mbox{for the flame propagation model~\eqref{sys-haty-KS-Intro}};\notag\\
&\clN(z^N)=z^N\p_x z^N,\quad&&\mbox{for the fluid flow model~\eqref{sys-haty-KS-Intro-fluid}}.
\end{align}

Recall that $P_{\clE_N}$ stands for the orthogonal projection in the pivot space~$H=L^2(\Omega)$ onto~$\clE_N$, and also that both~$\Delta^2$ and~$\Delta$ map~$\clE_N$ into itself. Essentially, we have to solve a system of~$N$ ordinary differential equations (one equation for each spectral coordinate index~$n$) as 
\begin{subequations}\label{galerkin-coord}
\begin{align}
 &\dot{z}^N_n=- (\alpha_n^2-\alpha_n)z^N_n-\left(P_{\clE_N}\clN(z^N)\right)_n+C_{\rm feed}\left(P_{\clE_N}K(t,z^N)\right)_n+f_n,\\
 &z^N_n(0,\Bigcdot)=z_{0,n}.
\end{align}
\end{subequations}

Thus, hereafter our computed solutions~$\widehat z$ and~$\widetilde z$ are in fact  linear combinations  of the first eigenfunctions in the spatial interval~$[0,1)$ with coordinates
given as in~\eqref{galerkin-coord}.

We compute the orthogonal projections~$P_{\clE_N}h$  above by firstly evaluating $h$ in the nodes of a (regular) mesh/partition, for a space-step~$0<x^{\rm step}<1$,
\begin{equation}\label{disc-spInt}
[0,1)_{\rm disc}= \{nx^{\rm step}\mid 0\le n\le\tfrac{1}{x^{\rm step}}\},\qquad\tfrac{1}{x^{\rm step}}\in\bbN_+,
\end{equation}
 of the spatial interval~$[0,1)$ and use the associated finite-element mass matrix, corresponding to periodic boundary conditions, as an auxiliary tool, to compute the coordinates of~$P_{\clE_N}h$ following~\cite[sect.~8.1]{RodSturm20}. 
The finite-element basis vectors were taken as the classical hat-functions (piecewise-linear elements).

Afterwards, as temporal discretization we follow an implicit-explicit scheme, namely, an (implicit) Crank--Nicolson scheme for the affine component~$-(\alpha_n^2-\alpha_n)z^N_n+f_n$ and an (explicit) Adams-Bashford extrapolation for the nonlinear component~$-\nu_0\left(P_{\clE_N}\clN(z^N)\right)_n+C_{\rm feed}\left(P_{\clE_N}K(t,z^N)\right)_n$.
The temporal step~$0<t^{\rm step}$ was taken uniform,
\begin{equation}\label{disc-tInt}
[0,+\infty)_{\rm disc}= \{nt^{\rm step}\mid n\in\bbN\}.
\end{equation}

\begin{remark}
Essentially, the auxiliary finite-element mass matrix is used to compute the orthogonal projection of the nonlinearity onto the Galerkin space and also to compute the oblique projection based feedback control. This strategy is not restricted to any particular case of boundary conditions.
However, here we should mention that, for the particular case of periodic boundary conditions the use of the (discrete) Fourier transform (computed through the Fast Fourier Transform algorithm) is often used to solve the free dynamics equations on the so-called frequency space (at least for the fluid flow model). On the other hand, we do not have at our disposal a discretization of the  feedback control operator in the mentioned frequency space. For that we would need a discretization on the frequency space of the oblique projection~$P_{\clU_M}^{\clE_M^\perp}$. Working on the frequency space, avoiding the finite-element auxiliary computations,  will likely make the computations faster, but such approach is restricted to periodic boundary conditions (or, at least the authors do not know how to use it for other boundary conditions). On the other hand, we recall that our main goal is the confirmation of the stabilizing performance of the feedback control rather than speeding up the solver. 
\end{remark}
\begin{remark}
Though, for simplicity, we restrict the simulations for periodic boundary conditions on a one-dimensional interval, our discretization can be modified for higher-dimensional rectangles under the same periodic boundary conditions, for that we have  just  to replace the respective eigenfunctions and the ``periodic'' finite-element mass matrix.
Note also that for periodic boundary conditions, we actually have no boundary conditions, because we can see the evolution in the Torus which is a spatial boundaryless domain. However, our discretization can be modified
for other (actual) homogeneous  boundary conditions as bi-Dirichlet and bi-Neumann (i.e., with $g=0$ in system~\eqref{sys-haty-KS-Intro}). Again, we just need to replace the respective eigenfunctions and the finite-element mass matrix by the ones corresponding to the given homogeneous boundary conditions.
Finally, for nonhomogeneous bi-Dirichlet and bi-Neumann boundary conditions we can, for example, reduce the problem to the case of homogeneous boundary conditions if we know a function~$G$ with trace~$\clG G\rest{\p\Omega}=g$, by using the change of variables~$\breve y=y-G$ and solving for~$\breve y$, which will solve a similar system under homogeneous boundary conditions, $\clG \breve y\rest{\p\Omega}=0$. We can also find such function~$G$ numericaly (e.g., as the solution of the heat equation~$\frac{\p}{\p t} G-\nu_1\Delta G=0$ with~$G(0)=0$ and~$\clG G\rest{\p\Omega}=g$).
\end{remark}

 \subsection{Numerical results for the fluid flow model}

 We choose the parameters of the free dynamics as
 \begin{subequations}\label{param-fluid}
\begin{align}
\nu_2=10^{-6},\quad \nu_1=10^{-2},\quad \nu_0=1,\quad f=0 ,
\end{align}
the control parameters as
\begin{align}
r=0.2,\quad M=35,\quad \lambda=10,
\end{align}
and the initial states for targeted~$\widehat y$ and controlled~$\widetilde y$ trajectories as
\begin{align}
\widehat y_0=1+\norm{\sin(4\pi x)}{\bbR},\qquad \widetilde y_0=\cos(2\pi x)(1+\sin(2\pi x)).
\end{align}
 We solve, for time~$t\in[0,T]$, the $N$-dimensional Galerkin approximation, with
 \begin{align}
 T=1.5,\quad N=200.
 \end{align}
 The temporal and spatial time steps, in~\eqref{disc-spInt} and~\eqref{disc-tInt} were taken as
  \begin{align}
 t^{\rm step}=10^{-4},\quad  x^{\rm step}=10^{-4}.
 \end{align}
 
 \end{subequations}
 
 With the parameters~$\nu_2$ and $\nu_1$ above, we can see that the Kuramoto--Sivashinsky linear operator~$-\nu_2\Delta^2-\nu_1\Delta$ has $31$ nonnegative (unstabe) eigenvalues~$\alpha_n$. Indeed, under periodic boundary conditions, the ordered eigenvalues~$\alpha_{2j-1}<\alpha_{2j}=\alpha_{2j+1}$ (repeated accordingly to their multiplicity), $j\in\bbN_+$, satisfy
 \begin{align}
 &\alpha_{30}=\alpha_{31}=-\nu_2(30\pi)^4+\nu_1(30\pi)^2\approx +9.9251,\notag\\
  &\alpha_{32}=\alpha_{33}=-\nu_2(32\pi)^4+\nu_1(32\pi)^2\approx-1.0761\notag.
  \end{align}
  That is why we took a number of actuators~$M$ slightly larger as above in~\eqref{param-fluid}.

 In Fig.~\ref{Fig:ff-free} we plot the targeted trajectory and the uncontrolled trajectory (issued from different initial states). The figure shows that it is likely that the latter does not converge to the former as time increases.
In Fig.~\ref{Fig:ff-feed} we see that the latter converges to the former  if we use our feedback control. 

 \begin{figure}[ht]
\centering
\subfigure[targeted trajectory, $\widehat y(0,x)=\widehat y_0$.\label{Fig:ff-freehat}]
{\includegraphics[width=0.45\textwidth,height=0.31\textwidth]{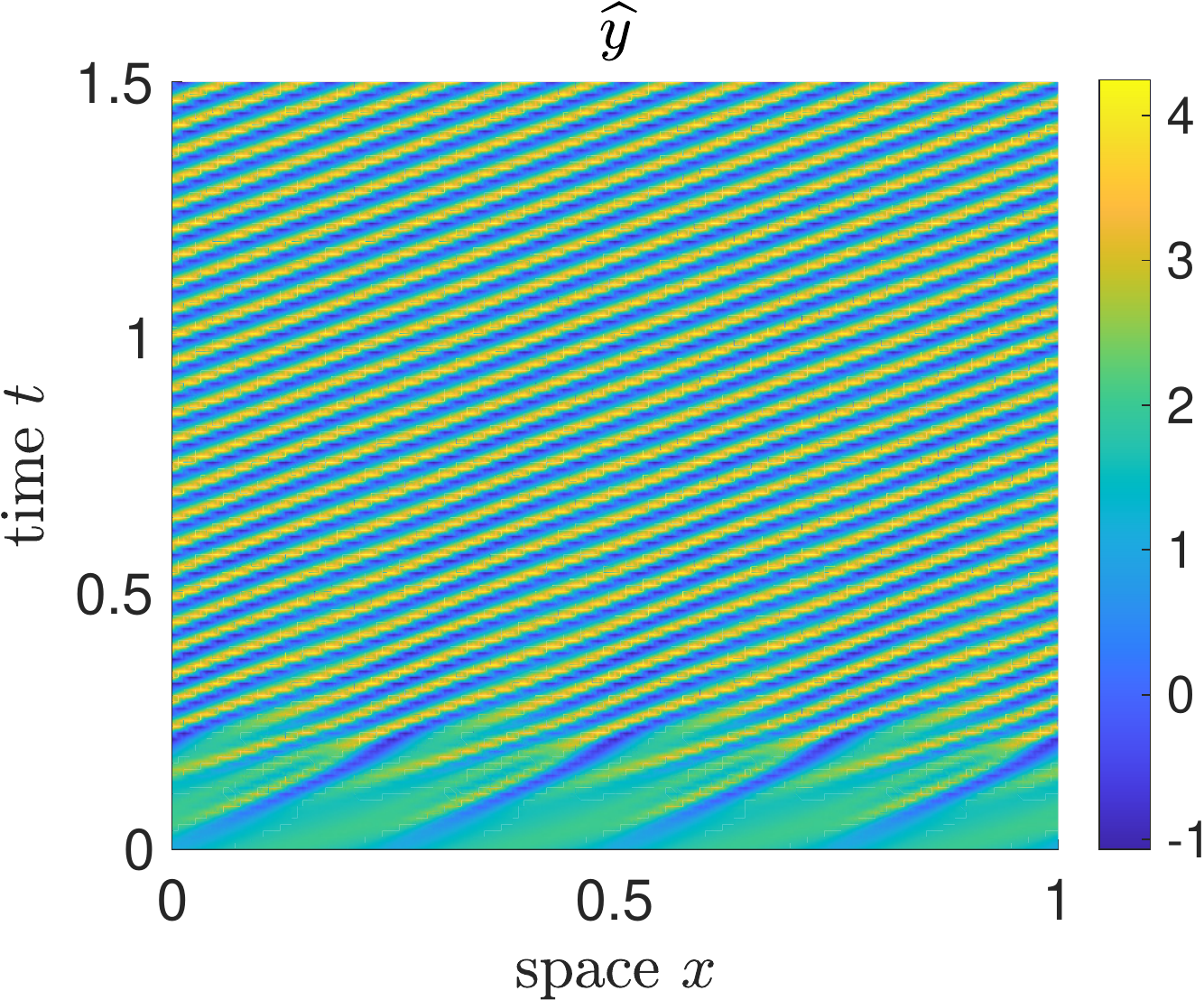}}
\qquad
\subfigure[uncontrolled trajectory, $\widetilde y_{\rm unc}(0,x)=\widetilde y_{0}$.\label{Fig:ff-freetil}]
{\includegraphics[width=0.45\textwidth,height=0.31\textwidth]{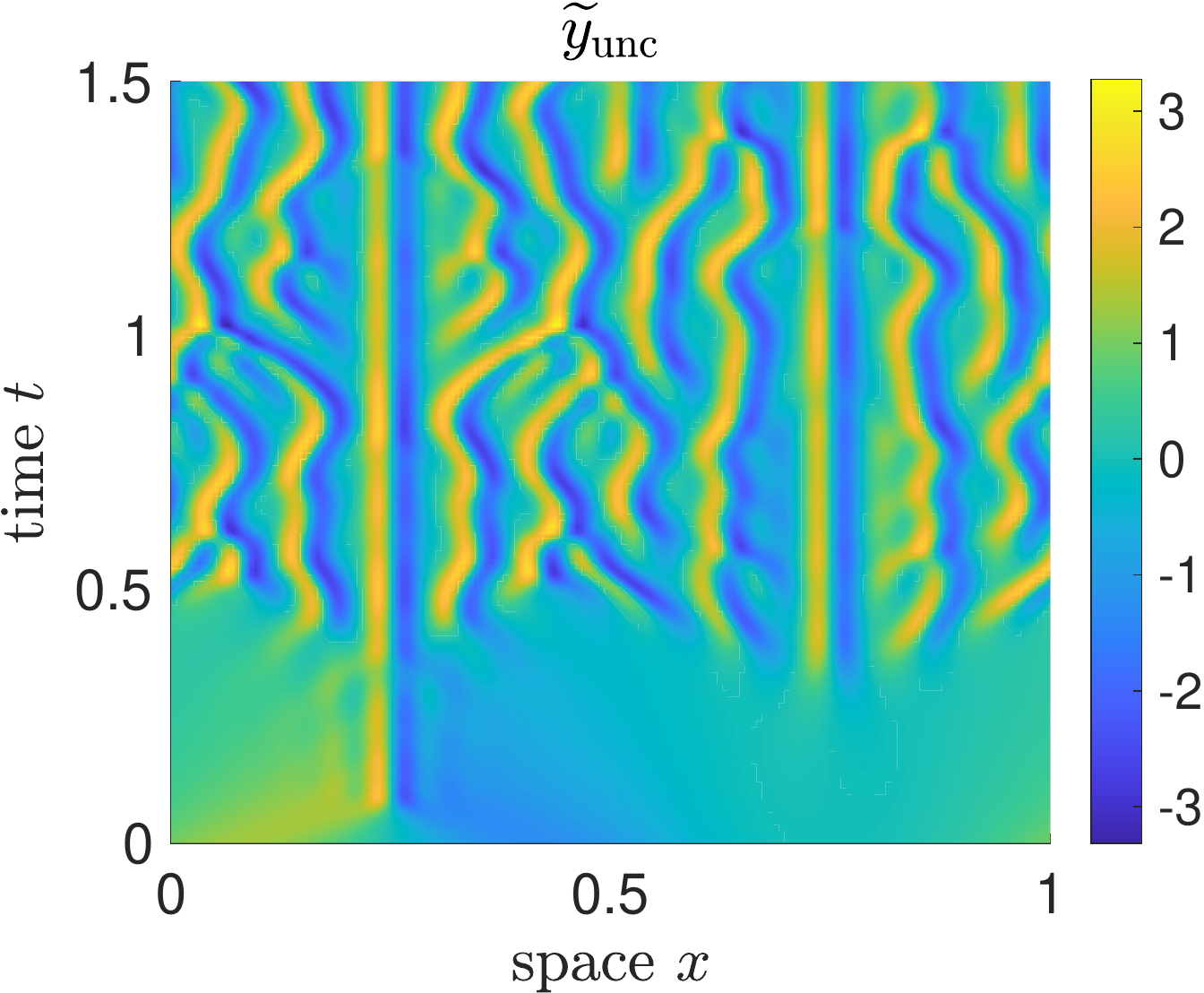}}
\caption{Fluid  model: free dynamics.}
\label{Fig:ff-free}
\end{figure}

 \begin{figure}[ht]
\centering
\subfigure[controlled trajectory, $\widetilde y(0,x)=\widetilde y_{0}$.]
{\includegraphics[width=0.45\textwidth,height=0.31\textwidth]{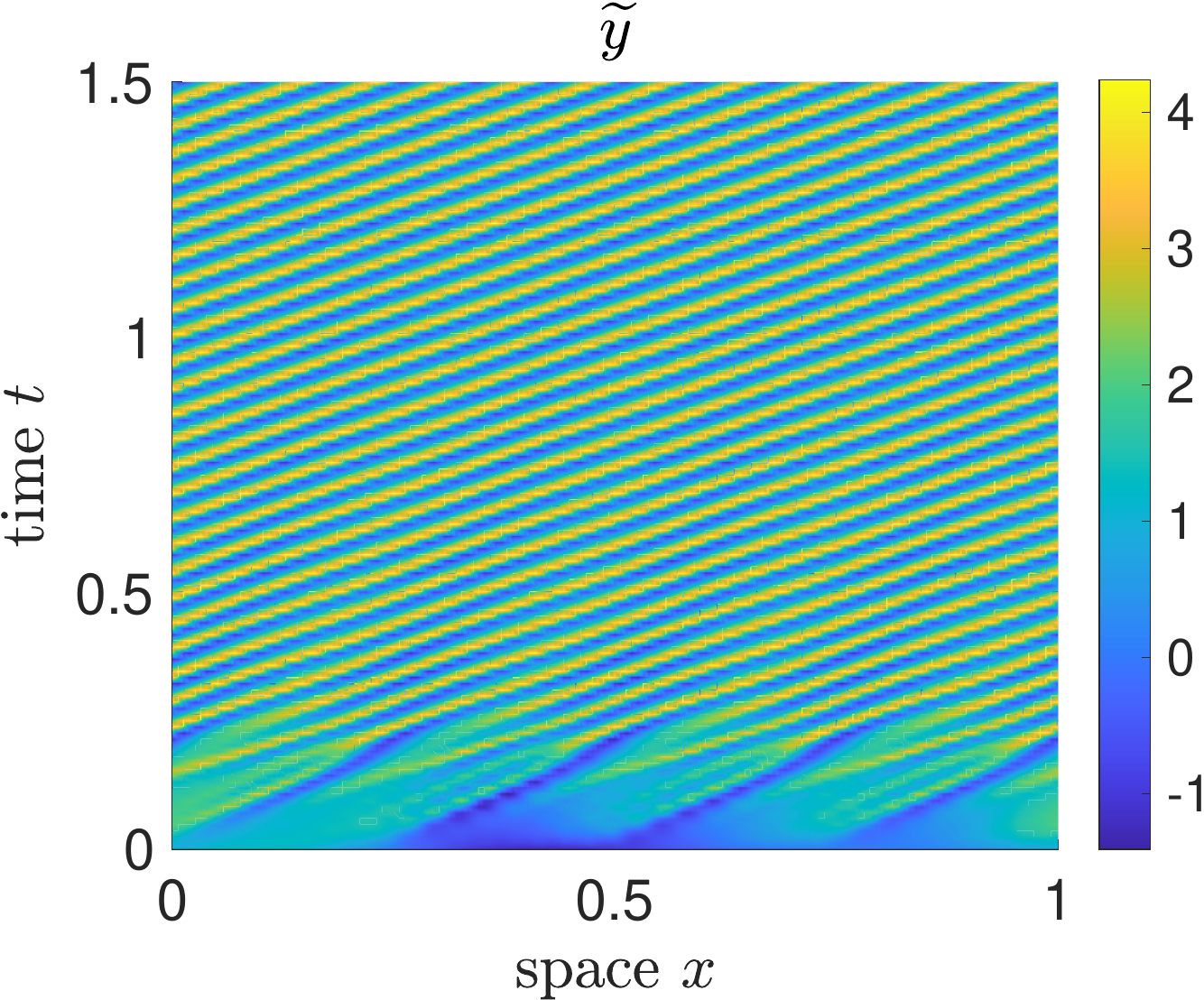}}
\qquad
\subfigure[evolution of distance to target.]
{\includegraphics[width=0.45\textwidth,height=0.31\textwidth]{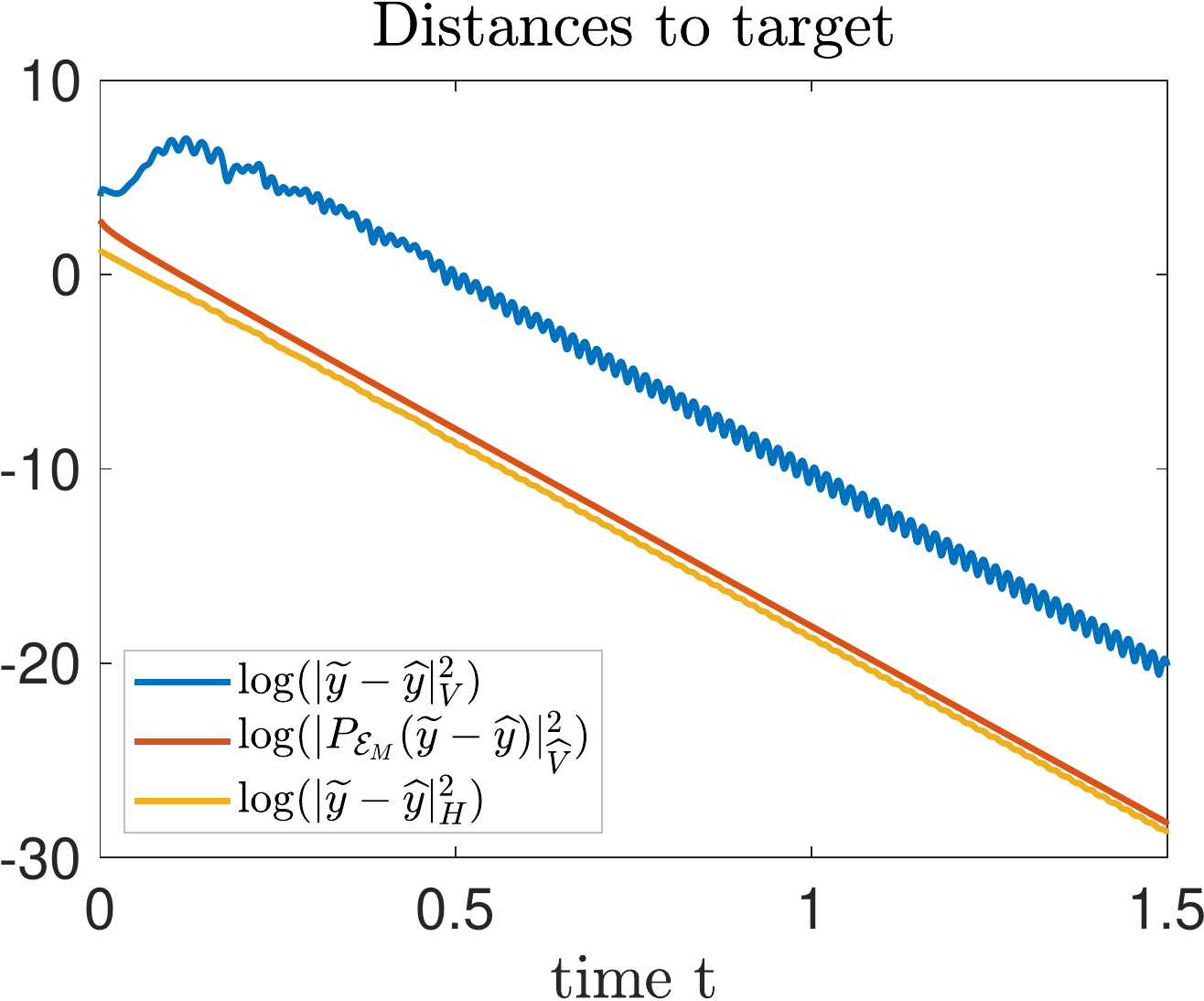}}
\caption{Fluid model: controlled dynamics.}
\label{Fig:ff-feed}
\end{figure}

Finally in Fig.~\ref{Fig:ff-magu} we plot the coordinates of the feedback control. The $M=M_\sigma$ actuators, whose total support covers~$r100\%=20\%$ of the spatial domain, were taken with supports as in~\eqref{omeM1D}--\eqref{cM}, with~$L_n=1$, $n=1$. Such actuators are illustrated in Fig.~\ref{Fig:ff-act}.
 \begin{figure}[ht]
\centering
\subfigure[signed magnitudes of the coordinates~$u_j$.\label{Fig:ff-magu}]
{\includegraphics[width=0.45\textwidth,height=0.31\textwidth]{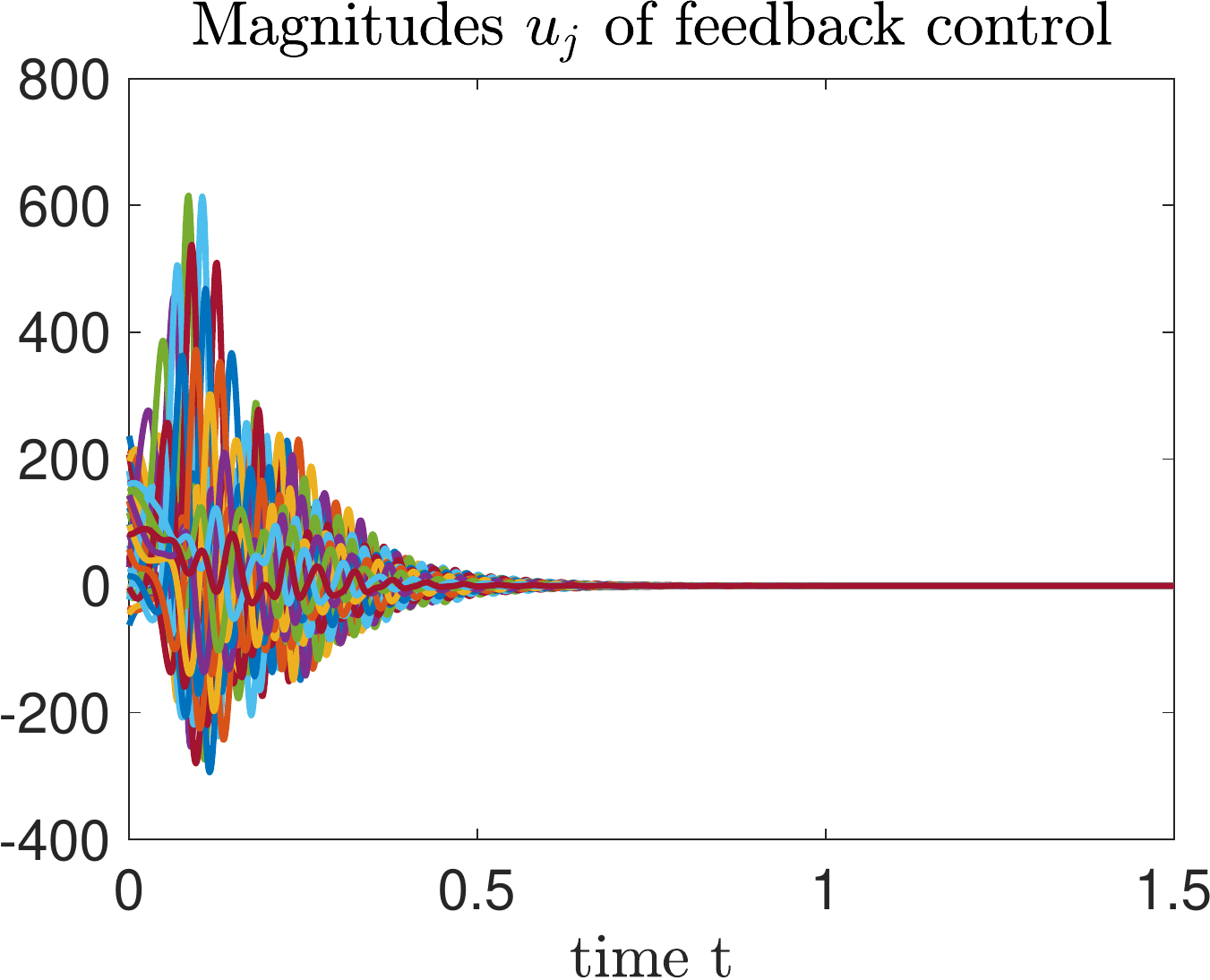}}
\qquad
\subfigure[a linear combination of the actuators.\label{Fig:ff-act}]
{\includegraphics[width=0.45\textwidth,height=0.31\textwidth]{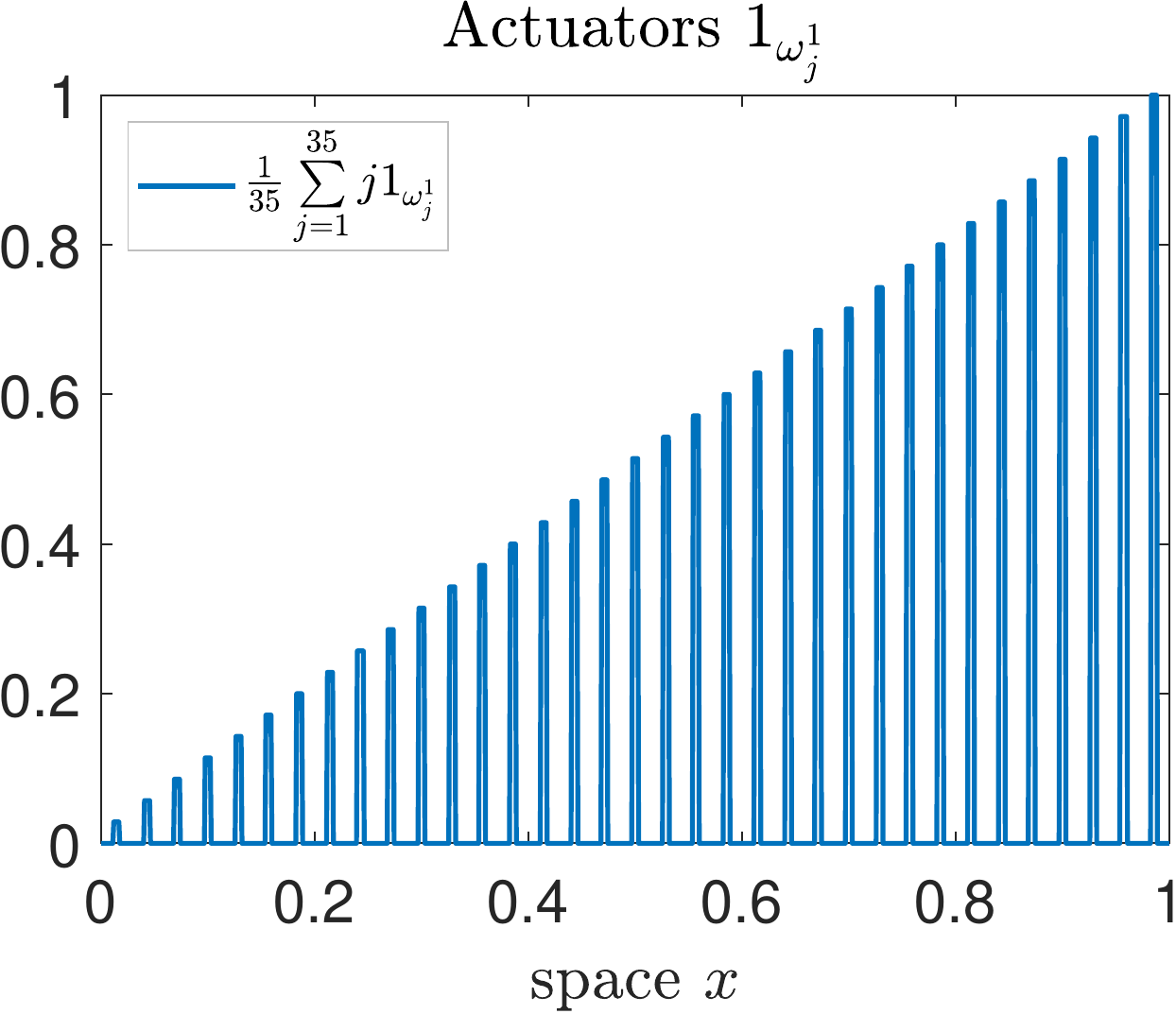}}
\caption{Fluid model: feedback control.}
\label{Fig:ff-magu-act}
\end{figure}

\begin{remark}
Solving the equation for time~$t\in[0,1.5]$ with the parameters~$(\nu_2,\nu_1,\nu_0)$ as in~\eqref{param-fluid} is equivalent (after a rescaling argument in both space and time) to solve the system
\begin{equation}\notag
\tfrac{\p}{\p t}{z}+ \Delta^2z+ \Delta z+z\tfrac{\p}{\p x}z=0
\end{equation}
in the spatial interval~$[0,100)$ with periodic boundary condition (with period $100$), for time $t\in[0,150]$ (and with time step~$100t^{\rm step}=10^{-2}$).
The latter system was solved in~\cite[Sect.~4]{KassamTrefethen05}, in the frequency space, using a scheme based on the Fast Fourier Algorithm  for a slightly different period, namely~$L=32\pi\approx100.53$,
with the initial condition~$z_0=\cos(\frac{2\pi x}{L})(1+\sin(\frac{2\pi x}{L}))$. We observe that in Fig.~\ref{Fig:ff-freetil} we obtain an analogous qualitative behavior to the one in~\cite[Sect.~4, Fig.~6]{KassamTrefethen05}. This is a first sign that we can trust our solver mixing spectral elements with an auxiliary finite elements mass matrix, which we use, in particular to compute the feedback control, the stabilizing performance of which is the main focus of this manuscript. However, for the reader more interested in the free dynamics (restricted to periodic boundary conditions)  the scheme~\cite{KassamTrefethen05} is faster (there the authors used also a different temporal scheme, which allowed them to take a larger time-step~$0.25>10^{-2}$).
\end{remark}

\subsection{Numerical results for the flame propagation model}
Here we choose all the parameters as in~\eqref{param-fluid}, except~$\nu_0$, which we choose as
\begin{equation}\label{nu0fire}
\nu_0=10^{-2}.
\end{equation}

\begin{remark}
We take a different~$\nu_0$ to have, again, a setting  equivalent to solve 
\begin{equation}\notag
\tfrac{\p}{\p t}{z}+ \Delta^2z+ \Delta z+\tfrac12\norm{\tfrac{\p}{\p x}z}{\bbR}^2=0
\end{equation}
in the spatial interval~$[0,100)$ with periodic boundary conditions, for time $t\in[0,150]$.
\end{remark}

Again, in Fig.~\ref{Fig:fp-free} we see that without control it is likely that the trajectory~$\widetilde y_{\rm unc}$ will not converge to the targeted trajectory~$\widehat y$. Then, in Fig.~\ref{Fig:fp-feed} we see that with our feedback control the controlled solution converges to the targeted one as time increases. The magnitudes of the coordinates of the corresponding feedback control are shown in Fig.~\ref{Fig:fp-magu}. The actuators were again taken with supports  as in~\eqref{omeM1D}--\eqref{cM}; see Fig.~\ref{Fig:ff-act}.
 \begin{figure}[ht]
\centering
\subfigure[targeted trajectory, $\widehat y(0,x)=\widehat y_{0}$.]
{\includegraphics[width=0.45\textwidth,height=0.31\textwidth]{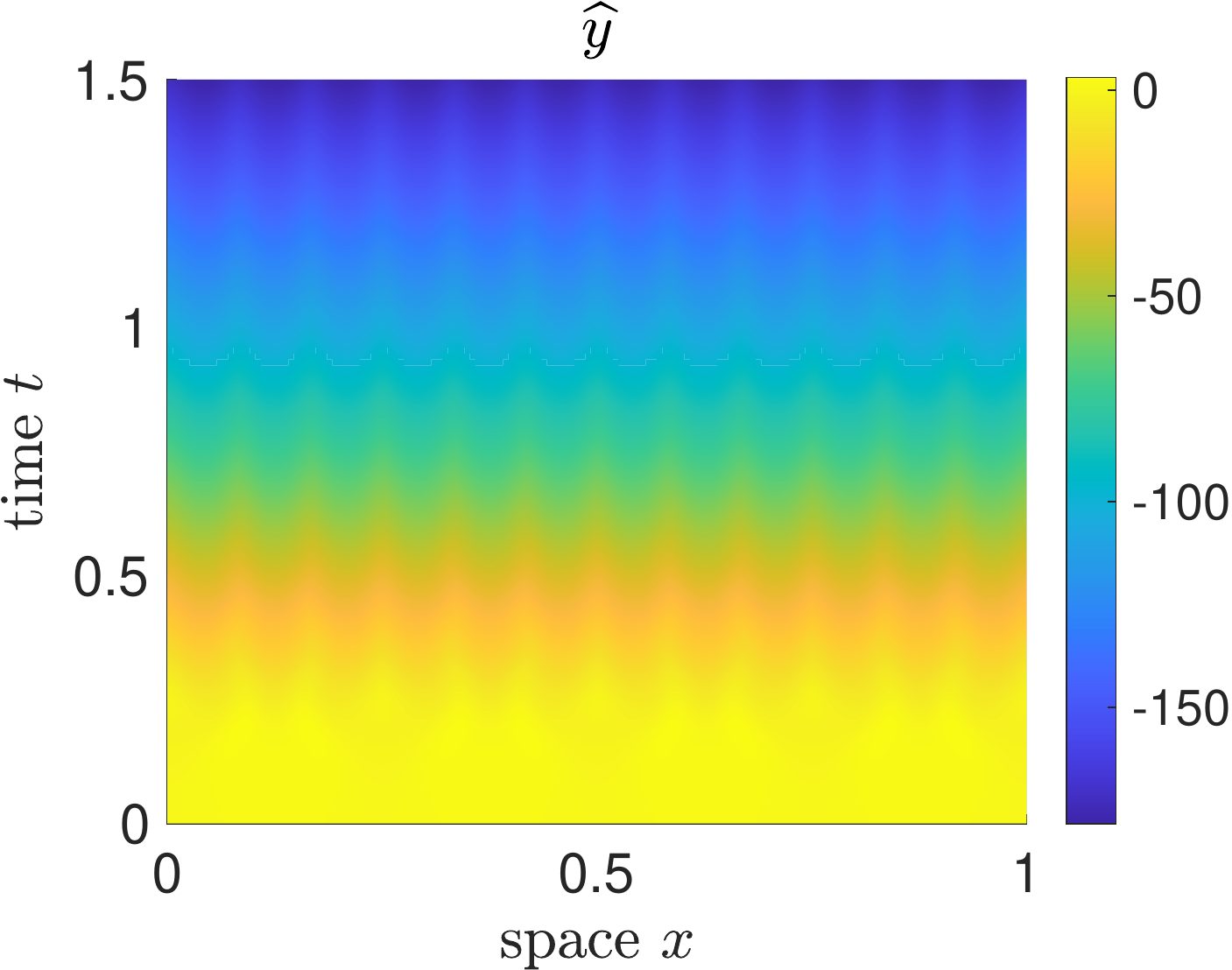}}
\qquad
\subfigure[uncontrolled trajectory, $\widetilde y_{\rm unc}(0,x)=\widetilde y_{0}$.]
{\includegraphics[width=0.45\textwidth,height=0.31\textwidth]{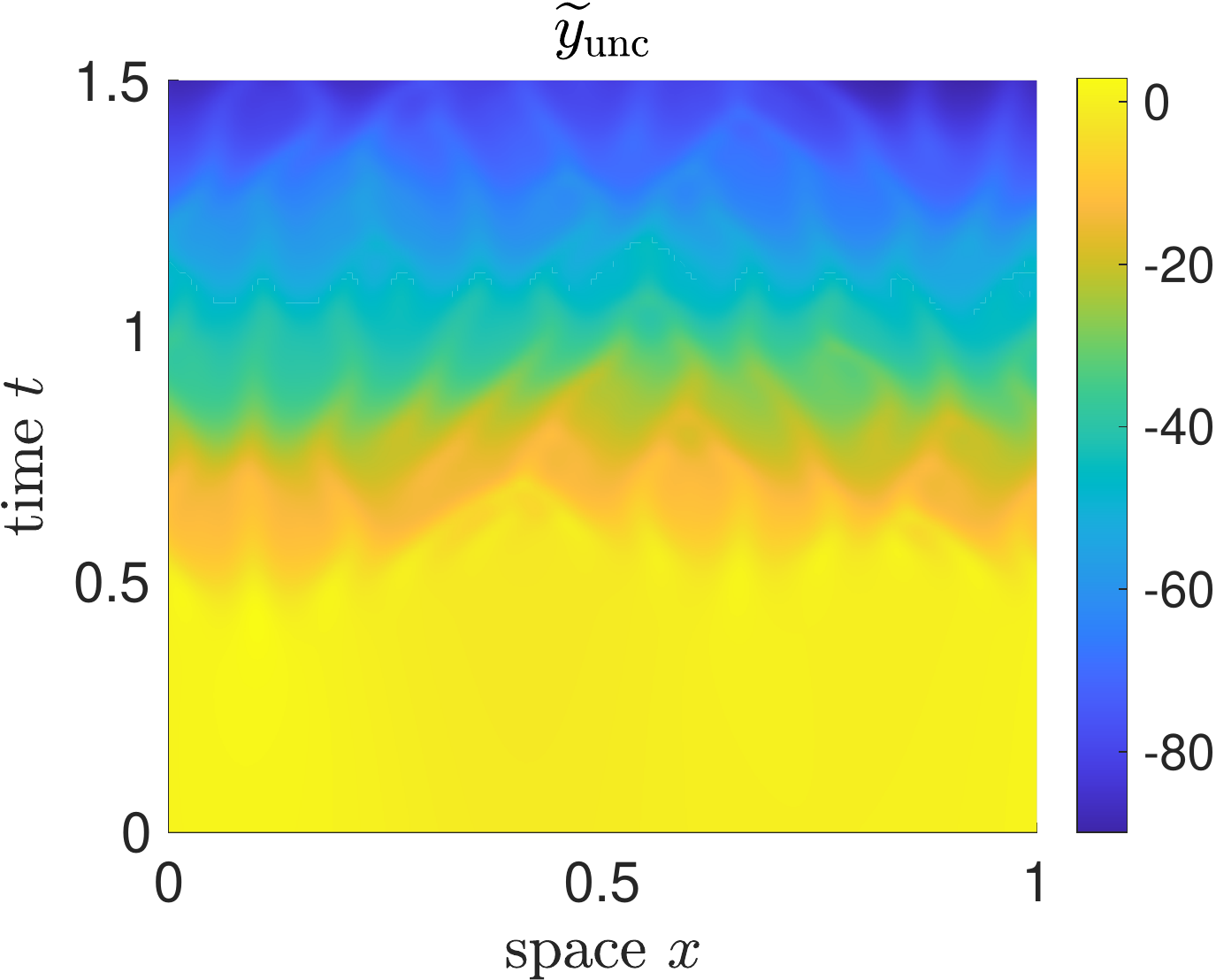}}
\caption{Flame model: free dynamics.}
\label{Fig:fp-free}
\end{figure}

 \begin{figure}[ht]
\centering
\subfigure[controlled trajectory, $\widetilde y(0,x)=\widetilde y_{0}$.]
{\includegraphics[width=0.45\textwidth,height=0.31\textwidth]{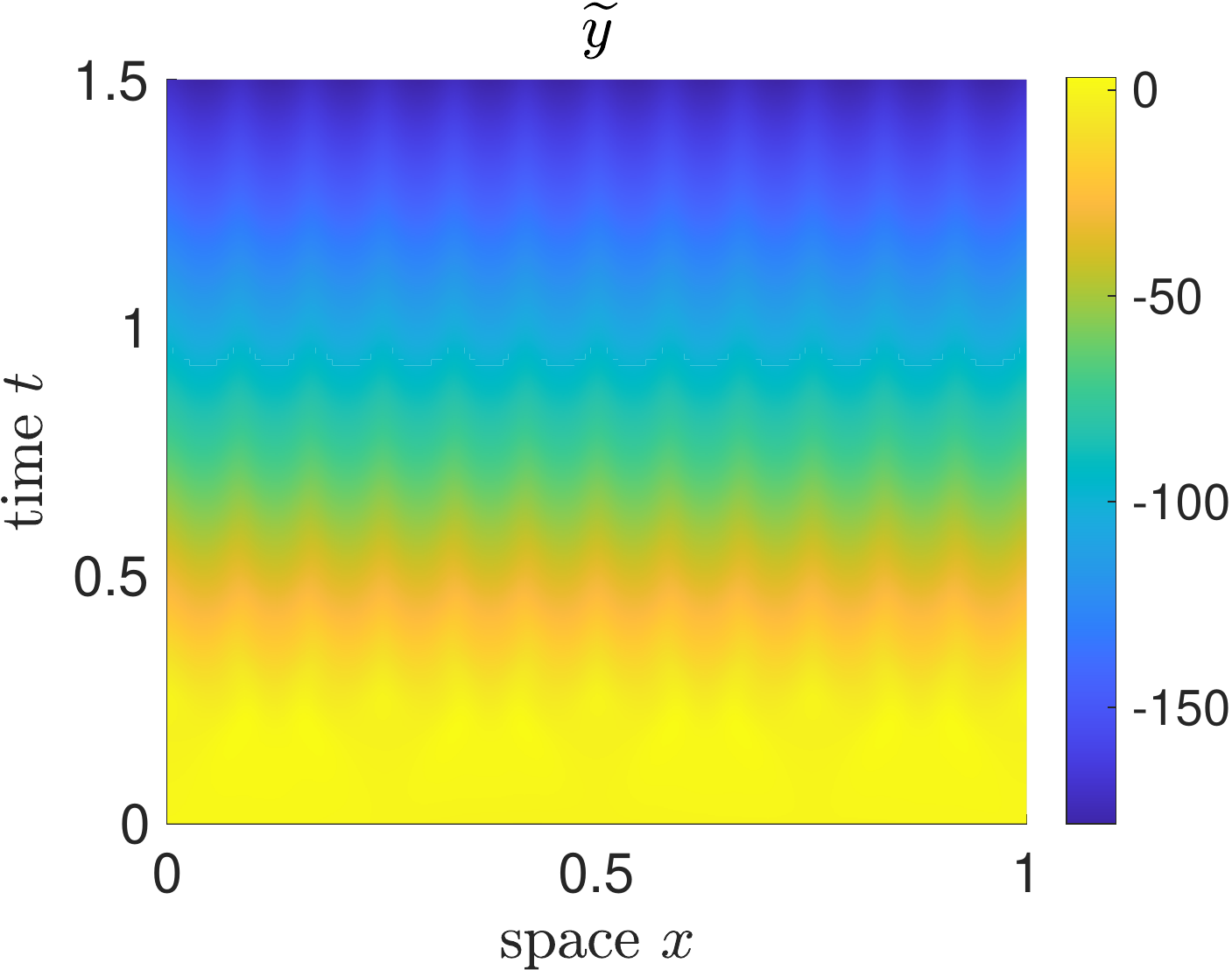}}
\qquad
\subfigure[evolution of distance to target.]
{\includegraphics[width=0.45\textwidth,height=0.31\textwidth]{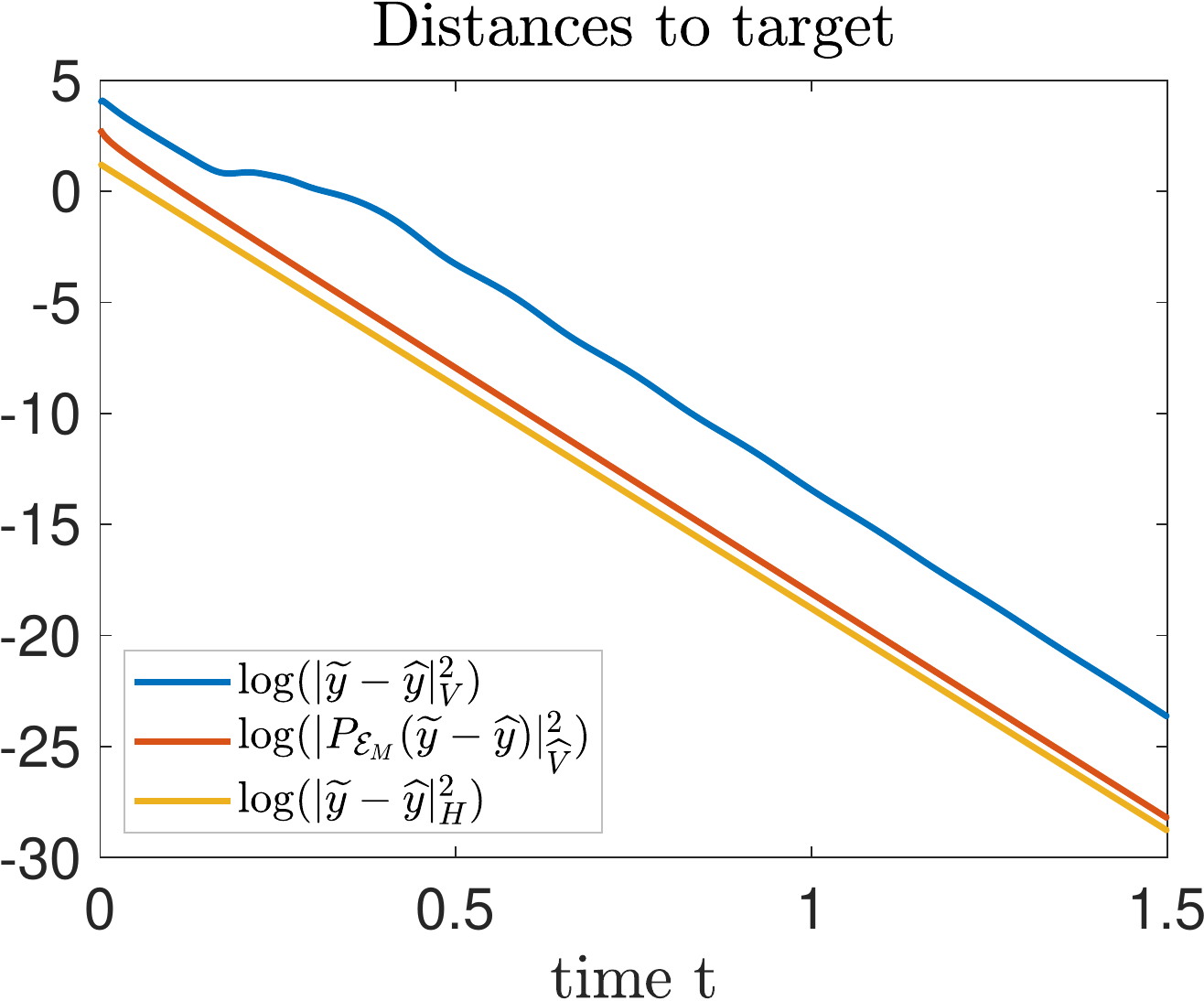}}
\caption{Flame model: controlled dynamics.}
\label{Fig:fp-feed}
\end{figure}

\begin{figure}[ht]
\centering
\subfigure
{\includegraphics[width=0.45\textwidth,height=0.31\textwidth]{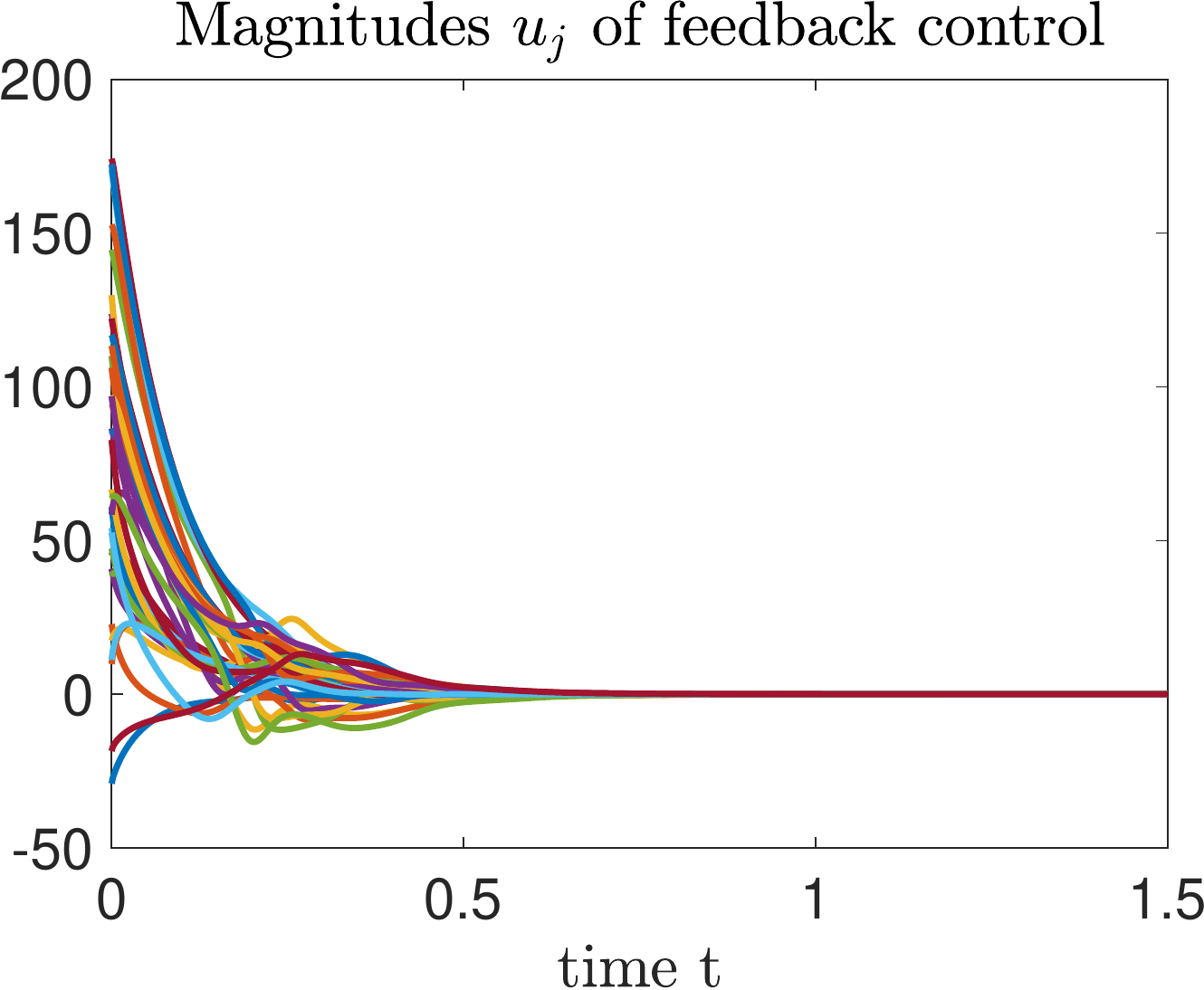}}
\caption{Flame model: signed magnitudes of control coordinates~$u_j$.}
\label{Fig:fp-magu}
\end{figure}

\subsection{Remark on the average of the solutions.}
It is clear that in the fluid flow model the average of the solution is preserved as time increases. Instead, for the flame propagation model the average is strictly decreasing when the gradient of the state is nonzero. This fact is observed, for example, in~\cite[Sect.~2, Eq.~(2.2)]{KalogKeavPapa14}. 
Therefore, as a curiosity, we substract the averages~$\widehat y_1=\int_0^1 \widehat y(x)\,\rmd x$ and~$\widetilde y_{\rm unc1}=\int_0^1 \widetilde y_{\rm unc}(x)\,\rmd x$ of the free dynamics solutions in Fig.~\ref{Fig:fp-free} and show the results in Fig.~\ref{Fig:fp-aver}. Comparing these two figures~\ref{Fig:fp-free} and~\ref{Fig:fp-aver}, we can see that the magnitudes of the free dynamics solutions are dominated by their average.  Likely, the average free components in Fig.~\ref{Fig:fp-aver} remain bounded and present an ``almost'' steady behavior (for~$\widehat y-\widehat y_1$) and a chaotic-like behavior (for~$\widetilde y_{\rm unc}-\widetilde y_{\rm unc1}$).
\begin{figure}[ht]
\centering
\subfigure
{\includegraphics[width=0.45\textwidth,height=0.31\textwidth]{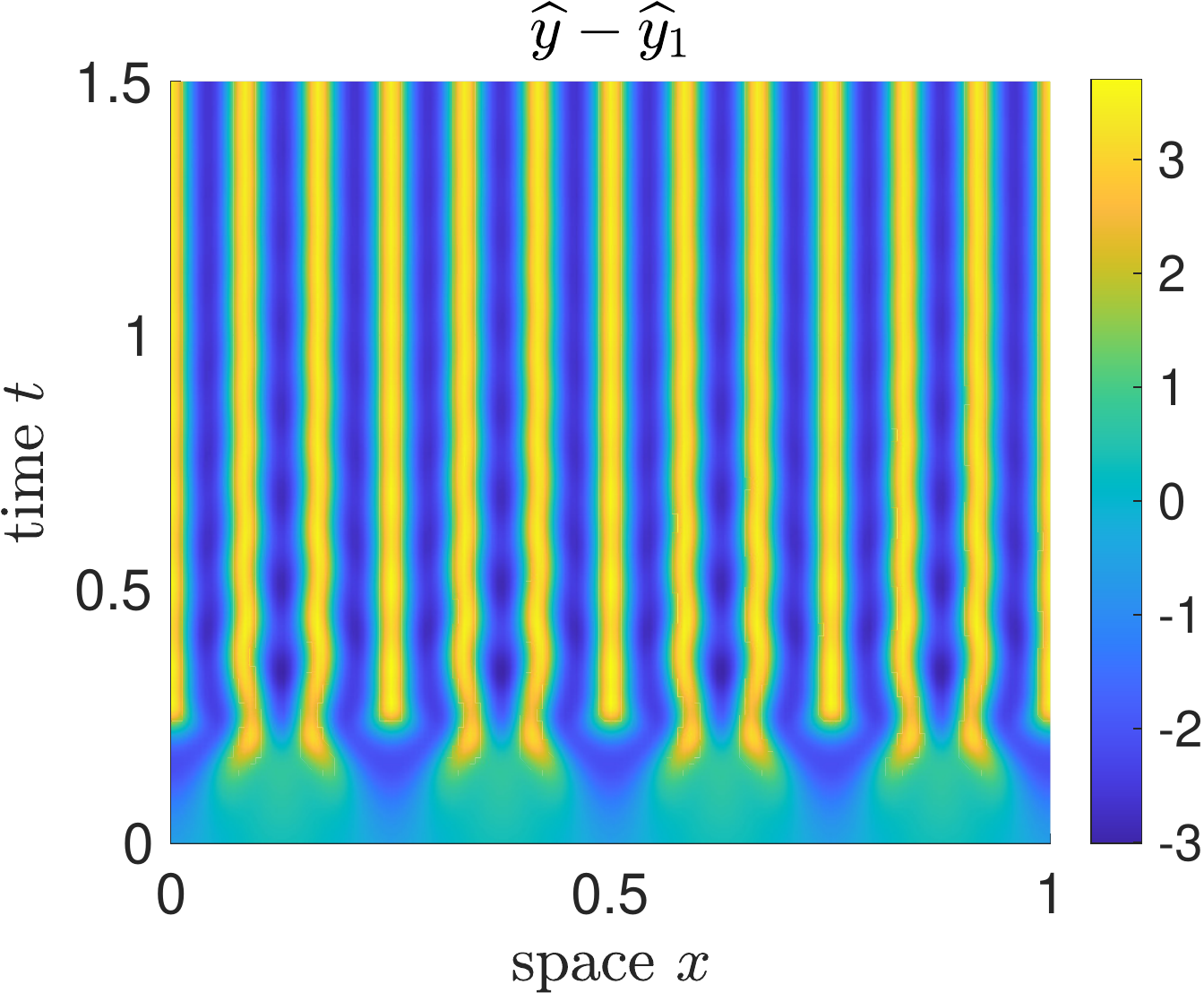}}
\qquad
\subfigure
{\includegraphics[width=0.45\textwidth,height=0.31\textwidth]{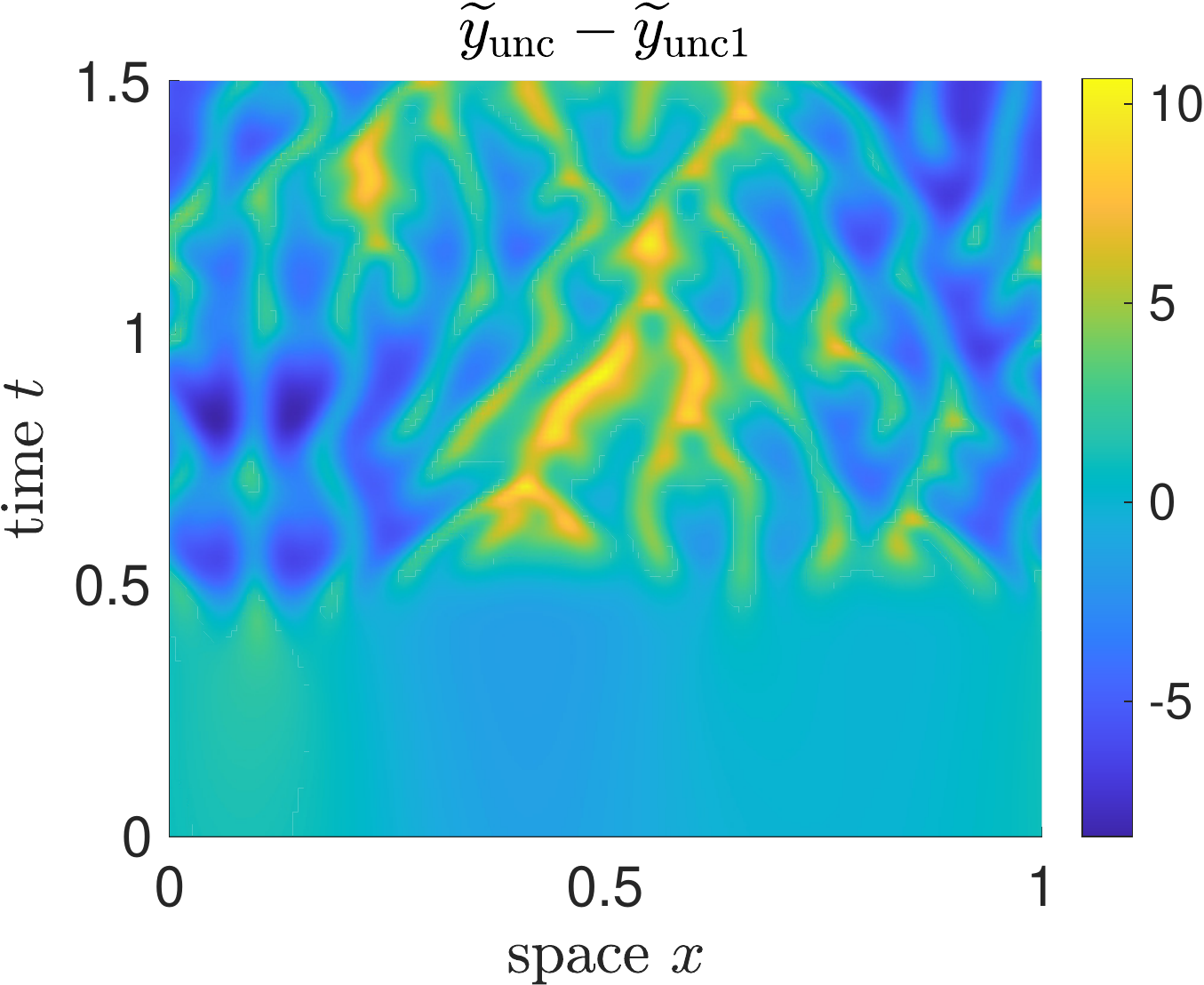}}
\caption{Flame  model: zero average components for free dynamics.}
\label{Fig:fp-aver}
\end{figure}

\subsection{Remarks on the numerical solutions.}
Here we present the results validating the numerical solution of the proposed numerical scheme by comparing the numerical solution with an exact analytical solution, in a toy example, chosen to be a linear combination of~$3$ eigenfunctions as follows,
\begin{subequations}\label{analytical}
\begin{align}
y_{\rm ex}(t,x)=3+2\tfrac{t}{t+1} +10\cos(4t)^2\cos(6\pi x)+(3t+2)^2\rme^{-2t}\sin(2\pi x).
\end{align}
To make such~$y_{\rm ex}$ a solution of our free dynamics system we have taken the initial state
\begin{align}
y_{\rm ex0}=y_{\rm ex}(0,x)=3 +10\cos(6\pi x)+16\sin(2\pi x)
\end{align}
and the external forcings
\begin{align}
&f(t,x)=\tfrac{\p}{\p t}{y_{\rm ex}}+ \nu_2\Delta^2y_{\rm ex}+ \nu_1\Delta y_{\rm ex}+\nu_0y_{\rm ex}\tfrac{\p}{\p x}y_{\rm ex},\quad&&\mbox{for fluid flow model;}
\\
&f(t,x)=\tfrac{\p}{\p t}{y_{\rm ex}}+ \nu_2\Delta^2y_{\rm ex}+ \nu_1\Delta y_{\rm ex}+\tfrac{1}{2}\nu_0\norm{\tfrac{\p}{\p x}y_{\rm ex}}{\bbR^d}^2,\quad&&\mbox{for flame prop. model.}
\end{align}
\end{subequations}

Next, we fix reference discretization parameters as
\begin{align}
N_0=50,\quad t^{\rm step}_0=10^{-4},\quad x^{\rm step}_0=10^{-3},
\end{align}
where~$N_0$ is a reference for the dimension~$N$ for the Galerkin approximation we will solve, $t^{\rm step}_0$ is a reference time-step for the time-step~$t^{\rm step}$ of the temporal discretization and  $x^{\rm step}_0$ is a reference space-step for the spatial time-step $x^{\rm step}$ for the 
discretization of the spatial domain~$\Omega=[0,1)$, which we use to compute the piecewise linear (hat functions based) finite-element mass matrix used in the computation of the nonlinearity.

Then we consider refinements as
\[
N_{\rho+1}\coloneqq 2N_{\rho},\quad t^{\rm step}_{\rho+1}\coloneqq\tfrac12 t^{\rm step}_{\rho},\quad x^{\rm step}_{\rho+1}\coloneqq\tfrac12 x^{\rm step}_{\rho},
\]
and run our solver, for the free dynamics, for each of the~$4$ refinement levels
\[
(N, t^{\rm step}, x^{\rm step})=(N_{\rho}, t^{\rm step}_{\rho}, x^{\rm step}_{\rho}), \quad\mbox{with}\quad\rho\in\{0,1,2,3\}.
\]
As in~\eqref{param-fluid} and~\eqref{nu0fire} we take
\begin{align}
&\nu_2=10^{-6},\quad \nu_1=10^{-2},\quad \nu_0=1,\quad&&\mbox{for fluid flow model;}\notag\\
&\nu_2=10^{-6},\quad \nu_1=10^{-2},\quad \nu_0=10^{-2},\quad&&\mbox{for flame propagation model.}\notag
\end{align}

In Fig.~\ref{Fig:yex} we see a plot of the exact analytical solution~$y_{\rm ex}$.
In Figs.~\ref{Fig:ff-exnum-conv} and~\ref{Fig:fp-exnum-conv} we see the error of our numerical solution for the fluid flow and flame propagation models . In particular, we observe that such error decreases as $\rho$ increases, that is, as we refine the discretization. After each refinement, from the data in the legend of the figures,  the error is likely divided by~$4$, approximately. This suggests a quadratic convergence for the solver, which agrees with our second order  discretizations in time (Crank--Nicolson/Adams--Bashford) and in space (piecewise-linear finite-elements for constructing the auxiliary mass matrix).
\begin{figure}[ht]
\centering
\subfigure
{\includegraphics[width=0.45\textwidth,height=0.31\textwidth]{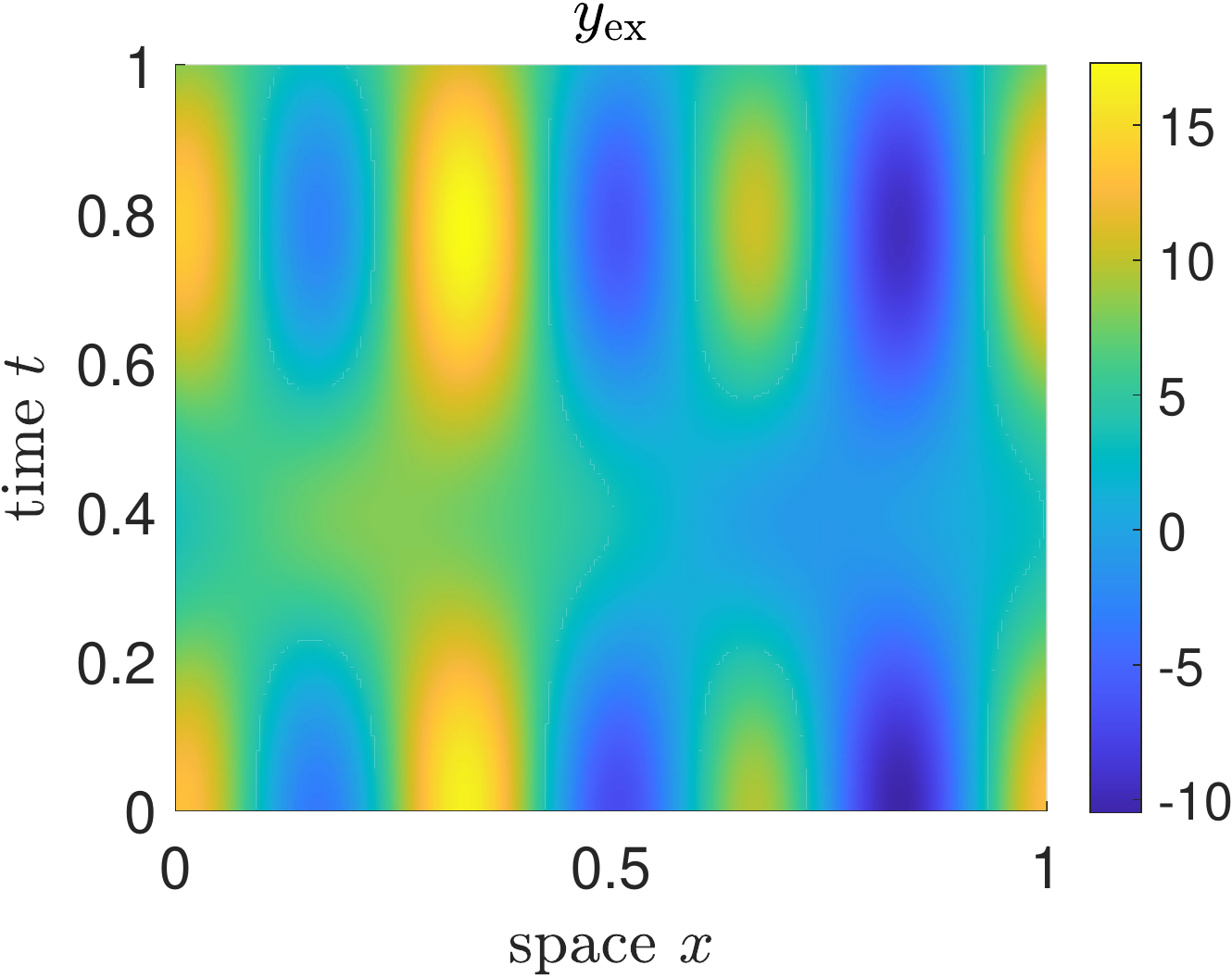}}
\caption{Exact solution~$y_{\rm ex}$.}
\label{Fig:yex}
\end{figure}
\begin{figure}[ht]
\centering
\subfigure[error for level discretization~$\rho=0$]
{\includegraphics[width=0.45\textwidth,height=0.31\textwidth]{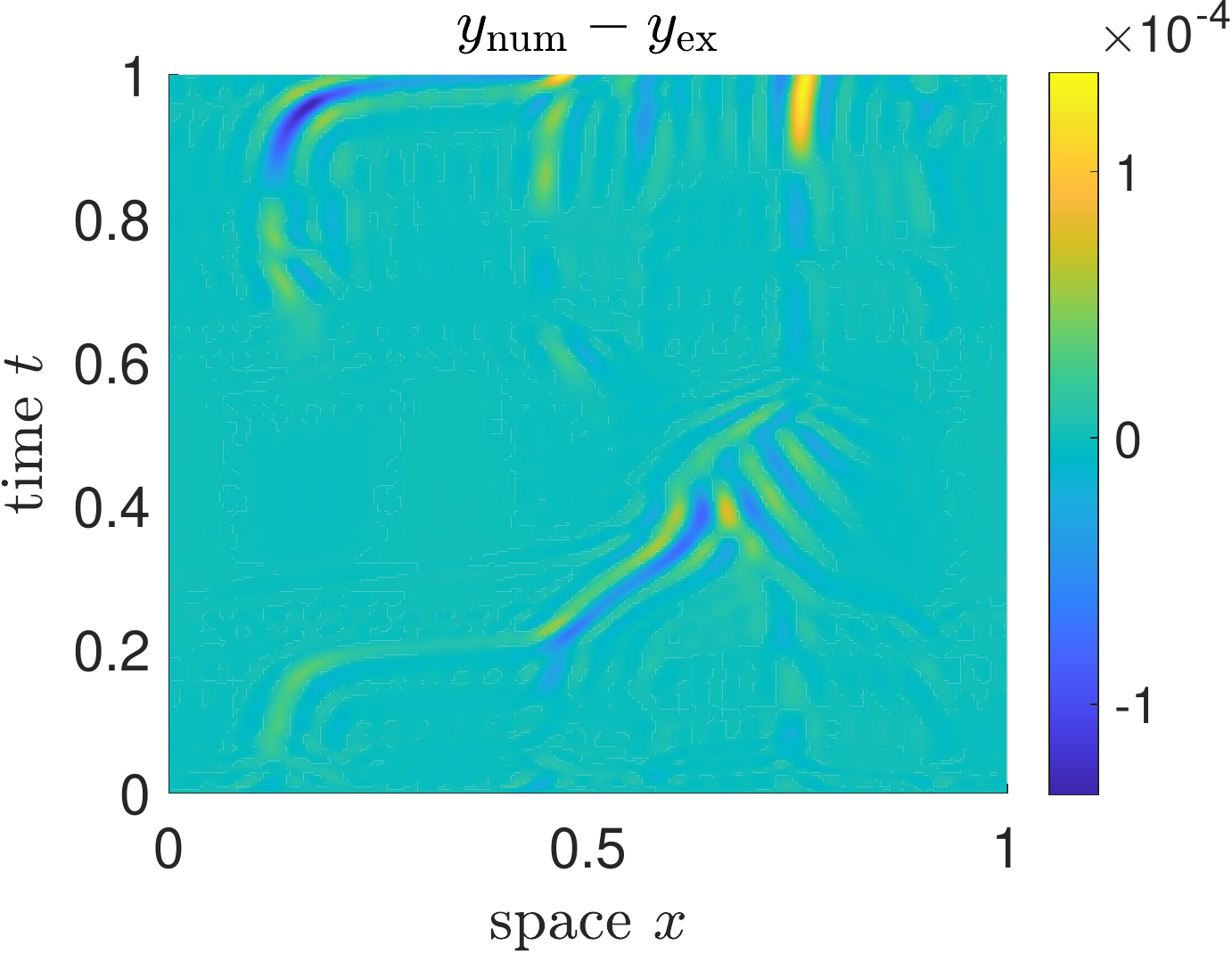}}
\qquad
\subfigure[evolution of error norm; convergence.]
{\includegraphics[width=0.45\textwidth,height=0.31\textwidth]{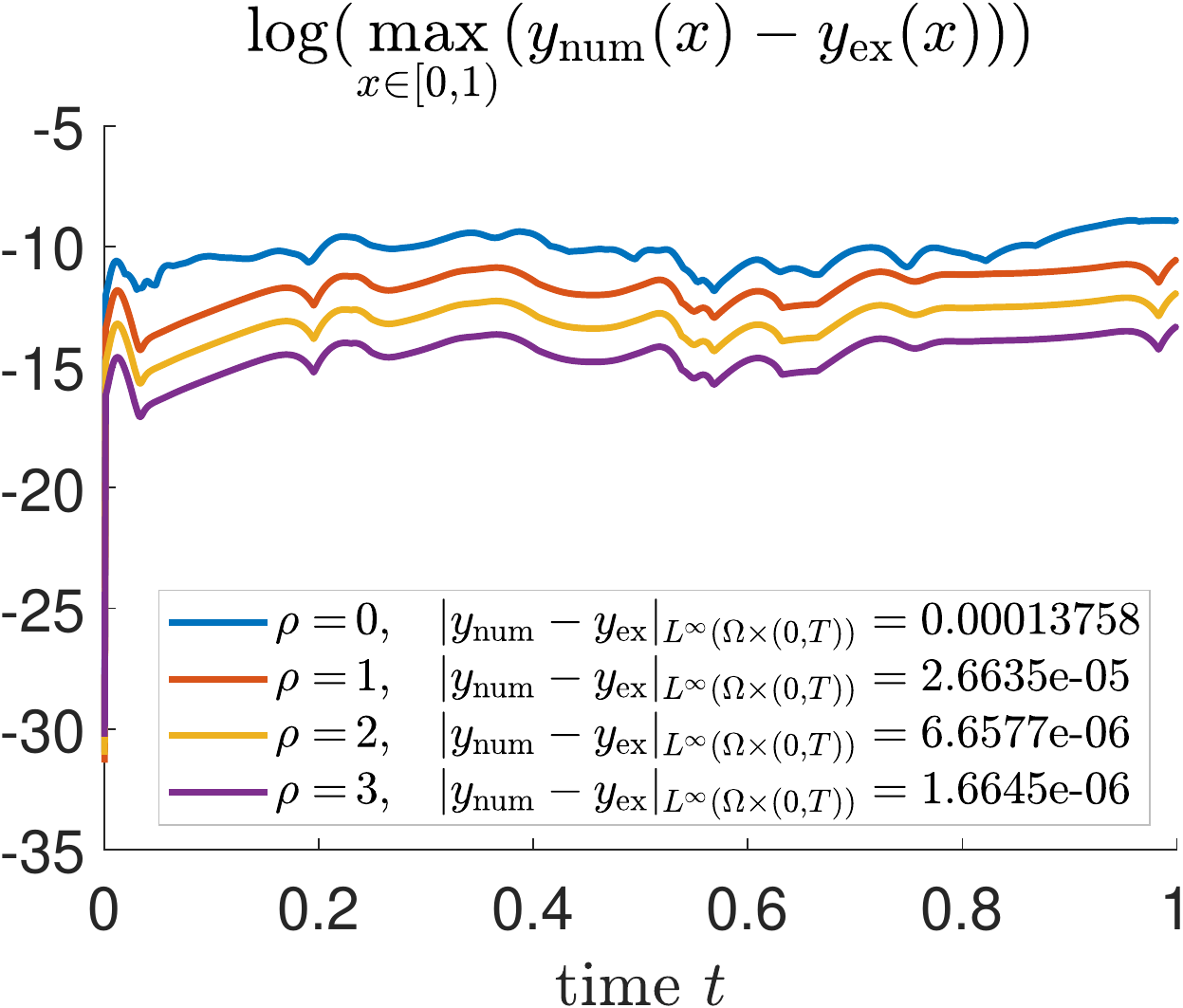}}
\caption{Fluid model: numerical error.}
\label{Fig:ff-exnum-conv}
\end{figure}

\begin{figure}[ht]
\centering
\subfigure[error for level discretization~$\rho=0$]
{\includegraphics[width=0.45\textwidth,height=0.31\textwidth]{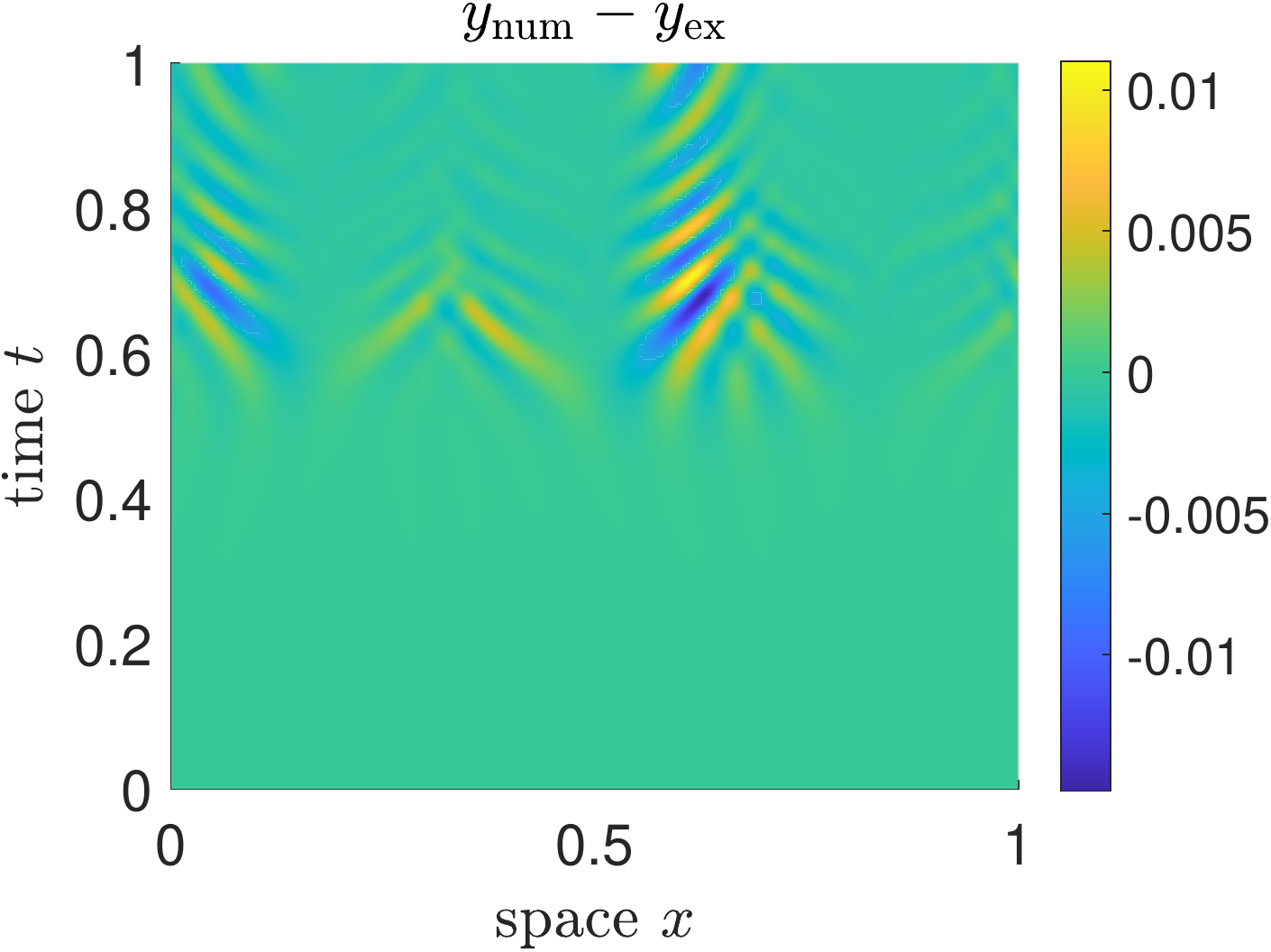}}
\qquad
\subfigure[evolution of error norm; convergence.]
{\includegraphics[width=0.45\textwidth,height=0.31\textwidth]{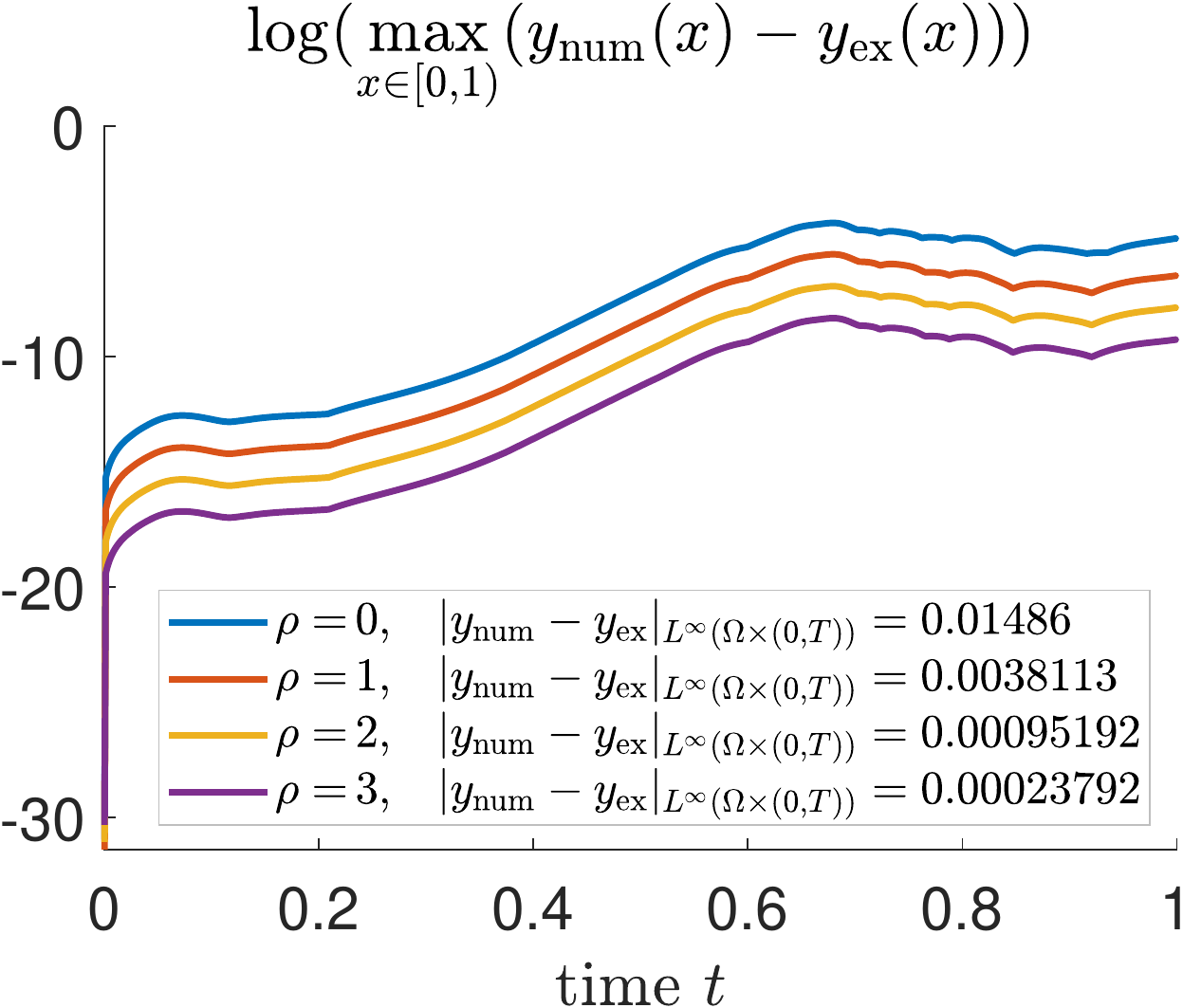}}
\caption{Flame model: numerical error.}
\label{Fig:fp-exnum-conv}
\end{figure}

\section{Conclusion} We have shown that an explicit oblique projection based feedback is able to stabilize the Kuramoto--Sivashinsky models for flame propagation and fluid flow. The stabilizing performance of the feedback operator has been confirmed through the results of numerical simulations. The numerical simulations are focused on the case of periodic boundary conditions. To show the theoretical result under such boundary conditions we showed that the oblique projection used to construct the feedback control is well defined and its norm remains bounded as we increase the number of actuators (for particularly placed and explicitly given actuators), thus extending existing results in the literature for Dirichlet and Neumann boundary conditions.

We addressed the construction of stabilizing feedback controls is an important task in control applications. There are other tasks equally important for control related applications, as state estimation, that could/should be addressed in future works.
Thus, as future work we can think of designing an observer as in~\cite{Rod21-jnls} to construct an estimate~$\breve y$ of the controlled state~$\widetilde y$ of the system (needed to compute/construct the the feedback control coordinates~$u_j$, $\sum u_j(t)\Phi_j(x)=K(t,\widetilde y(t)-\widehat y(t))$, and often not entirely available/``measurable'' in real world applications, e.g., we may not know the initial state~$\widetilde y_0=\widetilde y(0,x)$ exactly), and then use such estimate to construct an approximation~$\breve u$ of the feedback control, $\sum\breve u_j(t)\Phi_j(x)=K(t,\breve y(t)-\widehat y(t))$. For linear parabolic equations, we could make use of the work in~\cite{Rod21-aut},  but for nonlinear equations (as Kuramoto--Sivashinsky) the situation is not so clear due to the lack of so called ``separation principle'', where the dynamics of the estimate error~$\breve y-\widetilde y$  does not decouple/separate from the dynamics of the ``distance to target'' ~$\widetilde y-\widehat y$. Here, we are assuming that the target~$\widehat y$ is known, but we can  also assume (and probably should, in real world problems) an uncertainty on the initial state~$\widehat y_0=\widehat y(0,x)$, in that case we will need an estimate~$\overline y$ for~$\widehat y$ as well, and then use a control like~$\sum\breve u_j(t)\Phi_j(x)=K(t,\breve y(t)-\overline y(t))$.

The robustness of the ``feedback--observer'' coupled closed-loop system against noise and disturbances (model uncertainties) is an important property for applications. The investigation of such robustness is thus another interesting subject for further research, in particular, for nonlinear systems.

\bigskip\noindent
{\bf Aknowlegments.}
D. Seifu was supported by Upper Austria Government and Austrian Science
Fund (FWF): P 33432-NBL,  S. Rodrigues acknowledges partial support from the same grant.

 \bibliographystyle{plainurl}
 \bibliography{KS_Stabiliz0}

\end{document}

%% file: Mathcommands.tex
\newcommand{\linspan}{\mathop{\rm span}\nolimits}

\newcommand{\rest}{\left.\kern-2\nulldelimiterspace\right|_}
\newcommand{\norm}[2]{\left|#1\right|_{#2}}
\newcommand{\dnorm}[2]{\left\|#1\right\|_{#2}}

\newcommand{\Id}{{\mathbf1}}
\newcommand{\indf}{1}

\newcommand{\ex}{\mathrm{e}}
\newcommand{\p}{\partial}
\newcommand{\ed}{\mathrm d}

\newcommand*{\Bigcdot}{\raisebox{-.25ex}{\scalebox{1.25}{$\cdot$}}}

\newcommand{\clA}{{\mathcal A}}

\newcommand{\clC}{{\mathcal C}}

\newcommand{\clE}{{\mathcal E}}

\newcommand{\clG}{{\mathcal G}}

\newcommand{\clL}{{\mathcal L}}
\newcommand{\clM}{{\mathcal M}}
\newcommand{\clN}{{\mathcal N}}

\newcommand{\clU}{{\mathcal U}}
\newcommand{\clV}{{\mathcal V}}

\newcommand{\bbN}{{\mathbb N}}

\newcommand{\bbR}{{\mathbb R}}

\newcommand{\bbT}{{\mathbb T}}

\newcommand{\bfA}{{\mathbf A}}

\newcommand{\bfF}{{\mathbf F}}
\newcommand{\bfG}{{\mathbf G}}
\newcommand{\bfH}{{\mathbf H}}

\newcommand{\bfK}{{\mathbf K}}

\newcommand{\bfN}{{\mathbf N}}

\newcommand{\bfV}{{\mathbf V}}

\newcommand{\rmD}{{\mathrm D}}

\newcommand{\bfe}{{\mathbf e}}
\newcommand{\bff}{{\mathbf f}}
\newcommand{\bfg}{{\mathbf g}}
\newcommand{\bfh}{{\mathbf h}}
\newcommand{\bfi}{{\mathbf i}}
\newcommand{\bfj}{{\mathbf j}}

\newcommand{\bfn}{{\mathbf n}}

\newcommand{\bfw}{{\mathbf w}}

\newcommand{\bfy}{{\mathbf y}}
\newcommand{\bfz}{{\mathbf z}}

\newcommand{\rmd}{{\mathrm d}}
\newcommand{\rme}{{\mathrm e}}

\newcommand{\fkr}{{\mathfrak r}}

\definecolor{DarkBlue}{rgb}{0,0.08,0.45}
\definecolor{DarkRed}{rgb}{.65,0,0}
\definecolor{applegreen}{rgb}{0.55, 0.71, 0.0}

\newcounter{mymac@matlab}
\newcommand{\matlab}{MATLAB%
   \ifnum\value{mymac@matlab}<1%
   \textregistered%
   \setcounter{mymac@matlab}{1}%
   \fi%
  }
